\documentclass
[
    a4paper,
    DIV=11
]
{scrartcl}
\usepackage{fullpage,authblk,amsmath,amssymb,amsthm,amsfonts}
\usepackage{bm, bbm, dsfont}
\usepackage{enumerate}
\usepackage{tikz}
\usepackage{caption}
\usepackage{subcaption}
\usepackage{url,mathtools}
\usepackage{tensor}
\usepackage{footmisc}
\usepackage{graphicx}
\usepackage{mathrsfs}

\usepackage[pdffitwindow=true,
		   pdfstartview={FitH},
            plainpages=false,
            pdfpagelabels=true,
            pdfpagemode=UseOutlines,
            pdfpagelayout=SinglePage,
            bookmarks=false,
            colorlinks=true,
            hyperfootnotes=false,
            linkcolor=blue,
            urlcolor=blue!30!black,
            citecolor=green!50!black]{hyperref}

\usepackage{cleveref}

\usepackage{comment} 

\newtheorem{theorem}{Theorem}[section]

\newtheorem{lemma}[theorem]{Lemma}
\newtheorem{proposition}[theorem]{Proposition}

\theoremstyle{definition}

\newcommand{\mfrak}[1]{\mathfrak{#1}}
\newcommand{\mscr}[1]{\mathscr{#1}}
\newcommand{\mcal}[1]{\mathcal{#1}}
\newcommand{\bb}[1]{\mathbb{#1}}

\newcommand{\Norm}[1]{\left\Vert #1 \right\Vert}
\newcommand{\rd}[0]{\mathrm{d} }

\newcommand{\ang}[1]{\langle #1 \rangle}
\newcommand{\Abs}[1]{\left\vert#1\right\vert}
\newcommand{\Ang}[1]{\left\langle #1 \right\rangle}

\newcommand{\closure}[1]{\overline{#1}} 

\newcommand{\tderiv}[0]{\partial_t} 

\newcommand{\data}[0]{y} 
\newcommand{\dataRV}[0]{Y} 
\newcommand{\state}[0]{u} 
\newcommand{\statespace}[0]{\mcal{U}} 
\newcommand{\model}[0]{\mcal{M}} 
\newcommand{\param}[0]{\theta} 
\newcommand{\paramspace}[0]{\Theta} 
\newcommand{\modelerror}[0]{\delta} 
\newcommand{\modelerrorspace}[0]{\mscr{D}} 
\newcommand{\obsoperator}[0]{\mcal{O}} 
\newcommand{\noise}[0]{\varepsilon} 

\newcommand{\normaldist}[2]{\mcal{N}(#1,#2)} 

\newcommand{\appr}[0]{\textup{A}} 
\newcommand{\best}[0]{\dagger} 
\newcommand{\enh}[0]{\textup{E}} 
\newcommand{\enhancednoisemodel}[0]{\mfrak{u}} 
\newcommand{\enhancednoisecovariance}[0]{\Sigma_\noise+\obsoperator \Sigma_{\enhancednoisemodel}\obsoperator^{\ast}} 
\newcommand{\joint}[0]{\textup{J}} 
\newcommand{\marg}[0]{\textup{M}} 

\newcommand{\op}[0]{\textup{op}} 
\newcommand{\spn}[0]{\textup{span}} 
\newcommand{\kernel}[0]{\textup{ker}} 
\newcommand{\range}[0]{\textup{ran}} 

\newcommand{\law}[1]{\mu_{#1}} 
\newcommand{\KL}[0]{\textup{KL}} 

\numberwithin{equation}{section}

\begin{document}
\title{Choosing observation operators to mitigate model error in Bayesian inverse problems}

\author[1]{Nada Cvetkovi\'{c}}
\author[2]{Han Cheng Lie}
\author[1]{Harshit Bansal}
\author[1]{Karen Veroy}
\affil[1]{~Centre for Analysis, Scientific Computing and Applications, Eindhoven University of Technology, Groene Loper 5, 5612 AZ Eindhoven, the Netherlands}
\affil[2]{~Institut f\"ur Mathematik, Universit\"at Potsdam, Campus Golm, Haus 9, Karl-Liebknecht-Stra{\ss}e 24--25, Potsdam OT Golm 14476, Germany}
\renewcommand\Affilfont{\small}

\date{}
\maketitle

\begin{abstract}
In statistical inference, a discrepancy between the parameter-to-observable map that generates the data and the parameter-to-observable map that is used for inference can lead to misspecified likelihoods and thus to incorrect estimates.
In many inverse problems, the parameter-to-observable map is the composition of a linear state-to-observable map called an `observation operator' and a possibly nonlinear parameter-to-state map called the `model'.
We consider such Bayesian inverse problems where the discrepancy in the parameter-to-observable map is due to the use of an approximate model that differs from the best model, i.e. to nonzero `model error'.
Multiple approaches have been proposed to address such discrepancies, each leading to a specific posterior.
We show how to use local Lipschitz stability estimates of posteriors with respect to likelihood perturbations to bound the Kullback--Leibler divergence of the posterior of each approach with respect to the posterior associated to the best model.
Our bounds lead to criteria for choosing observation operators that mitigate the effect of model error for Bayesian inverse problems of this type.
We illustrate one such criterion on an advection-diffusion-reaction PDE inverse problem from the literature, and use this example to discuss the importance and challenges of model error-aware inference.
\end{abstract}



\section{Introduction}
\label{sec_introduction}

In many applications, one considers an inverse problem where the data is a noisy observation of the true state of some phenomenon of interest, where the true state is the output of a parameter-to-state mapping or `model' corresponding to an unknown parameter.
That is, for the unknown true parameter $\param^{\best}$ and the best model $\model^{\best}$, the true state is $\state^{\best}=\model^{\best}(\param^{\best})$, and the data $y$ is a realisation of the random variable
\begin{equation*}
 \dataRV\coloneqq \obsoperator \circ \model^{\best}(\param^{\best})+\noise
\end{equation*}
for a state-to-observable map or `observation operator' $\obsoperator$ and additive noise $\noise$.
The inverse problem consists in inferring the data-generating parameter $\param^{\best}$ from the data $y$.

We consider Bayesian inverse problems with data taking values in $\bb{R}^{n}$, centred Gaussian observation noise, and linear observation operators.
In this setting, the noise $\noise$ has the normal distribution $\normaldist{m_{\noise}}{\Sigma_{\noise}}$ for $m_{\noise}=0\in\bb{R}^{n}$ and positive definite covariance $\Sigma_{\noise}\in\bb{R}^{n\times n}$, and the observation operator $\obsoperator$ is a linear mapping from the `state space' $\statespace$ of admissible states to $\bb{R}^{n}$.
In the Bayesian approach to inverse problems, one fixes a possibly infinite-dimensional parameter space $\paramspace$ consisting of admissible values of $\param^{\best}$, and describes the unknown true parameter $\param^{\best}$ using a random variable $\param$ with prior law $\mu_\param$ supported on $\paramspace$. 
Under certain assumptions on the prior $\mu_\param$, the observation operator $\obsoperator$, and the model $\model^{\best}$, the corresponding posterior measure $\mu^{\data,\best}_{\param}$ is well-defined and admits the following density with respect to the prior $\mu_\param$:
\begin{equation*}
\paramspace\ni \param' \mapsto \frac{\rd\mu^{\data,\best}_{\param}}{\rd\mu_\param}(\param')=\frac{\exp(-\tfrac{1}{2}\Norm{\data-\obsoperator \circ \model^{\best}(\param')}_{\Sigma_\noise^{-1}}^{2})}{\displaystyle{\int_{\paramspace}}\exp( -\tfrac{1}{2}\Norm{\data-\obsoperator \circ \model^{\best}(\hat{\param})}_{\Sigma_\noise^{-1}}^{2})\rd\mu_\param(\hat{\param})}.
\end{equation*}
See e.g. \cite[Theorem 4.1]{Stuart2010} for sufficient conditions for well-definedness in the case of a Gaussian prior $\mu_\param$.

In practice, the best model $\model^{\best}:\paramspace\to\statespace$ is not available, so an approximate model $\model:\paramspace\to\statespace$ is used instead.
Alternatively, $\model^{\best}$ may be available but impractical or costly to evaluate: in the context of multifidelity or reduced order models, $\model^{\best}$ may be the element of a collection of models that yields the most accurate predictions of state, and $\model$ may be a reduced-order model, an emulator, or a surrogate, i.e. a model that yields less accurate predictions but can be evaluated more cheaply and quickly than $\model^{\best}$.

We shall refer to the difference $\modelerror^{\best}\coloneqq \model^{\best}-\model$ as the `model error of the approximate model', or simply, the `model error'.
Model error is also known as `model inadequacy', `model discrepancy', or `structural uncertainty', see e.g. \cite{KennedyOHagan2001,Brynjarsdottir2014}.
We do not use the term `model misspecification', since this term is used in the statistics literature to refer to the setting where the parameter space $\paramspace$ does not contain the true parameter $\param^{\best}$; see e.g. \cite[Section 8.5]{GhosalvanderVaart2017}.

In the context of inverse problems, model error is important because it may lead to a wrong or `misspecified' likelihood, which in turn may lead to incorrect inference.
The negative effects may persist even after applying or approximating common limits from statistics.
For example, numerical experiments in \cite[Section 3]{Brynjarsdottir2014} show how ignoring the model error results in posterior distributions that do not converge to the true data-generating parameter as the number of observations grows larger.
An analytical example in \cite[Section 8.1]{AbdulleGaregnani2020} shows a similar phenomenon for the problem of inferring the initial condition of a heat equation over a finite time interval from a noisy observation of the terminal state, in the limit of small noise.

The potential for incorrect inference due to model error raises the question of how to mitigate the effect of model error in Bayesian inverse problems.
We approach this question from the perspective of selecting an appropriate observation operator $\obsoperator$.
To see why this perspective is reasonable, note that the linearity of $\obsoperator$ and the definition of the model error $\modelerror^\best$ imply that $\obsoperator\model^\best=\obsoperator\model +\obsoperator \modelerror^\best$. 
Using this fact in the equation for $\dataRV$ shows that only the observed model error $\obsoperator \modelerror^\best$ has the potential to affect inference.
In other words, ignoring model error $\modelerror^{\best}$ leads to a misspecified likelihood if and only if $\obsoperator   \modelerror^{\best}(\param')$ is nonzero with positive $\mu_\param$-probability.
This suggests that a possible approach to reduce the effect of model error on Bayesian inference is to choose an observation operator $\obsoperator$ for which $\obsoperator   \modelerror^{\best}(\param')$ is closer to zero with higher $\mu_\param$-probability.

The approach of choosing observation operators suggests a connection with experimental design.
In Bayesian experimental design, the main task is to select observations in order to maximise information about the parameter to be inferred.
To quantify the information gain, one may use the Kullback--Leibler divergence of the posterior with respect to the prior, for example.
In contrast, we control the Kullback--Leibler divergence between pairs of posteriors defined by the same prior but different likelihoods, by using the $L^1_{\mu_\param}$ difference between the pair of negative log-likelihoods or `misfits'.
The main task is then to select observations in order to minimise this $L^1_{\mu_\param}$ difference.
This approach can be viewed as experimental design for mitigating the effect of model error on inference.

\paragraph{Contributions} In this paper, we consider some approaches for accounting for model error in Bayesian inference: the approximate approach that ignores model error; the `enhanced error approach' of \cite{Kaipio2006}; the `joint inference approach' that aims to infer both $\param^{\best}$ and $\modelerror^{\best}$; and the `marginalisation approach', which integrates out the model error component of the posterior from the joint inference approach.
For the first three approaches, we compute upper bounds for the $L^1_{\mu_\param}$ difference between the misfit of each approach and the misfit associated to the best model $\model^\best$.
We use these bounds to bound the Kullback--Leibler divergence between the posterior that results from each approach and the best posterior $\mu^{\data,\best}_{\param}$.
We also compute upper bounds for the Kullback--Leibler divergence with respect to the approximate posterior $\mu^{\data,\appr}_{\param}$ that results from using the given approximate model $\model$ instead of the best model $\model^\best$.
We express each upper bound in terms of the model error $\modelerror^{\best}$ and the objects that the approach uses to account for the model error.

To prove these bounds, we exploit the assumption of additive Gaussian noise to prove a lemma concerning the difference of two quadratic forms that are weighted by different matrices; see \Cref{lem_L1norm_diff_quadratic_forms}. 
We prove the upper bounds on the Kullback--Leibler divergences by combining the upper bounds on the $L^1_{\mu_\param}$ differences between the misfits with a local Lipschitz stability estimate of posteriors from \cite{Sprungk2020}.
An important advantage of this estimate is that it holds for the Kullback--Leibler topology under the mild assumption of $L^1_{\mu_\param}$-integrable misfits; see \Cref{thm_theorem11_and_proposition6_Sprungk2020}.

We use the upper bounds on the Kullback--Leibler divergence with respect to the best posterior $\mu^{\data,\best}_{\param}$ to derive sufficient conditions on the observation operator $\obsoperator$, the model error $\modelerror^{\best}$, and the objects that the approach uses to account for the model error, in order for the resulting posterior $\mu^{\data,\bullet}_{\param}$ to closely approximate $\mu^{\data,\best}_{\param}$ in the Kullback--Leibler topology.
These conditions give `positive' criteria for choosing observation operators to mitigate the effect of model error on inference.
We use the upper bounds with respect to the approximate posterior $\mu^{\data,\appr}_{\param}$ to derive necessary conditions for the resulting posterior $\mu^{\data,\bullet}_{\param}$ to differ from $\mu^{\data,\appr}_{\param}$.
These conditions give `negative' criteria which one can use to exclude observation operators from consideration when one aims to mitigate the effect of model error on Bayesian inference.

We complement our theoretical findings with numerical experiments for an inverse problem from the literature, namely the inverse problem of inferring the time amplitude of a source term in an advection-diffusion-reaction partial differential equation (PDE) when the initial condition is unknown.

\subsection{Overview of related work}
\label{ssec_related_work}

The importance of accounting for the model error is well-documented in the literature on Bayesian inverse problems; see e.g. \cite{Brynjarsdottir2014, Kaipio2006, KennedyOHagan2001} and the references therein. 
The `Bayesian approximation error' and `enhanced error' approaches due to \cite{Kaipio2007} and \cite{Kaipio2006} respectively have been applied in various contexts.
For example, the enhanced error approach has been applied with pre-marginalisation to electrical impedance tomography \cite{Nissinen2009}, diffuse optical tomography \cite{Kolehmainen2011}, and inversion for coefficient fields in the presence of uncertain conductivities \cite{Nicholson2018}.

Various methods have been developed to estimate or account for model error in Bayesian inference.
For example, the work \cite{Calvetti2018} presented an iterative algorithm to update an estimate of the model error in a model order reduction context, and proved geometric convergence of the algorithm.
The authors of \cite{Sargsyan2015,Sargsyan2019} take a different perspective: instead of viewing model errors as additive perturbations to an approximate model, they incorporate these model errors into parametrisations of some phenomenon of interest, and use polynomial chaos expansions.
Information theory has been used to quantify model error uncertainty or model bias in goal-oriented inference settings \cite{Hall2018} and by exploiting concentration inequalities \cite{Gourgoulias2020}.
Optimal transport was applied to tackle problems due to model errors in \cite{Scarinci2021}.

In the context of Bayesian optimal experimental design, model error is also referred to as `model uncertainty'.
The work \cite{Koval2020} considers A-optimal designs for inverse problems in the presence of so-called `irreducible' model uncertainties, i.e. uncertainties that cannot be reduced by collecting more data. 
In contrast, the work \cite{Alexanderian2021} considers reducible uncertainties, and describes an A-optimality criterion that involves marginalising out this reducible uncertainty. 
The work \cite{Alexanderian2022} combines the Laplace approximation and the Bayesian approximation error approach to find A-optimal designs for nonlinear Bayesian inverse problems.

As far as we are aware, the works most closely related to our paper are \cite{Cao2022,Duong2022}.
The work \cite{Cao2022} considers Bayesian inverse problems where the observation operator may be nonlinear and the model is approximated by a neural network.
In particular, \cite[Theorem 1]{Cao2022} bounds the Kullback--Leibler divergence between the original and approximated posterior in terms of an $L^p$ norm for $p\geq 2$ of the model error itself. 
In contrast, we consider linear observation operators, do not focus on any specific class of approximate models, and bound the Kullback--Leibler divergence in terms of $L^1$ norms of the observed model error.
In \cite{Duong2022}, the main stability estimate \cite[Theorem 3.4]{Duong2022} bounds the expected utility of a design in terms of a sequence of likelihoods.
The focus of \cite{Duong2022} is not model error, but on the convergence of the utility of approximate optimal designs corresponding to a convergent sequence of likelihoods.
The goal is `standard' optimal experimental design, i.e. on choosing observations to maximise the information content of each additional data point and thereby to reduce posterior uncertainty.
In contrast, we make pairwise comparisons of likelihoods and focus on choosing observations to mitigate the effect of model error on Bayesian inference.

\subsection{Outline}
\label{ssec_outline}

We describe the notation that we use in \Cref{ssec_notation}.
In \Cref{sec_bounds_on_KL_divergence}, we state and prove upper bounds on $L^1_{\mu_\param}$ errors between misfits and upper bounds on Kullback--Leibler divergences between the best posterior and posteriors that account for model error in different ways.
We use these bounds to discuss positive and negative criteria for choosing observation operators to mitigate the effect of model error on Bayesian inference.
In \Cref{sec_example}, we consider a specific inverse problem and use it to illustrate a positive criterion for choosing observation operators, in the setting where one uses the approximate posterior for inference.
In \Cref{sec_numerical_simulations}, we describe and discuss the numerical experiments that we used to investigate the behaviour of the approximate posterior.
We conclude in \Cref{sec_conclusion}.
In \Cref{sec_technical_lemmas}, we present proofs of lemmas and results that are not stated in the main text.

\subsection{Notation}
\label{ssec_notation}

The notation $a\leftarrow b$ indicates the replacement of $a$ using $b$, while $a\coloneqq b$ indicates equality by definition.
For $N\in\bb{N}$, $[N]\coloneqq\{1,\ldots,N\}$, and $[N]_0\coloneqq\{0,1,\ldots,N\}$.
Let $(\Omega,\mcal{F},\bb{P})$ be the underlying probability space for all random variables below. 
Given a measurable space $(E,\mcal{E})$ and an $E$-valued random variable $\xi:(\Omega,\mcal{F})\to (E,\mcal{E})$, we denote the law of $\xi$ by $\law{\xi}$. If $E$ is a topological space, we denote its Borel $\sigma$-algebra by $\mcal{B}(E)$, the set of all Borel probability measures on $E$ by $\mcal{P}(E)$, and the support of an arbitrary $\mu\in\mcal{P}(E)$ by $\textup{supp}(\mu)$.
For $S\in\mcal{B}(E)$, $\closure{S}$ denotes the closure of $S$.
For $\nu\in \mcal{P}(E)$ and an $E$-valued random variable $\xi$, the expression $\xi \sim \nu$ means that $\xi$ is distributed according to $\nu$, i.e. $\law{\xi}=\nu$ as measures.
Given $\mu,\nu\in\mcal{P}(E)$, we denote the absolute continuity of $\mu$ with respect to $\nu$ by $\mu\ll\nu$, the corresponding Radon--Nikodym derivative by $\tfrac{\rd \mu}{\rd \nu}$, and the Kullback--Leibler (KL) divergence of $\mu$ with respect to $\nu$ by $d_{\KL}(\mu\Vert \nu)$, which has the value $\int_{E} \log \frac{\rd \mu}{\rd \nu} \rd\mu$ if $\mu\ll\nu$ and $+\infty$ otherwise. 
For $\mu\in\mcal{P}(E)$ and $p\geq 1$, $\Norm{\cdot}_{L^p_\mu}$ denotes the $L^p$ norm with respect to $\mu$; we shall specify the domain and codomain when necessary.
We denote a Gaussian measure with mean $m$ and covariance operator $\Sigma$ by $\normaldist{m}{\Sigma}$.
Some useful facts about Gaussian measures on Banach spaces are given in \cite[Section 6.3]{Stuart2010}, for example.

Given $\mu\in\mcal{P}(E)$ and $\Phi:E\to\bb{R}$, we write $\mu_\Phi$ to denote the measure that has Radon--Nikodym derivative $\tfrac{\rd \mu_\Phi}{\rd \mu}=\tfrac{\exp(-\Phi)}{Z} $, where $Z\coloneqq \int_E \exp(-\Phi(x')) \rd \mu(x')$.
We shall use this notation to denote the posterior associated to a given prior and misfit below.

For $d\in\bb{N}$ and a symmetric positive semidefinite matrix $L\in\bb{R}^{d\times d}$, 
\begin{equation}
 \label{eq_matrix_weighted_inner_product_norm}
 \ang{a,b}_{L}\coloneqq a^\top L b=\ang{L^{1/2}a,L^{1/2}b},\quad \Abs{a}_{L}\coloneqq \ang{a,a}_{L}^{1/2}=\Abs{L^{1/2}a}.
\end{equation}
We use single bars `$\Abs{\cdot}$' to emphasise when we take norms of vectors in finite-dimensional spaces.
If $L$ is the identity matrix, then we omit the subscript.
Given normed vector spaces $(V_i,\Norm{\cdot}_{V_i})$ for $i=1,2$, $V_i^\ast$ denotes the continuous dual space of $V_i$.
For a bounded linear operator $L:V_1\to V_2$, we denote the kernel, adjoint, and operator norm of $L$ by $\kernel( L)$, $L^\ast$, and $\Norm{L}_{\op}$, and we write the image of $v\in V_1$ under $L$ as $Lv$.
For a matrix $A\in\bb{R}^{m\times n}$, $A^\ast$ and $A^+$ denote the adjoint and Moore--Penrose pseudoinverse of $A$ respectively.

\section{Bounds on KL divergence in terms of observed model error}
\label{sec_bounds_on_KL_divergence}

In this section, we state the theoretical results of our paper, namely bounds on the KL divergence between pairs of posteriors, expressed in terms of the model error and the objects that some approaches use to account for model error.
These bounds show that it is not the model error itself but rather the \emph{observed} model error that affects inference.
We consider four natural types of posteriors that account for model error in different ways.
After introducing the common objects and basic assumptions, we state bounds on certain KL divergences for the approximate posterior, enhanced noise posterior, joint posterior, and marginal posterior in Sections \ref{ssec_KL_bound_approximate_posterior}, \ref{ssec_KL_bound_enhanced_noise}, \ref{ssec_KL_error_joint_posterior}, and \ref{ssec_KL_error_marginal_posterior} respectively.
We emphasise that we use the KL divergence only as a measure of the discrepancy between different posteriors, and not because the KL divergence is often used in classical optimal experimental design. Indeed, our bounds could have also been expressed using other measures of discrepancy that are considered in \cite{Sprungk2020}, such as the Hellinger distance or the total variation distance. Moreover, in the example inverse problem in \Cref{sec_example}, we use one such bound to guide the choice of observations, but do not estimate the terms in the bounds. The efficient and accurate estimation of the terms in the bounds below is beyond the scope of this paper.

We assume that the `parameter space' $\paramspace$ of admissible values of the unknown true data-generating parameter $\param^{\best}$ is a Borel subset of a separable Banach space, and fix a prior $\mu_\param\in\mcal{P}(\paramspace)$ that is supported on $\paramspace$.
Let $(\statespace,\Norm{\cdot}_{\statespace})$ denote the `state space', which is a Banach space containing both the image of $\paramspace$ under the best model $\model^\best$ and the image of $\paramspace$ under the approximate model $\model$.
In many inverse problems, a state $\state\in\statespace$ is a function on some domain that describes some observable phenomenon of interest, and is often expressed as the solution of a differential equation; see \Cref{sec_example} below.
We shall refer to any measurable mapping from $\paramspace$ to $\statespace$ as a `model', and to any continuous linear mapping $\obsoperator$ from $\statespace$ to $\bb{R}^{n}$ as an `observation operator', where $n$ is fixed but arbitrary.
Thus, $\paramspace$ and $\statespace$ may be infinite-dimensional, but the data space is finite-dimensional. Infinite-dimensional data settings will involve additional technical issues --- see e.g. \cite[Remark 3.8]{Stuart2010} --- that are not directly related to the main ideas of this paper.

We shall use the following local Lipschitz stability result to prove our bounds.
\begin{theorem}
\label{thm_theorem11_and_proposition6_Sprungk2020}
 Let $\mu\in\mcal{P}(\paramspace)$ and $\Phi^{(i)}\in L^1_\mu(E,\bb{R}_{\geq 0})$ for $i=1,2$.
Then 
\begin{equation*}
 d_{\KL}(\mu_{\Phi^{(1)}}\Vert \mu_{\Phi^{(2)}})\leq 2\exp\left(\Norm{\Phi^{(1)}}_{L^1_\mu}+\Norm{\Phi^{(1)}-\Phi^{(2)}}_{L^1_\mu}\right)\Norm{\Phi^{(1)}-\Phi^{(2)}}_{L^1_\mu},
\end{equation*}
and thus $\mu_{\Phi^{(1)}}$ is absolutely continuous with respect to $\mu_{\Phi^{(2)}}$.
In particular,
\begin{align*}
 \max\{d_{\KL}(\mu_{\Phi^{(1)}}\Vert \mu_{\Phi^{(2)}}),d_{\KL}(\mu_{\Phi^{(2)}}\Vert \mu_{\Phi^{(1)}})\}\leq 2\exp\left(2\Norm{\Phi^{(1)}}_{L^1_\mu}+2\Norm{\Phi^{(2)}}_{L^1_\mu}\right)\Norm{\Phi^{(1)}-\Phi^{(2)}}_{L^1_\mu}
\end{align*}
and thus $\mu_{\Phi^{(1)}}$ and $\mu_{\Phi^{(2)}}$ are mutually equivalent.
\end{theorem}
\begin{proof}
 The first inequality follows by combining \cite[Theorem 11]{Sprungk2020} and \cite[Proposition 6]{Sprungk2020}.
 The second inequality follows by combining the triangle inequality with the first inequality.
\end{proof}

We now state the main assumptions of this paper.
Given $\paramspace$ and $\statespace$, there exists a unique best parameter $\param^{\best}\in \paramspace$ and unique best model $\model^{\best}:\paramspace\to\statespace$, such that for $\state^\best=\model^{\best}(\param^\best)$ and for a chosen observation operator $\obsoperator$, the data is a realisation of the random variable
\begin{equation}
 \label{eq_true_data}
 \dataRV =\obsoperator \state^{\best}+\noise = \obsoperator \model^\best(\param^\best)+\noise,\quad \noise \sim \normaldist{0}{\Sigma_\noise}
\end{equation}
where $\Sigma_\noise$ is positive definite.
We write $\obsoperator \model^\best(\param^\best)$ to emphasise that $\obsoperator$ is linear but $\model^\best$ may be nonlinear.
The data model \eqref{eq_true_data} leads to the best misfit $\Phi^{\data,\best}$, which together with the prior $\mu_\param$ yields the best posterior $\mu^{\data,\best}_{\param}$:
\begin{equation}
\label{eq_best_misfit_and_posterior}
 \paramspace\ni \param'\mapsto \Phi^{\data,\best} (\param')\coloneqq \tfrac{1}{2}\Abs{\data-\obsoperator \model^{\best}(\param')}^{2}_{\Sigma_\noise^{-1}},\quad \mu^{\data,\best}_{\param}\coloneqq (\mu_\param)_{\Phi^{\data,\best}}.
\end{equation}
We assume that $\Phi^{\data,\best}\in L^1_{\mu_\param}$, and that $\param\sim\mu_\param$ and $\noise$ are statistically independent.
The best posterior $\mu^{\data,\best}_{\param}$ is important because it is the posterior that is consistent with the true data model \eqref{eq_true_data}.

\subsection{Bound on KL divergence for approximate posterior}
\label{ssec_KL_bound_approximate_posterior}

In general, one cannot perform inference with the best model $\model^\best$, and one uses instead an approximate model $\model$ distinct from $\model^\best$.
For example, many mathematical models of natural phenomena based on differential equations do not admit closed-form solutions, so one must perform inference using some numerical approximation $\model$ of $\model^\best$.
The `model error' of $\model$ is given by the difference
\begin{equation}
 \label{eq_model_error_definition}
 \modelerror^{\best} \coloneqq \model^{\best}-\model: \paramspace\to\statespace.
\end{equation}
Given the assumption that $\noise\sim\normaldist{0}{\Sigma_\noise}$, we obtain the approximate misfit $\Phi^{\data,\appr}$ and the approximate posterior $\mu^{\data,\appr}_{\param}$ by replacing $\model^\best$ with the approximate model $\model$ in \eqref{eq_best_misfit_and_posterior}:
\begin{equation}
\label{eq_approx_misfit_and_posterior}
\param'\mapsto \Phi^{\data,\appr} (\param')\coloneqq \tfrac{1}{2}\Abs{\data-\obsoperator \model(\param')}^{2}_{\Sigma_{\noise}^{-1}},
 \quad \mu^{\data,\appr}_{\param}\coloneqq (\mu_\param)_{\Phi^{\data,\appr}}.
\end{equation}
By comparing the definitions of $\Phi^{\data,\appr}$ and $\Phi^{\data,\best}$ in \eqref{eq_best_misfit_and_posterior} and \eqref{eq_approx_misfit_and_posterior}, we conclude that if there exists some set $S\subset \paramspace$ such that $\mu_\param(S)>0$ and $\obsoperator\modelerror^\best$ does not vanish on $S$, then the likelihood corresponding to $\Phi^{\data,\appr}$ will be misspecified, and thus using $\mu^{\data,\appr}_\param$ instead of $\mu^{\data,\best}_\param$ will lead to incorrect inference.
To make this observation more quantitative, we first bound the misfit error in \Cref{lem_L1error_best_misfit_approx_misfit}, and then bound the KL divergence $d_{\KL}(\mu^{\data,\appr}_{\param}\Vert \mu^{\data,\best}_{\param})$ in \Cref{prop_KL_divergence_best_posterior_approx_posterior}.
\begin{lemma}
\label{lem_L1error_best_misfit_approx_misfit}
If $\Phi^{\data,\appr}\in L^1_{\mu_\param}$, then 
\begin{equation}
\label{eq_L1error_best_misfit_approx_misfit}
\Norm{ \Phi^{\data,\best} - \Phi^{\data,\appr} }_{L^{1}_{\mu_{\param}}}\leq 2^{-1/2} \Norm{ \Abs{\obsoperator \modelerror^{\best}}_{\Sigma_{\noise}^{-1}}^{2}}_{L^1_{\mu_\param}}^{1/2} \left( \Norm{ \Phi^{\data,\best}}^{1/2}_{L^{1}_{\mu_\param}}+\Norm{\Phi^{\data,\appr}}^{1/2}_{L^{1}_{\mu_\param}}\right).
\end{equation}
where $\Norm{ \Abs{\obsoperator \modelerror^{\best}}_{\Sigma_{\noise}^{-1}}^{2}}_{L^{1}_{\mu_\param}}\leq 2^{1/2}(\Norm{\Phi^{\data,\best}}^{1/2}_{L^{1}_{\mu_\param}}+ \Norm{\Phi^{\data,\appr}}^{1/2}_{L^{1}_{\mu_\param}})$.
\end{lemma}
We defer the proof of \Cref{lem_L1error_best_misfit_approx_misfit} to \Cref{ssec_proofs_KL_bound_approximate_posterior}.

Combining \eqref{eq_L1error_best_misfit_approx_misfit} with the bound on $\Norm{ \Abs{\obsoperator \modelerror^{\best}}_{\Sigma_{\noise}^{-1}}^{2}}_{L^{1}_{\mu_\param}}$ implies 
 \begin{equation*}
  \Norm{ \Phi^{\data,\best} - \Phi^{\data,\appr} }_{L^{1}_{\mu_{\param}}}\leq \left(\Norm{\Phi^{\data,\best}}^{1/2}_{L^{1}_{\mu_\param}}+ \Norm{\Phi^{\data,\appr}}^{1/2}_{L^{1}_{\mu_\param}}\right)^2.
 \end{equation*}
Since the right-hand side of the inequality above is larger than $\Norm{\Phi^{\data,\best}}_{L^{1}_{\mu_\param}}+ \Norm{\Phi^{\data,\appr}}_{L^{1}_{\mu_\param}}$, it follows that the bound given in \Cref{lem_L1error_best_misfit_approx_misfit} is not optimal, because it is worse than the bound we could obtain using the triangle inequality. 
Moreover, the term inside the parentheses on the right-hand side of \eqref{eq_L1error_best_misfit_approx_misfit} cannot be evaluated if $\model^{\best}$ is unavailable.
Nevertheless, the bound \eqref{eq_L1error_best_misfit_approx_misfit} is useful, because it bounds $\Norm{ \Phi^{\data,\best} - \Phi^{\data,\appr} }_{L^{1}_{\mu_{\param}}}$ in terms of the average observed model error $\Norm{ \Abs{\obsoperator \modelerror^{\best}}_{\Sigma_{\noise}^{-1}}^{2}}_{L^1_{\mu_\param}}^{1/2}$ and quantities that are assumed to be finite.

\begin{proposition}
\label{prop_KL_divergence_best_posterior_approx_posterior}
If $\Phi^{\data,\appr}\in L^{1}_{\mu_{\param}}$, then
\begin{align*}
 \max\{ d_{\KL}(\mu^{\data,\appr}_{\param}\Vert \mu^{\data,\best}_{\param}),d_{\KL}(\mu^{\data,\best}_{\param}\Vert \mu^{\data,\appr}_{\param})\}\leq & C  \Norm{ \Abs{\obsoperator \modelerror^{\best}}_{\Sigma_{\noise}^{-1}}^{2}}_{L^1_{\mu_\param}}^{1/2}
\end{align*} 
for $C=2^{1/2} \exp\left(2\Norm{\Phi^{\data,\best}}_{L^1_{\mu_\param}}+2\Norm{\Phi^{\data,\appr}}_{L^{1}_{\mu_\param}}\right)\left( \Norm{ \Phi^{\data,\best}}^{1/2}_{L^{1}_{\mu_\param}}+\Norm{\Phi^{\data,\appr}}^{1/2}_{L^{1}_{\mu_\param}}\right)$.
\end{proposition}
\begin{proof}[Proof of \Cref{prop_KL_divergence_best_posterior_approx_posterior}]
We apply \Cref{lem_L1error_best_misfit_approx_misfit} and the second statement of \Cref{thm_theorem11_and_proposition6_Sprungk2020} with $\Phi^{(1)}\leftarrow \Phi^{\data,\best}$, $\Phi^{(2)}\leftarrow \Phi^{\data,\appr}$, and $\mu\leftarrow \mu_{\param}$ to obtain
\begin{align*}
&\max\{ d_{\KL}(\mu^{\data,\appr}_{\param}\Vert \mu^{\data,\best}_{\param}),d_{\KL}(\mu^{\data,\best}_{\param}\Vert \mu^{\data,\appr}_{\param})\}
\\
\leq 
& 2\exp\left(2\Norm{\Phi^{\data,\appr}}_{L^1_{\mu_\param}}+2\Norm{\Phi^{\data,\best}}_{L^1_{\mu_\param}}\right)\Norm{\Phi^{\data,\best}-\Phi^{\data,\appr}}_{L^1_{\mu_\param}}
\\
\leq & 2^{1/2}\exp\left(2\Norm{\Phi^{\data,\appr}}_{L^1_{\mu_\param}}+2\Norm{\Phi^{\data,\best}}_{L^1_{\mu_\param}}\right)\left( \Norm{ \Phi^{\data,\best}}^{1/2}_{L^{1}_{\mu_\param}}+\Norm{\Phi^{\data,\appr}}^{1/2}_{L^{1}_{\mu_\param}}\right)\Norm{ \Abs{\obsoperator \modelerror^{\best}}_{\Sigma_{\noise}^{-1}}^{2}}_{L^1_{\mu_\param}}^{1/2}.
\end{align*}
This completes the proof of \Cref{prop_KL_divergence_best_posterior_approx_posterior}.
\end{proof}
Although $C$ in \Cref{prop_KL_divergence_best_posterior_approx_posterior} is not optimal, it is finite under the stated hypotheses.

\Cref{prop_KL_divergence_best_posterior_approx_posterior} shows that the Kullback--Leibler divergences between $\mu^{\data,\appr}_{\param}$ and $\mu^{\data,\best}_{\param}$ are controlled by the average observed model error $\Norm{ \Abs{\obsoperator \modelerror^{\best}}_{\Sigma_{\noise}^{-1}}^{2}}_{L^1_{\mu_\param}}^{1/2}$.
In order for these divergences to be small, one should choose the observation operator $\obsoperator$ such that $\modelerror^{\best}$ takes values in or near $\kernel(\obsoperator)$ with high $\mu_{\param}$-probability.
In particular, if $\mu_\param(\modelerror^{\best}\in\kernel(\obsoperator))=1$, then using $\mu^{\data,\appr}_{\param}$ will yield the same inference as $\mu^{\data,\best}_{\param}$. 
This is an example of a positive criterion for choosing observation operators to mitigate the effect of model error on Bayesian inference.

The condition $\mu_\param(\obsoperator \modelerror^{\best}=0)=1$ can be useful for guiding the choice of observation operator $\obsoperator$ even if $\modelerror^{\best}$ is not fully known.
For example, if one can determine a priori that $\modelerror^{\best}$ takes values in some proper subspace $V$ of $\statespace$, then any choice of observation operator $\obsoperator$ such that $V\subseteq \kernel(\obsoperator)$ will imply that $\mu^{\data,\appr}_{\param}=\mu^{\data,\best}_{\param}$.
However, the example $\obsoperator\equiv 0$ shows that one can choose $\obsoperator$ so that $\kernel(\obsoperator)$ is too large, in the sense that observations of state yield little or no information about the state itself.
Thus, the approach of choosing $\obsoperator$ to mitigate model error involves the following tradeoff: $\kernel(\obsoperator)$ should be as small as possible, but large enough to ensure that the model error takes values in $\kernel(\obsoperator$) $\mu_\param$-almost surely.
In \Cref{sec_example}, we revisit this observation on a specific example.

The preceding discussion emphasises the main idea of this work: to mitigate the effect of model error on Bayesian inference, one should exploit all available knowledge about the model error and about the inverse problem in order to guide the choice of observations.
This main idea is valid, independently of the fact that the bounds in \Cref{lem_L1error_best_misfit_approx_misfit} and \Cref{prop_KL_divergence_best_posterior_approx_posterior} are not sharp.

\subsection{Bounds on KL divergence for enhanced noise posterior}
\label{ssec_KL_bound_enhanced_noise}

Recall that the definition \eqref{eq_best_misfit_and_posterior} of the best misfit $\Phi^{\data,\best}$ follows from the best model \eqref{eq_true_data} for the data random variable $\dataRV$.
The approximate misfit in \eqref{eq_approx_misfit_and_posterior} can be seen as a misfit that results from assuming that $\obsoperator \modelerror^{\best}(\param^\best)=0$.
In the enhanced noise approach, one models the unknown state error $\modelerror^{\best}(\param^{\best})$ as a random variable $\enhancednoisemodel \sim \normaldist{m_\enhancednoisemodel}{\Sigma_\enhancednoisemodel}$ that is assumed to be statistically independent of $\param\sim\mu_\param$ and $\noise\sim\normaldist{0}{\Sigma_\noise}$ \cite{Kaipio2006}.
This is related to the so-called `pre-marginalisation' approach; see e.g. \cite{Kolehmainen2011}. 
We do not assume that $\text{supp}(\normaldist{m_\enhancednoisemodel}{\Sigma_\enhancednoisemodel})=\statespace$.
The corresponding enhanced noise data model is
\begin{equation*}
 \dataRV=\obsoperator \model(\param^{\best})+\obsoperator \enhancednoisemodel+\noise, 
\end{equation*}
where the `enhanced noise' $\obsoperator \enhancednoisemodel+\noise$ has the law $\normaldist{\obsoperator m_{\enhancednoisemodel}}{\Sigma_\noise+\obsoperator \Sigma_{\enhancednoisemodel}\obsoperator^{\ast}}$ because $\obsoperator$ is linear.
Since we assume that $\Sigma_\noise$ is positive definite and since $\obsoperator \Sigma_{\enhancednoisemodel} \obsoperator^{\ast}$ is nonnegative definite, $\enhancednoisecovariance$ is positive definite, hence invertible.

The enhanced noise misfit and enhanced noise posterior are
\begin{equation}
\label{eq_enhanced_noise_misfit_and_posterior}
\param'\mapsto \Phi^{\data,\enh} (\param')\coloneqq \tfrac{1}{2}\Abs{\data-\obsoperator \model(\param')-\obsoperator m_{\enhancednoisemodel}}^{2}_{(\enhancednoisecovariance)^{-1}},\quad \mu^{\data,\enh}_{\param}\coloneqq (\mu_\param)_{\Phi^{\data,\enh}}.
\end{equation}
\Cref{lem_L1error_best_misfit_enhanced_noise_misfit} below bounds the error between $\Phi^{\data,\best}$ and $\Phi^{\data,\enh}$ in terms of the shifted observed model error term $\obsoperator (\modelerror^{\best}-m_{\enhancednoisemodel})$ and the difference of covariance matrices $\Sigma_\noise^{-1}-(\enhancednoisecovariance)^{-1}$.
Define the scalar
\begin{equation}
 \label{eq_equivalence_of_norms_enhanced_noise}
 C_{\enh}\coloneqq \Norm{\Sigma_\noise^{-1/2}(\enhancednoisecovariance)^{1/2}}_{\op}.
\end{equation}
By \Cref{lem_linear_algebra}, $C_{\enh}$ satisfies $\Abs{z}_{\Sigma_\noise^{-1}}\leq C_{\enh}\Abs{z}_{(\Sigma_\noise+\obsoperator\Sigma_{\enhancednoisemodel} \obsoperator^{\ast})^{-1}}$ for all $z \in \bb{R}^{n}$.
\begin{lemma}
\label{lem_L1error_best_misfit_enhanced_noise_misfit}
If $\Phi^{\data,\enh}\in L^1_{\mu_\param}$, then 
\begin{align}
\Norm{ \Phi^{\data,\best} - \Phi^{\data,\enh} }_{L^{1}_{\mu_{\param}}}\leq &  2^{-1/2} \Norm{ \Abs{ \obsoperator (\modelerror^{\best}-m_{\enhancednoisemodel})}_{\Sigma_{\noise}^{-1}}^{2}}_{L^1_{\mu_\param}}^{1/2} \left( \Norm{\Phi^{\data,\best}}^{1/2}_{L^1_{\mu_\param}}+ C_{\enh}\Norm{\Phi^{\data,\enh}}_{L^1_{\mu_\param}}^{1/2}\right)
\label{eq_L1error_best_misfit_enhanced_noise_misfit}
\\
&+2^{-1}\Norm{ \Abs{\data   -\obsoperator \model-\obsoperator m_{\enhancednoisemodel}}^{2}_{ \Sigma_\noise^{-1}-(\enhancednoisecovariance)^{-1}}}_{L^1_{\mu_\param}}.
\nonumber
\end{align}
Furthermore,
\begin{align*}
 \Norm{ \Abs{ \obsoperator (\modelerror^{\best} - m_{\enhancednoisemodel})}_{\Sigma_{\noise}^{-1}}^{2}}_{L^1_{\mu_\param}}^{1/2}
\leq & 2^{1/2}\left(\Norm{\Phi^{\data,\best}}_{L^1_{\mu_\param}}^{1/2}+C_{\enh}\Norm{\Phi^{\data,\enh}}_{L^1_{\mu_\param}}^{1/2}\right),
\\
\Norm{ \Abs{\data   -\obsoperator \model-\obsoperator m_{\enhancednoisemodel}}^{2}_{ \Sigma_\noise^{-1}-(\enhancednoisecovariance)^{-1}}}_{L^1_{\mu_\param}} \leq & (C_{\enh}+1)\Norm{2\Phi^{\data,\enh}}_{L^1_{\mu_\param}}.
\end{align*}
\end{lemma}
We defer the proof of \Cref{lem_L1error_best_misfit_enhanced_noise_misfit} to \Cref{ssec_proofs_KL_bound_enhanced_noise_posterior}.
\begin{proposition}
\label{prop_KL_divergence_best_posterior_enhanced_noise_posterior}
If $\Phi^{\data,\enh}\in L^{1}_{\mu_{\param}}$, then
\begin{align*}
 &\max\{d_{\KL}(\mu^{\data,\best}_{\param}\Vert \mu^{\data,\enh}_{\param}),d_{\KL}(\mu^{\data,\enh}_{\param}\Vert \mu^{\data,\best}_{\param})\}
 \\
 \leq & C\left( \Norm{ \Abs{ \obsoperator (\modelerror^{\best}- m_{\enhancednoisemodel})}_{\Sigma_{\noise}^{-1}}^{2}}_{L^1_{\mu_\param}}^{1/2} +\Norm{ \Abs{\data   -\obsoperator \model-\obsoperator m_{\enhancednoisemodel}}^{2}_{ \Sigma_\noise^{-1}-(\enhancednoisecovariance)^{-1}}}_{L^1_{\mu_\param}}\right)
\end{align*} 
for $C= \exp\left(2\Norm{\Phi^{\data,\best}}_{L^1_{\mu_\param}}+2\Norm{\Phi^{\data,\enh}}_{L^1_{\mu_\param}}\right)\max\left\{\sqrt{2}\left( \Norm{\Phi^{\data,\best}}^{1/2}_{L^1_{\mu_\param}}+ C_{\enh}\Norm{\Phi^{\data,\enh}}_{L^1_{\mu_\param}}^{1/2}\right),1\right\}$.
\end{proposition}
\begin{proof}[Proof of \Cref{prop_KL_divergence_best_posterior_enhanced_noise_posterior}]
Given that $\Phi^{\data,\enh},\Phi^{\data,\best}\in L^1_{\mu_\param}$, we may apply \Cref{thm_theorem11_and_proposition6_Sprungk2020} with $\Phi^{(1)}\leftarrow \Phi^{\data,\best}$, $\Phi^{(2)}\leftarrow \Phi^{\data,\enh}$, and $\mu\leftarrow \mu_{\param}$, to obtain
\begin{align*}
  &\max\{d_{\KL}(\mu^{\data,\best}_{\param}\Vert \mu^{\data,\enh}_{\param}), d_{\KL}(\mu^{\data,\enh}_{\param} \Vert \mu^{\data,\best}_{\param})\}
  \\
  \leq & 2\exp\left(2\Norm{\Phi^{\data,\enh}}_{L^1_{\mu_\param}}+2\Norm{\Phi^{\data,\best}}_{L^1_{\mu_\param}}\right)\Norm{\Phi^{\data,\best}-\Phi^{\data,\enh}}_{L^1_{\mu_\param}}
  \\
  \leq & \exp\left(2\Norm{\Phi^{\data,\enh}}_{L^1_{\mu_\param}}+2\Norm{\Phi^{\data,\best}}_{L^1_{\mu_\param}}\right)\max\left\{ \sqrt{2} \left( \Norm{\Phi^{\data,\best}}^{1/2}_{L^1_{\mu_\param}}+ C_{\enh}\Norm{\Phi^{\data,\enh}}_{L^1_{\mu_\param}}^{1/2}\right),1\right\}
\\
& \times \left(\Norm{ \Abs{ \obsoperator (\modelerror^{\best}-m_{\enhancednoisemodel})}_{\Sigma_{\noise}^{-1}}^{2}}_{L^1_{\mu_\param}}^{1/2} +\Norm{ \Abs{\data   -\obsoperator \model-\obsoperator m_{\enhancednoisemodel}}^{2}_{ \Sigma_\noise^{-1}-(\enhancednoisecovariance)^{-1}}}_{L^1_{\mu_\param}}\right)
\end{align*}
where the second inequality follows from the bound \eqref{eq_L1error_best_misfit_enhanced_noise_misfit} of \Cref{lem_L1error_best_misfit_enhanced_noise_misfit}.
\end{proof}

The significance of \Cref{prop_KL_divergence_best_posterior_enhanced_noise_posterior} is similar to that of \Cref{prop_KL_divergence_best_posterior_approx_posterior}.
The main difference consists in the fact that the error with respect to the best misfit is now controlled by
\begin{equation}
\label{eq_enhanced_noise_sum_to_vanish}
\Norm{ \Abs{ \obsoperator (\modelerror^{\best}-m_{\enhancednoisemodel})}_{\Sigma_{\noise}^{-1}}^{2}}_{L^1_{\mu_\param}}^{1/2}+\Norm{ \Abs{\data   -\obsoperator \model-\obsoperator m_{\enhancednoisemodel}}^{2}_{ \Sigma_\noise^{-1}-(\enhancednoisecovariance)^{-1}}}_{L^1_{\mu_\param}},
\end{equation}
where only the first term depends on the model error $\modelerror^\best$.
Since we assume that $\Sigma_{\noise}$ is positive definite, the first term in \eqref{eq_enhanced_noise_sum_to_vanish} vanishes if and only if $\modelerror^{\best}-m_{\enhancednoisemodel} \in \kernel(\obsoperator)$ $\mu_\param$-almost surely.
This condition differs from the sufficient condition for $\mu^{\data,\best}_{\param}=\mu^{\data,\appr}_{\param}$ that was implied by \Cref{lem_L1error_best_misfit_approx_misfit}, namely, that $\modelerror^{\best}\in \kernel(\obsoperator)$ $\mu_\param$-almost surely.
By \eqref{eq_matrix_weighted_inner_product_norm}, the second term in \eqref{eq_enhanced_noise_sum_to_vanish} vanishes if and only if
\begin{equation}
\label{eq_enhanced_noise_second_term_sum_to_vanish_equivalent_condition}
\mu_\param\left(\data   -\obsoperator \model-\obsoperator m_{\enhancednoisemodel}\in \kernel\left(\Sigma_\noise^{-1}-(\enhancednoisecovariance)^{-1}\right)\right) =1.
\end{equation}
If $\bb{P}(\enhancednoisemodel\in m_{\enhancednoisemodel}+\kernel(\obsoperator))=1$, then $\obsoperator \enhancednoisemodel=\obsoperator m_\enhancednoisemodel$ almost surely, which implies that $\obsoperator \Sigma_\enhancednoisemodel\obsoperator^\ast=0$.
This in turn implies that \eqref{eq_enhanced_noise_second_term_sum_to_vanish_equivalent_condition} holds.
Thus, if $\mu_\param(\modelerror^{\best}\in m_{\enhancednoisemodel}+ \kernel(\obsoperator))=1$ and $\bb{P}(\enhancednoisemodel\in m_{\enhancednoisemodel}+\kernel(\obsoperator))=1$, then $\mu^{\data,\best}_{\param}=\mu^{\data,\enh}_{\param}$.
These sufficient conditions do not assume that $m_{\enhancednoisemodel}\in\kernel(\obsoperator)$.
In fact, if in addition to $\obsoperator \Sigma_{\enhancednoisemodel}\obsoperator^{\ast}=0$ it also holds that $m_{\enhancednoisemodel}\in\kernel(\obsoperator)$, then it follows from \eqref{eq_enhanced_noise_misfit_and_posterior} that $\Phi^{\data,\enh}=\Phi^{\data,\appr}$.
The next lemma provides a more quantitative version of this statement, using the constant $C_{\enh}$ given in \eqref{eq_equivalence_of_norms_enhanced_noise}.
\begin{lemma}
\label{lem_L1norm_diff_Phi_approx_minus_Phi_enhanced}
If $\Phi^{\data,\appr},\Phi^{\data,\enh}\in L^1_{\mu_\param}$, then 
\begin{align}
 \Norm{ \Phi^{\data,\appr} - \Phi^{\data,\enh} }_{L^{1}_{\mu_{\param}}}\leq & 2^{-1/2}\Abs{ \obsoperator m_{\enhancednoisemodel}}_{\Sigma_{\noise}^{-1}}\left( \Norm{\Phi^{\data,\appr}}^{1/2}_{L^1_{\mu_\param}}+C_{\enh} \Norm{ \Phi^{\data,\enh}}^{1/2}_{L^1_{\mu_\param}}\right) 
\label{eq_L1norm_diff_Phi_approx_minus_Phi_enhanced}
\\
&+2^{-1}\Norm{ \Abs{\data   -\obsoperator \model -\obsoperator m_{\enhancednoisemodel}}^{2}_{ \Sigma_\noise^{-1}-(\enhancednoisecovariance)^{-1}}}_{L^1_{\mu_\param}}.
\nonumber
\end{align}
\end{lemma}
For the proof of \Cref{lem_L1norm_diff_Phi_approx_minus_Phi_enhanced}, see \Cref{ssec_proofs_KL_bound_enhanced_noise_posterior}.

\begin{proposition}
\label{prop_KL_divergence_approx_posterior_enhanced_noise_posterior}
If $\Phi^{\data,\appr}, \Phi^{\data,\enh}\in L^{1}_{\mu_{\param}}$, then
\begin{align*}
&\max\{d_{\KL}(\mu^{\data,\appr}_{\param}\Vert \mu^{\data,\enh}_{\param}),d_{\KL}(\mu^{\data,\enh}_{\param}\Vert \mu^{\data,\appr}_{\param})\}
 \\
 \leq & C\left(  \Abs{ \obsoperator m_{\enhancednoisemodel}}_{\Sigma_{\noise}^{-1}} +\Norm{ \Abs{\data   -\obsoperator \model-\obsoperator m_{\enhancednoisemodel}}^{2}_{ \Sigma_\noise^{-1}-(\enhancednoisecovariance)^{-1}}}_{L^1_{\mu_\param}}\right)
\end{align*} 
for $C= \exp\left(2\Norm{\Phi^{\data,\appr}}_{L^1_{\mu_\param}}+2\Norm{\Phi^{\data,\enh}}_{L^1_{\mu_\param}}\right)\max\left\{\sqrt{2}\left( \Norm{\Phi^{\data,\appr}}^{1/2}_{L^1_{\mu_\param}}+ C_{\enh}\Norm{\Phi^{\data,\enh}}_{L^1_{\mu_\param}}^{1/2}\right),1\right\}$.
\end{proposition}
The proof of \Cref{prop_KL_divergence_approx_posterior_enhanced_noise_posterior} is similar to that of \Cref{prop_KL_divergence_best_posterior_enhanced_noise_posterior} except that one uses \Cref{lem_L1norm_diff_Phi_approx_minus_Phi_enhanced} instead of \Cref{lem_L1error_best_misfit_enhanced_noise_misfit}, so we omit it.

\Cref{prop_KL_divergence_approx_posterior_enhanced_noise_posterior} implies the following: given an enhanced noise model with mean $m_{\enhancednoisemodel}$ and covariance $\Sigma_{\enhancednoisemodel}$, if one chooses the observation operator $ \obsoperator$ so that
\begin{equation*}
  \Abs{ \obsoperator m_{\enhancednoisemodel}}_{\Sigma_{\noise}^{-1}}=0=\Norm{ \Abs{\data   -\obsoperator \model-\obsoperator m_{\enhancednoisemodel}}^{2}_{ \Sigma_\noise^{-1}-(\enhancednoisecovariance)^{-1}}}_{L^1_{\mu_\param}},
\end{equation*}
then the enhanced noise posterior $\mu^{\data,\enh}_{\param}$ and the approximate posterior $\mu^{\data,\appr}_{\param}$ coincide.
Equivalently, $\obsoperator m_{\enhancednoisemodel}=0$ and \eqref{eq_enhanced_noise_second_term_sum_to_vanish_equivalent_condition} together imply that $\mu^{\data,\enh}_{\param}=\mu^{\data,\appr}_{\param}$.
Thus, such a choice of observation operator yields an enhanced noise posterior $\mu^{\data,\enh}_{\param}$ that does not account for model error.
In this case, the approximate posterior $\mu^{\data,\appr}_{\param}$ and the enhanced noise posterior $\mu^{\data,\enh}_{\param}$ perform equally well in terms of inference, and it would be better to use the approximate posterior instead, since it is simpler.
Thus, the conditions $\obsoperator m_{\enhancednoisemodel}=0$ and \eqref{eq_enhanced_noise_second_term_sum_to_vanish_equivalent_condition} give an example of negative criteria that one can use to exclude observation operators from consideration when one aims to mitigate the effect of model error on Bayesian inference.

\subsection{Bounds on KL divergence for joint posterior}
\label{ssec_KL_error_joint_posterior}

In the enhanced noise approach given by \eqref{eq_enhanced_noise_misfit_and_posterior}, one accounts for the uncertainty in the unknown state error $\modelerror^{\best}(\param^{\best})$ by approximating the unknown state error using a random variable $\enhancednoisemodel$.
The goal is to infer $\param^\best$ only.
In the joint parameter-error inference approach, one aims to infer $(\param^{\best},\modelerror^{\best})$ jointly, by using a random variable $(\param,\modelerror)$ with prior $\mu_{\param,\modelerror}$ and using Bayes' formula.
The joint prior on the random variable $(\param,\modelerror)$ is given by the product $\mu_\param\otimes \mu_\modelerror$ of the prior $\mu_\param$ on the unknown $\param^\best$ and a prior $\mu_\modelerror$ on the unknown $\modelerror^\best$.
It is possible to specify the prior $\mu_\modelerror$ on the unknown $\modelerror^\best$ without knowing $\modelerror^\best$, but one must specify a space $\modelerrorspace$ of admissible values of $\modelerror^\best$.
We assume that the space $\modelerrorspace$ is a Radon space and that $\mu_\modelerror$ is a Borel probability measure.

The assumption that the prior $\mu_{\param,\modelerror}$ on $(\param,\modelerror)$ has product structure is equivalent to the assumption that $\param$ and $\modelerror$ are independent random variables.
This independence assumption is reasonable, since by \eqref{eq_model_error_definition} it follows that $\modelerror^\best\coloneqq \model^\best-\model$ does not depend on $\param^\best$.

Define the joint misfit and joint posterior by
\begin{equation}
\label{eq_joint_misfit_and_posterior}
\Phi^{\data,\joint} (\param',\modelerror')\coloneqq \tfrac{1}{2}\Abs{\data-\obsoperator   \model(\param')-\obsoperator  \modelerror'(\param')}^{2}_{\Sigma_\noise^{-1}},
 \quad \mu^{\data,\joint}_{\param,\modelerror}(\param',\modelerror') \coloneqq  (\mu_\param \otimes\mu_\modelerror)_{\Phi^{\data,\joint}}.
\end{equation}
An important disadvantage of jointly inferring the parameter and model error is that the dimension of the space on which one performs inference increases; this tends to make the inference task more computationally expensive. 
It is also known that the problem of identifiability may arise; see e.g. \cite{Brynjarsdottir2014}.
We shall not consider the problem of identifiability here.
On the other hand, jointly inferring the parameter and model error is consistent with the Bayesian approach of treating all unknowns as random variables and updating these distributions using the data.
In addition, the joint inference approach offers the possibility to improve the approximate model $\model$ in a data-driven way, e.g. by using the $\modelerror$-marginal posterior mean.
This improvement can be used to obtain better estimates of both the parameter $\param^{\best}$ and the state $\state^{\best}=\model^{\best}(\param^{\best})$.

To compare the joint posterior $\mu^{\data,\joint}_{\param}$ with the posterior measures $\mu^{\data,\best}_{\param}$, $\mu^{\data,\appr}_{\param}$ and $\mu^{\data,\enh}$ that we introduced earlier, we `lift' the misfits and posteriors according to
\begin{equation}
\label{eq_lifted_misfit_and_posterior}
  \Phi^{\data,\bullet}(\param',\modelerror') \coloneqq  \Phi^{\data,\bullet}(\param'),\qquad \mu^{\data,\bullet}_{\param,\modelerror}\coloneqq (\mu_{\param}\otimes \mu_\modelerror)_{\Phi^{\data,\bullet}},\quad \bullet\in\{\appr,\enh,\best\}.
  \end{equation}
In \eqref{eq_lifted_misfit_and_posterior}, we use the notation $\Phi^{\data,\bullet}$ to refer to two functions defined on different domains. 
The following lemma shows that this does not result in any fundamental changes, and that the lifted posterior is the product of the original posterior with the prior law $\mu_\modelerror$ for the unknown $\modelerror^\best$.
\begin{lemma}
 \label{lem_properties_of_lifted_placeholder_objects}
 For $\bullet \in \{\appr,\best,\enh\}$, 
 \begin{equation}
\label{eq_lifted_posteriors_product_structure}
 \mu^{\data,\bullet}_{\param,\modelerror}=  \mu^{\data,\bullet}_{\param}\otimes \mu_\modelerror.
\end{equation}
If $\Phi^{\data,\bullet}\in L^q_{\mu_\param}$ for some $q>0$, then 
\begin{equation}
\label{eq_lifted_misfits_same_norms}
\Norm{\Phi^{\data,\bullet}}_{L^q_{\mu_\param\otimes\mu_\modelerror}}=\Norm{\Phi^{\data,\bullet}}_{L^q_{\mu_\param}}.
\end{equation}
\end{lemma}
\begin{proof}
Since $\mu_\modelerror$ is a probability measure, the definition of $\Phi^{\data,\bullet}(\param',\modelerror')$ in \eqref{eq_lifted_misfit_and_posterior} implies that
\begin{align*}
 \int_{\paramspace\times\modelerrorspace} \exp(-\Phi^{\data,\bullet}(\param',\modelerror'))\rd\mu_\param\otimes\mu_\modelerror(\param',\modelerror')=\int_\paramspace \exp(-\Phi^{\data,\bullet}(\param'))\rd\mu_\param(\param').
\end{align*}
Thus, the normalisation constant for the likelihood $\exp(-\Phi^{\data,\bullet})$ that generates the lifted posterior $\mu^{\data,\bullet}_{\param,\modelerror}$ on $\paramspace\times\modelerrorspace$ agrees with the corresponding normalisation constant for the posterior $\mu^{\data,\bullet}_{\param}$ on $\paramspace$.
This proves \eqref{eq_lifted_posteriors_product_structure}.
The fact that $\Phi^{\data,\bullet}$ is constant with respect to $\modelerror'$ implies the second statement of the lemma.
\end{proof}

Recall that $\bb{P}$ denotes the probability measure in the underlying probability space.
Below, we will write the random variables $\param$ and $\modelerror$ explicitly, and take $L^p$-norms with respect to $\bb{P}$ instead of $\mu_\param\otimes\mu_\modelerror$, e.g.
\begin{equation*}
 \Norm{ \Abs{\obsoperator (\modelerror^{\best}-\modelerror)(\param)}_{\Sigma_\noise^{-1}}^{2}}_{L^1_{\bb{P}}}=\int_{\paramspace\times\modelerrorspace} \Abs{\obsoperator   (\modelerror^{\best}-\modelerror')(\param')}_{\Sigma_\noise^{-1}}^{2}\rd \mu_\param\otimes\mu_\modelerror(\param',\modelerror').
\end{equation*}
See \eqref{eq_L1norm_diff_Phi_best_minus_Phi_joint} below for an example.
\begin{lemma}
 \label{lem_L1norm_diff_Phi_best_minus_Phi_joint}
 If $\Phi^{\data,\joint}\in L^1_{\mu_\param\otimes\mu_\modelerror}$, then
 \begin{equation}
 \label{eq_L1norm_diff_Phi_best_minus_Phi_joint}
  \Norm{\Phi^{\data,\best}-\Phi^{\data,\joint}}_{L^1_{\mu_\param\otimes\mu_\modelerror}}\leq 2^{-1/2}\Norm{ \Abs{\obsoperator (\modelerror^{\best}-\modelerror)(\param)}_{\Sigma_\noise^{-1}}^{2}}_{L^1_{\bb{P}}}^{1/2} \left( \Norm{\Phi^{\data,\joint}}_{L^1_{\mu_\param\otimes\mu_\modelerror}}^{1/2}+ \Norm{\Phi^{\data,\best}}_{L^1_{\mu_\param}}^{1/2}\right).
 \end{equation}
 Furthermore, $\Norm{ \Abs{\obsoperator (\modelerror^{\best}-\modelerror)(\param)}_{\Sigma_\noise^{-1}}^{2}}_{L^1_{\bb{P}}}^{1/2} \leq 2^{1/2}\left( \Norm{\Phi^{\data,\joint}}_{L^1_{\mu_\param\otimes\mu_\modelerror}}^{1/2}+ \Norm{\Phi^{\data,\best}}_{L^1_{\mu_\param}}^{1/2}\right)$.
\end{lemma}
\begin{proposition}
 \label{prop_KL_divergence_best_posterior_joint_posterior}
 If $\Phi^{\data,\joint}\in L^1_{\mu_\param\otimes\mu_\modelerror}$, then $\mu^{\data,\joint}_{\param,\modelerror}$ and $\mu^{\data,\best}_{\param,\modelerror}$ satisfy
 \begin{align*}
  &\max\{ d_{\KL}(\mu^{\data,\best}_{\param,\modelerror}\Vert\mu^{\data,\joint}_{\param,\modelerror}), d_{\KL}(\mu^{\data,\joint}_{\param,\modelerror}\Vert \mu^{\data,\best}_{\param,\modelerror})\}
  \leq C  \Norm{ \Abs{\obsoperator (\modelerror^{\best}-\modelerror)(\param)}_{\Sigma_\noise^{-1}}^{2}}_{L^1_{\bb{P}}}^{1/2},
 \end{align*}
 where $C=2^{1/2} \exp\left(2 \Norm{\Phi^{\data,\joint}}_{L^1_{\mu_\param\otimes\mu_\modelerror}}+2\Norm{\Phi^{\data,\best}}_{L^1_{\mu_\param}}\right) \left( \Norm{\Phi^{\data,\joint}}_{L^1_{\mu_\param\otimes\mu_\modelerror}}^{1/2}+ \Norm{\Phi^{\data,\best}}_{L^1_{\mu_\param}}^{1/2}\right)$.
\end{proposition}
We prove \Cref{lem_L1norm_diff_Phi_best_minus_Phi_joint} in \Cref{ssec_proofs_KL_bound_joint_posterior}.
The proof of \Cref{prop_KL_divergence_best_posterior_joint_posterior} is similar to the proofs of \Cref{prop_KL_divergence_best_posterior_approx_posterior} and \Cref{prop_KL_divergence_best_posterior_enhanced_noise_posterior}, so we omit it.

By \Cref{prop_KL_divergence_best_posterior_joint_posterior}, a sufficient condition for $\mu^{\data,\joint}_{\param,\modelerror}=\mu^{\data,\best}_{\param,\modelerror}$ is that 
\begin{equation*}
 \bb{P}\left(\modelerror^{\best}(\param)-\modelerror(\param)\in\kernel(\obsoperator)\right)=1.
\end{equation*}
In order to use this condition to choose a prior $\mu_\modelerror$ and observation operator $\obsoperator$, one must have some a priori knowledge about the set $\{\modelerror^{\best}(\param')\ :\ \param'\in\paramspace\}$, e.g. that this set is contained in a proper affine subspace $x+V=\{x+v\ :\ v\in V\}$ of $\statespace$, where $x$ and $V$ are an element and subspace of $\statespace$ respectively.
However, by the discussion before \Cref{ssec_KL_bound_enhanced_noise}, if one knows that $\{\modelerror^{\best}(\param')\ :\ \param'\in\paramspace\}\subseteq x+V$, then one can use this information to choose $\obsoperator$ so that $\mu^{\data,\appr}_{\param}=\mu^{\data,\best}_{\param}$, e.g. by setting $\obsoperator$ so that $\spn(x+V)\subseteq \kernel(\obsoperator)$.

Next, we consider the lifted approximate posterior $\mu^{\data,\appr}_{\param,\modelerror}$ and the joint posterior $\mu^{\data,\joint}_{\param,\modelerror}$.
\begin{lemma}
 \label{lem_L1norm_diff_Phi_approx_minus_Phi_joint}
 If $\Phi^{\data,\appr}\in L^1_{\mu_\param}$ and $\Phi^{\data,\joint}\in L^1_{\mu_\param\otimes\mu_\modelerror}$, then
 \begin{equation}
 \label{eq_L1norm_diff_Phi_approx_minus_Phi_joint}
    \Norm{\Phi^{\data,\appr}-\Phi^{\data,\joint}}_{L^1_{\mu_\param\otimes\mu_\modelerror}}\leq 2^{-1/2}\Norm{ \Abs{\obsoperator \modelerror(\param)}_{\Sigma_\noise^{-1}}^{2}}_{L^1_{\bb{P}}}^{1/2} \left( \Norm{\Phi^{\data,\joint}}_{L^1_{\mu_\param\otimes\mu_\modelerror}}^{1/2}+ \Norm{\Phi^{\data,\appr}}_{L^1_{\mu_\param}}^{1/2}\right).
 \end{equation}
 Furthermore, $\Norm{ \Abs{\obsoperator \modelerror(\param)}_{\Sigma_\noise^{-1}}^{2}}_{L^1_{\bb{P}}}^{1/2} \leq 2^{1/2}\left( \Norm{\Phi^{\data,\joint}}_{L^1_{\mu_\param\otimes\mu_\modelerror}}^{1/2}+ \Norm{\Phi^{\data,\appr}}_{L^1_{\mu_\param}}^{1/2}\right)$.
\end{lemma}
\begin{proposition}
 \label{prop_KL_divergence_approx_posterior_joint_posterior}
 If $\Phi^{\data,\appr}\in L^1_{\mu_\param}$ and $\Phi^{\data,\joint}\in L^1_{\mu_\param\otimes\mu_\modelerror}$, then
 \begin{align*}
  &\max\{ d_{\KL}(\mu^{\data,\appr}_{\param,\modelerror}\Vert\mu^{\data,\joint}_{\param,\modelerror}), d_{\KL}(\mu^{\data,\joint}_{\param,\modelerror}\Vert \mu^{\data,\appr}_{\param,\modelerror})\}
  \leq C \Norm{ \Abs{\obsoperator \modelerror(\param)}_{\Sigma_\noise^{-1}}^{2}}_{L^1_{\bb{P}}}^{1/2},
 \end{align*}
 where $C=2^{1/2} \exp\left(2 \Norm{\Phi^{\data,\joint}}_{L^1_{\mu_\param\otimes\mu_\modelerror}}+2\Norm{\Phi^{\data,\appr}}_{L^1_{\mu_\param}}\right) \left( \Norm{\Phi^{\data,\joint}}_{L^1_{\mu_\param\otimes\mu_\modelerror}}^{1/2}+ \Norm{\Phi^{\data,\appr}}_{L^1_{\mu_\param}}^{1/2}\right)$.
\end{proposition}
We prove \Cref{lem_L1norm_diff_Phi_approx_minus_Phi_joint} in \Cref{ssec_proofs_KL_bound_joint_posterior} and omit the proof of \Cref{prop_KL_divergence_approx_posterior_joint_posterior} due to its similarity with the proofs of \Cref{prop_KL_divergence_best_posterior_approx_posterior} and \Cref{prop_KL_divergence_best_posterior_enhanced_noise_posterior}.

Recall that if $ \bb{P}\left(\modelerror^{\best}(\param)-\modelerror(\param)\in\kernel(\obsoperator)\right)=1$, then the joint posterior $\mu^{\data,\joint}_{\param,\modelerror}$ and lifted best posterior $\mu^{\data,\best}_{\param,\modelerror}$ agree.
\Cref{prop_KL_divergence_approx_posterior_joint_posterior} shows that if $\bb{P}\left(\modelerror(\param)\in\kernel(\obsoperator)\right)=1$, then the joint posterior $\mu^{\data,\joint}_{\param,\modelerror}$ and the lifted aproximate posterior $\mu^{\data,\appr}_{\param,\modelerror}$ agree.
In this case, there is no additional benefit in terms of inference from using the joint posterior, so $\bb{P}\left(\modelerror(\param)\in\kernel(\obsoperator)\right)=1$ is a negative criterion by which one can exclude an observation operator from consideration when one aims to mitigate the effect of model error on Bayesian inference.
We can interpret the negation of this condition, i.e. that $\modelerror(\param)\notin\kernel(\obsoperator)$ with positive $\bb{P}$-probability, as follows: if the joint posterior $\mu^{\data,\joint}_{\param,\modelerror}$ differs from the lifted approximate posterior $\mu^{\data,\appr}_{\param,\modelerror}$, then the observed state update $\obsoperator\modelerror(\param)$ cannot vanish $\bb{P}$-almost surely.
This implies that the model $\model$ must be updated by some nonzero $\modelerror'\in\modelerrorspace$ with positive $\bb{P}$-probability.

The preceding observation establishes a connection with the so-called `parametrised background data-weak approach' for state estimation.
Consider the setting where the state space $\statespace$ is a Hilbert space and $\obsoperator=(\ell_1,\ldots,\ell_n)$ for $n$ continuous linear functionals $(\ell_i)_{i\in[n]}\subset\statespace^\ast$.
By a Hilbert space identity, $\kernel(\obsoperator)$ is the orthogonal complement of the closure of the range of $\obsoperator^\ast$.
Denote the Riesz representation operator by $R$, so that $R\ell_i\in\statespace$ for every $i\in[n]$.
Then one can verify by direct calculations that $\range(\obsoperator^\ast)=\spn(R\ell_i,\ i\in[n])$.
Since every finite-dimensional subspace of $\statespace$ is closed, $\modelerror(\param)\notin\kernel(\obsoperator)$ holds if $\modelerror(\param) \in \spn(R\ell_i,\ i\in[n])$.
The subspace $\spn(R\ell_i,\ i\in[n])$ corresponds to the `variational update space' in \cite[Section 2.3]{maday2019adaptive}. 
The term `update space' refers to the goal of updating an existing model that maps parameters to states in order to improve state prediction.

It is possible to state and prove the analogues of the preceding bounds for the lifted enhanced noise posterior $\mu^{\data,\enh}_{\param,\modelerror}$.
These bounds are not relevant for the main goal of this paper.
However, for the sake of completeness, we state them in \Cref{ssec_proofs_KL_bound_joint_posterior}.

\subsection{Kullback--Leibler error of marginal posterior}
\label{ssec_KL_error_marginal_posterior}

We define the marginal posterior as the $\param$-marginal of the joint posterior $\mu^{\data,\joint}_{\param,\modelerror}$:
\begin{equation}
\label{eq_marginal_posterior}
 \mu^{\data,\marg}_{\param}(S)=\int_{S\times\modelerrorspace} \rd\mu^{\data,\joint}_{\param,\modelerror}(\param',\modelerror'),\quad S\in\mcal{B}(\paramspace).
\end{equation}
One can approximate the marginal posterior $\mu^{\data,\marg}_{\param}$ by using Monte Carlo integration of the joint posterior $\mu^{\data,\joint}_{\param,\modelerror}$ over possible realisations of $\modelerror$, for example.
Since the marginal posterior requires the joint posterior, it inherits the disadvantages and advantages of the joint posterior.

In \Cref{ssec_KL_error_joint_posterior}, we bounded the KL divergence of the joint posterior $\mu^{\data,\joint}_{\param,\modelerror}$ with respect to the lifted posteriors $\mu^{\data,\bullet}_{\param,\modelerror}$ for $\bullet\in\{\appr,\best,\enh\}$ that were defined in \eqref{eq_lifted_misfit_and_posterior}, and we observed in \Cref{lem_properties_of_lifted_placeholder_objects} that the $\param$-marginal of the lifted posterior $\mu^{\data,\bullet}_{\param,\modelerror}$ is exactly $\mu^{\data,\bullet}_{\param}$.
These bounds imply analogous bounds on the KL divergence of the marginal posterior $\mu^{\data,\marg}$ with respect to $\mu^{\data,\best}_{\param}$, $\mu^{\data,\appr}_{\param}$, and $\mu^{\data,\enh}_{\param}$.

Below, $\mu^{\data,\joint}_{\modelerror|\param}$ denotes the regular version of the joint posterior law of $\modelerror$ conditioned on $\param$.
Given our assumption that the random variable $\modelerror$ takes values in a Radon space, it follows that the regular conditional distribution $\mu^{\data,\joint}_{\modelerror|\param}$ exists and is unique up to sets of $\mu_\param$-measure zero.
\begin{proposition}
 \label{prop_KL_divergence_marginal_posterior}
 Let $\Phi^{\data,\joint}\in L^1_{\mu_\param\otimes\mu_\modelerror}$ and $\bullet\in\{\appr,\best,\enh\}$.
 Suppose $\Phi^{\data,\bullet}\in L^1_{\mu_\param}$.
 Then 
 \begin{align*}
  d_{\KL}(\mu^{\data,\marg}_{\param}\Vert \mu^{\data,\bullet}_{\param})=& d_{\KL}(\mu^{\data,\joint}_{\param,\modelerror}\Vert \mu^{\data,\bullet}_{\param,\modelerror})-\int_{\paramspace} d_{\KL}(\mu^{\data,\joint}_{\modelerror|\param}\Vert \mu_\modelerror)\rd \mu^{\data,\marg}_\param,
  \\
  d_{\KL}(\mu^{\data,\bullet}_{\param}\Vert \mu^{\data,\marg}_{\param})=& d_{\KL}(\mu^{\data,\bullet}_{\param,\modelerror}\Vert \mu^{\data,\joint}_{\param,\modelerror})-\int_{\paramspace} d_{\KL}(\mu_\modelerror\Vert \mu^{\data,\joint}_{\modelerror|\param})\rd\mu^{\data,\bullet}_\param.
 \end{align*}
\end{proposition}
\begin{proof}
We use the chain rule for the KL divergence, see e.g. \cite[Exercise 3.2]{Wainwright2019}: 
 \begin{align*}
   d_{\KL }(\mu^{\data,\joint}_{\param,\modelerror}\Vert  \mu^{\data,\bullet}_{\param,\modelerror})=& d_{\KL}(\mu^{\data,\marg}_{\param}\Vert \mu^{\data,\bullet}_{\param})+ \int_{\paramspace} d_{\KL}(\mu^{\data,\joint}_{\modelerror |\param}\Vert \mu^{\data,\bullet}_{\modelerror|\param})\rd \mu^{\data,\marg}_{\param},
   \\
   d_{\KL }( \mu^{\data,\bullet}_{\param,\modelerror} \Vert \mu^{\data,\joint}_{\param,\modelerror})=& d_{\KL}(\mu^{\data,\bullet}_{\param} \Vert \mu^{\data,\marg}_{\param})+ \int_{\paramspace} d_{\KL}( \mu^{\data,\bullet}_{\modelerror|\param}\Vert \mu^{\data,\joint}_{\modelerror |\param})\rd \mu^{\data,\bullet}_{\param}.
 \end{align*}
 Recall that by \Cref{lem_properties_of_lifted_placeholder_objects}, $\mu^{\data,\bullet}_{\param,\modelerror}=\mu^{\data,\bullet}_{\param}\otimes \mu_\modelerror$, for $\bullet\in\{\appr,\best,\enh\}$.
 This implies that $ \mu^{\data,\bullet}_{\modelerror|\param}=\mu_\modelerror$ for $\bullet\in\{\appr,\best,\enh\}$.
\end{proof}

\Cref{prop_KL_divergence_marginal_posterior} implies that the KL divergence of the marginal posterior $\mu^{\data,\marg}_{\param}$ with respect to $\mu^{\data,\bullet}_{\param}$ for $\bullet\in\{\appr,\best,\enh\}$ is smaller than the KL divergence of the joint posterior $\mu^{\data,\joint}_{\param,\modelerror}$ and the corresponding lifted posterior $\mu^{\data,\bullet}_{\param,\modelerror}$.
In other words, marginalisation can only reduce the KL divergence.
As a result, the sufficient conditions for $\mu^{\data,\joint}_{\param,\modelerror}$ to agree with $\mu^{\data,\bullet}_{\param,\modelerror}$ are also sufficient conditions for $\mu^{\data,\marg}_{\param}$ to agree with $\mu^{\data,\bullet}_{\param}$, for $\bullet\in\{\appr,\best,\enh\}$.

\section{Example}
\label{sec_example}

In this section, we consider an inverse problem based on a PDE initial boundary value problem (PDE-IBVP), and use it to illustrate the positive criterion from \Cref{ssec_KL_bound_approximate_posterior} for choosing an observation operator that eliminates the effect of model error on inference.
In other words, combining this observation operator with the approximate model $\model$ yields exactly the same inference as if one had used the best model $\model^\best$ instead of $\model$.
This PDE-IBVP was discussed from the perspective of classical optimal experimental design in \cite{Alexanderian2014,Alexanderian2021}, for example.

Below, we use `$w$' as a placeholder symbol whose value may change from line to line.

\subsection{Advection-diffusion initial boundary value problem}
\label{ssec_advection_diffusion_IBVP_example}

Let $\mcal{T}\coloneqq (0,T)$ be a time interval and $\mcal{D}$ be a bounded open domain in $\bb{R}^{2}$.
Fix a constant $\kappa>0$, a vector field $v:\mcal{D}\to\bb{R}^{2}$, nonzero $s:\mcal{D}\to\bb{R}$ and nonnegative $b^{\best}:\mcal{D}\to\bb{R}$.
Define the parabolic operator $\mcal{L}\coloneqq -\kappa \Delta +v\cdot \nabla $.
For an arbitrary, sufficiently regular function $\param' :\mcal{T}\to\bb{R}$, consider the PDE-IBVP with homogeneous Neumann boundary condition and inhomogeneous initial condition
\begin{subequations}
\label{eq_PDE_IBVP_true}
 \begin{align}
    (\tderiv+\mcal{L})w(x,t)=&s(x)\param'(t), & &(x,t)\in \mcal{D}\times\mcal{T}, 
    \label{eq_PDE_true}
    \\
    \nabla w(x,t)\cdot n(x)=& 0, & &(x,t)\in\partial\mcal{D}\times\closure{\mcal{T}},
    \label{eq_BC}
    \\
    w(x,0)=& b^{\best}(x), & &(x,0)\in \mcal{D}\times \{0\},
    \label{eq_IC_true}
\end{align}
\end{subequations}
where $n(x)$ in \eqref{eq_BC} denotes the unit exterior normal to $\mcal{D}$ at $x \in\partial\mcal{D}$.
We assume that a unique solution $w=w(\param')$ of \eqref{eq_PDE_IBVP_true} exists, and define the best model $\model^{\best}$ by $\model^{\best}(\param')\coloneqq w(\param')$.
For a fixed $\param^{\best}:\mcal{T}\to\bb{R}$, the true state is the function $\state^{\best}\coloneqq \model^{\best}(\param^{\best}):\mcal{D}\times\mcal{T}\to\bb{R}$.
For the theory of existence and regularity of solutions of parabolic PDE-IBVPs with Neumann boundary conditions, see e.g. \cite[Part 2, Section 10]{Friedman1969}.

We define the approximate model $\model$ in a similar way as we defined $\model^{\best}$.
For the same coefficients and $\param'$ as above, we consider the IBVP with homogeneous Neumann boundary condition and homogeneous initial condition
\begin{subequations}
\label{eq_PDE_IBVP_zero_IC}
 \begin{align}
    (\tderiv+\mcal{L})w(x,t)=&s(x)\param'(t), & & (x,t)\in \mcal{D}\times\mcal{T}, 
\label{eq_PDE_zero_IC}
    \\
    \nabla w(x,t)\cdot n(x)=& 0, & & (x,t)\in\partial\mcal{D}\times\closure{\mcal{T}},
    \\
    w(x,0)=& 0, & &(x,0)\in \mcal{D}\times \{0\}.
    \label{eq_IC}
\end{align}
\end{subequations}
Note that \eqref{eq_PDE_IBVP_true} and \eqref{eq_PDE_IBVP_zero_IC} differ only in their respective initial conditions.
We define $\model(\param')\coloneqq w(\param')$, for $w(\param')$ the unique solution of the IBVP \eqref{eq_PDE_IBVP_zero_IC}. 
By homogeneity of the initial condition and the boundary condition, and using the superposition principle for linear PDEs, it follows that the approximate model $\model$ is a linear function of its argument $\param'$. 

Recall from \eqref{eq_model_error_definition} that $\modelerror^{\best}\coloneqq \model^{\best}-\model$. 
Then $\modelerror^{\best}$ solves the PDE-IBVP 
\begin{subequations}
 \label{eq_PDE_IBVP_modelerror}
 \begin{align}
   (\tderiv+\mcal{L})w(x,t)=&0, &  & (x,t)\in \mcal{D}\times\mcal{T} ,
    \label{eq_PDE_delta}
    \\
    \nabla w(x,t) \cdot n(x)=& 0, & & (x,t)\in\partial\mcal{D}\times\closure{\mcal{T}},
    \\
    w(x,0)=& b^{\best}(x), & & (x,0)\in \mcal{D}\times \{0\}.
    \label{eq_PDE_IBVP_modelerror_IC}
\end{align}
\end{subequations}
Since $w$ does not depend on $\param'$, it follows that $\modelerror^{\best}$ is constant with respect to $\param'$.
Rewriting \eqref{eq_model_error_definition} as $\model^{\best}=\model+\modelerror^{\best}$, it follows that $\model^{\best}(\param^\best)$ is an affine, non-linear function of $\param'$.

Recall the decomposition $\modelerror^\best=\model^\best-\model$ in \eqref{eq_model_error_definition}.
Suppose we keep the definition of the approximate model $\model$ as the mapping that sends $\param'$ to the solution of \eqref{eq_PDE_IBVP_zero_IC}, after replacing the initial condition \eqref{eq_IC} with $w(x,0)=b(x)$ for some nonnegative, nonzero $b$.
If $b^\best-b$ is small, then one expects the corresponding model error $\modelerror^\best=\model^\best-\model$ to also be small. However, $\modelerror^\best$ is now the solution of the PDE-IBVP \eqref{eq_PDE_IBVP_modelerror}, after replacing the initial condition \eqref{eq_PDE_IBVP_modelerror_IC} with $w(x,0)=b^{\best}(x)-b(x)$. 
If $b^\best-b$ assumes negative values on $\mcal{D}$, then $\modelerror^\best$ no longer describes the concentration of a substance. This may be undesirable, given that one uses advection-diffusion PDEs to describe the spatiotemporal evolution of concentrations.
Moreover, for differential equations that model chemical reactions, negative initial values can lead to finite-time blow-up, and small negative values in solutions can lead to instability, cf. \cite[Chapter I, Section 1.1]{Hundsdorfer2003}. Since we do not have a lower bound for $b^\best$ except for $b^\best\geq 0$, this explains the choice of the initial condition $b=0$ in the PDE-IBVP \eqref{eq_PDE_IBVP_zero_IC} for defining $\model$ and $\modelerror^\best$.

\subsection{Agreement of approximate and best posteriors}
\label{ssec_agreement_with_best_posterior}

Let $\paramspace$ be $C^\infty(\closure{\mcal{T}},\bb{R}_{>0})$, equipped with the supremum norm $\Norm{\cdot}_{L^\infty}$.
Assume that the vector field $v$ in the definition of $\mcal{L}$ satisfies $v\in C^\infty(\mcal{D},\bb{R}^{2})\cap L^\infty(\mcal{D},\bb{R}^{2})$, and that both the spatial component $s$ on the right-hand side of \eqref{eq_PDE_true} and the true initial condition $b^{\best}$ in \eqref{eq_IC_true} belong to $C^\infty(\closure{\mcal{D}},\bb{R}_{>0})$.
That is, all the terms that define the IBVPs are smooth and bounded on their respective domains, with strictly positive right-hand side $s(\cdot)\theta(\cdot)$ in \eqref{eq_PDE_true} and strictly positive initial condition $b^{\best}(\cdot)$ in \eqref{eq_IC_true}.
We define the state space $\statespace$ to be 
\begin{equation*}
\statespace\coloneqq  C^{2,1}(\mcal{D}\times\mcal{T})\cap\{w\text{ satisfies \eqref{eq_BC}}\},
\end{equation*}
where for $k,j\in\bb{N}_0$, $C^{k,j}(\mcal{D}\times\mcal{T})\coloneqq \{w \ \vert\  \Norm{w}_{C^{k,j}(\mcal{D}\times\mcal{T})}<\infty\}$, and 
\begin{equation*}
 \Norm{ w}_{C^{k,j}(\mcal{D}\times\mcal{T})}\coloneqq \Norm{w}_{L^\infty}+\sum_{0< \Abs{\beta}\leq k}\Norm{\partial_x^\beta w}_{L^\infty}+\sum_{0< \alpha \leq j}\Norm{\tderiv^{\alpha} w}_{L^\infty},
\end{equation*}
$\partial_x^\beta $ denotes the spatial derivative associated to the multi-index $\beta \in \bb{N}_{0}^{2}$, and $\tderiv^\alpha$ denotes the $\alpha$-th order time derivative.
The first and second sums on the right-hand side are defined to be zero if $k=0$ and $j=0$ respectively.

Given that $\kappa>0$ is constant and $v:\mcal{D}\to\bb{R}^{2}$ is bounded, the operator $\tderiv+\mcal{L}$ satisfies 
\begin{equation*}
 \Norm{(\tderiv+\mcal{L})w}_{C^{0,0}(\mcal{D}\times\mcal{T})}\leq C\Norm{w}_{C^{2,1}(\mcal{D}\times\mcal{T})},\quad w\in\statespace,
\end{equation*}
for $C=C(\kappa,v)\in\bb{R}_{>0}$ independent of $w$. 
Thus, for arbitrary $(x^{(i)},t_i)\in \mcal{D}\times\mcal{T}$ for $i\in[n]$, the linear observation operator
\begin{equation}
 \label{eq_observation_operator_PDE_example}
 \obsoperator :\statespace\to\bb{R}^{n},\quad w\mapsto \obsoperator w\coloneqq ((\tderiv+\mcal{L})w(x^{(i)},t_i))_{i\in[n]}
\end{equation}
is continuous, since $\Norm{\obsoperator w}_{\bb{R}^{n}}\leq n \Norm{(\tderiv+\mcal{L})w}_{C^{0,0}(\mcal{D}\times\mcal{T})}$.
Moreover, for $\obsoperator$ given by \eqref{eq_observation_operator_PDE_example}, it follows from \eqref{eq_PDE_delta} that $\{w\ \vert\ w\text{ satisfies \eqref{eq_PDE_delta}}\}\subset \kernel(\obsoperator)$, since $\obsoperator$ is the composition of the vectorisation of the evaluation functionals at $(x^{(i)},t_i)_{i\in[n]}$ with the differential operator on the left-hand side of \eqref{eq_PDE_delta}.
Thus, given the definition of $\obsoperator$ above and the definition of $\modelerror^{\best}$ in \Cref{ssec_advection_diffusion_IBVP_example} as the solution of the IBVP defined by \eqref{eq_PDE_IBVP_modelerror}, we have
\begin{equation*}
 \obsoperator   \modelerror^{\best}(\param')=0,\quad \forall\param'\in\paramspace.
\end{equation*}
In particular, the right-hand side of \eqref{eq_L1error_best_misfit_approx_misfit} in \Cref{lem_L1error_best_misfit_approx_misfit} vanishes for this choice of $\obsoperator$.
By \Cref{prop_KL_divergence_best_posterior_approx_posterior}, it follows that $\mu^{\data,\appr}_{\param}=\mu^{\data,\best}_{\param}$, i.e. the approximate posterior and the best posterior agree.

In the discussion after \Cref{prop_KL_divergence_best_posterior_enhanced_noise_posterior}, we showed that if the image of $\paramspace$ under $\modelerror^{\best}$ is contained in some linear subspace $V$ of $\statespace$, then it is simpler and easier to use the approximate posterior $\mu^{\data,\appr}_{\param}$ with an appropriately chosen $\obsoperator$ instead of the enhanced noise posterior $\mu^{\data,\enh}_{\param}$. 
Since $\modelerror^{\best}$ takes values in the linear subspace $\{w\ \vert\ w\text{ satisfies \eqref{eq_PDE_delta}}\}$ of $\statespace$, the same conclusion holds for this example, with the appropriate choice of $\obsoperator$ given in \eqref{eq_observation_operator_PDE_example}.
Similarly, by the discussion after \Cref{prop_KL_divergence_best_posterior_joint_posterior}, it is simpler and easier to use the approximate posterior with $\obsoperator$ from \eqref{eq_observation_operator_PDE_example} instead of the joint posterior.

In this example, the only information about the best model error $\modelerror^{\best}$ that we used to find an observation operator $\obsoperator$ that satisfies $\mu_{\param}(\obsoperator   \modelerror^{\best}=0)=1$ is contained in the PDE \eqref{eq_PDE_delta}.
We exploited this information to choose the observation operator in \eqref{eq_observation_operator_PDE_example}.
In the absence of additional conditions on $\modelerror^{\best}$, an observation operator that does not satisfy \eqref{eq_observation_operator_PDE_example} may fail to satisfy the condition $\mu_{\param}(\obsoperator   \modelerror^{\best}=0)=1$.
 Thus, we can consider the collection of observation operators with the structure \eqref{eq_observation_operator_PDE_example} as being the largest possible collection, given the available information.

We end this section with some important remarks. 

Suppose $w$ solves \eqref{eq_PDE_IBVP_true} for a given $\param'$.
Then the operator $\obsoperator$ in \eqref{eq_observation_operator_PDE_example} satisfies $\obsoperator w=(s(x^{(i)})\param'(t_i))_{i\in[n]}$.
In particular, if the spatial function $s$ is known and nonzero, then one can obtain the numbers $(\param'(t_i))_{i\in[n]}$ from the vector $\obsoperator w\in \bb{R}^{n}$ directly, by division.
We will return to this observation in \Cref{ssec_behaviour_of_best_posterior_mean_in_small_noise_limit}.

For the operator $\obsoperator$ in \eqref{eq_observation_operator_PDE_example}, exact evaluation of the vector $\obsoperator w$ for any $w\in \statespace$ in practical experimental settings can be difficult. 
For example, in practical settings, one might only be able to measure the concentration of a certain substance at fixed spatial locations $(x^{(j)})_{j\in [J]}\subset\mcal{D}$ and at multiples $k\Delta t \in\mcal{T} $ of some time lag $\Delta t>0$, $k\in [K]$, $K\in\bb{N}$. The corresponding observation operator maps every state $w\in \statespace $ to a vector $(w(x^{(j)},k \Delta t))_{j \in [J], k\in[K]} \in\bb{R}^{JK}$ of concentrations.
We consider one such operator in \eqref{eq_basic_observation_operator} below. 
In this case, and in the absence of further conditions on the spatiotemporal locations $(x^{(j)},k \Delta t)_{j\in[J],k\in[K]}$, there is no guarantee that one can approximate the operator $\obsoperator$ from \eqref{eq_observation_operator_PDE_example}.
However, if the spatiotemporal locations lie on a finite difference stencil, then one could potentially combine the information in the vector $(w(x^{(j)},k \Delta t))_{j \in [J], k\in[K]} \in\bb{R}^{JK}$ to approximate $\obsoperator$ from \eqref{eq_observation_operator_PDE_example}. 
Such approximations may exhibit errors associated to the finite difference scheme, and may also be susceptible to noise in the vector $(w(x^{(j)},k \Delta t))_{j \in [J], k\in[K]} \in\bb{R}^{JK}$.
We postpone the design and analysis of approximations of $\obsoperator$ in \eqref{eq_observation_operator_PDE_example} for future work.

In order to identify $\obsoperator$ in \eqref{eq_observation_operator_PDE_example} as an observation operator that satisfies $\obsoperator \modelerror^\best=0$, we exploited the fact that the right-hand side of \eqref{eq_PDE_delta} vanishes. 
Suppose that we modified the right-hand sides of \eqref{eq_PDE_zero_IC} and \eqref{eq_PDE_delta} to be nonzero functions $h,h^\best:\mcal{D}\times\mcal{T}\to\bb{R}$ respectively, where for some fixed $0<\lambda<1$ and any $\param'\in\paramspace$, $h(x,t)\coloneqq \lambda s(x)\param'(t)$ and $h^\best\coloneqq (1-\lambda) s(x)\param'(t)$, for every $(x,t)\in\mcal{D}\times\mcal{T}$.
If we do not know the initial condition $b^\best$ in \eqref{eq_PDE_IBVP_modelerror_IC}, then the only remaining information that we can use to find a linear operator $\obsoperator$ such that $\obsoperator \modelerror^\best=0$ is captured by the homogeneous Neumann boundary conditions in \eqref{eq_PDE_IBVP_modelerror}. 
Since the PDE-IBVP \eqref{eq_PDE_IBVP_true} that defines $\model^\best$ uses the same boundary conditions, any observation operator based on only the Neumann boundary conditions will imply that $\obsoperator\modelerror^\best=\obsoperator \model^\best$ as maps on $\paramspace$.
Thus, if $\obsoperator\modelerror^\best=0$, then also $\obsoperator \model^\best=0$. This example illustrates that it may not always be possible to find an observation operator $\obsoperator$ that satisfies $\obsoperator \modelerror^\best=0$ and that yields informative observations.

\section{Numerical simulations}
\label{sec_numerical_simulations}

In this section, we describe some numerical simulations of the PDE-IBVP from \Cref{ssec_advection_diffusion_IBVP_example} for a specific vector field $v$. We use these simulations to illustrate the behaviour of the best posterior $\mu^{\data,\best}_{\param}$ and the approximate posterior $\mu^{\data,\appr}_{\param}$ that were defined in \eqref{eq_best_misfit_and_posterior} and \eqref{eq_approx_misfit_and_posterior} respectively.
The code for these simulations is available at \url{https://gitlab.tue.nl/hbansal/model-error-study}.

\subsection{Setup of PDE-IBVP and Bayesian inverse problem}

\subsubsection{Setup of PDE-IBVP}
\label{ssec_setup_PDE_IBVP}

We set the spatial domain $\mcal{D}$ in \eqref{eq_PDE_IBVP_true} to be $\mcal{D}=[0,1]^2\setminus(C_1\cup C_2)$, where $C_1$ and $C_2$ represent two obstacles.
See \Cref{fig_sensors_source_obstacles}.
For the differential operator $\mcal{L}=-\kappa \Delta+v\cdot\nabla$ in \Cref{ssec_advection_diffusion_IBVP_example}, we use a constant diffusivity $\kappa=0.05$.
For the vector field $v$ we use the solution to the stationary Navier--Stokes equation 
\begin{equation}
 \label{eq_vector_field}
 \begin{aligned}
  -\frac{1}{\textup{Re}}\Delta v+\nabla q+ v\cdot \nabla v=&0 \quad \text{in }\mcal{D}
  \\
  \nabla\cdot v =&0 \quad \text{in }\mcal{D}
  \\
  v=& g \quad \text{on }\partial\mcal{D},
 \end{aligned}
\end{equation}
for a given pressure field $q$, Reynolds number $\textup{Re}=50$, and boundary data 
\begin{equation*}
 g:\partial \mcal{D}\to \bb{R}^{2},\quad g\tderiv(x_1,x_2)=
 \begin{cases}
  (0,1) & x_1=0,\ 0<x_2<1
  \\
  (0,-1) & x_1=1,\ 0<x_2<1 
  \\
  (0,0) & \text{otherwise},
 \end{cases}
\end{equation*}
see \cite[Section 5]{Alexanderian2014}.
In \Cref{fig_vector_field}, we plot a numerical approximation $v$ of the solution to \eqref{eq_vector_field}.
The numerical approximation satisfies $\Norm{v}_{L^\infty}=1$.
The boundary data ensure that the inner product of the vector field $v$ with the exterior normal vanishes everywhere on the boundary $\partial\mcal{D}$.
Given the Neumann boundary condition \eqref{eq_BC}, one can use this fact to prove the coercivity of the bilinear form associated to \eqref{eq_PDE_true}; see e.g. \cite[Section 3.2]{Elman2014} for the proof of this statement for the time-independent version of \eqref{eq_PDE_true}--\eqref{eq_IC_true}.
The coercivity of the bilinear form ensures the existence of a weak solution to all the PDE-IBVPs from \Cref{ssec_advection_diffusion_IBVP_example}.
Given $\kappa=0.05$, the global P\'{e}clet number --- see e.g. \cite[Section 13.2, Eq. (13.19)]{Quarteroni2017} --- associated to the PDE-IBVP has the value $\Norm{v}_{L^\infty(\mcal{D})} /(2\kappa)=10$.
\begin{figure}
    \centering
    \begin{subfigure}[b]{0.475\textwidth}
        \includegraphics[width=\textwidth]{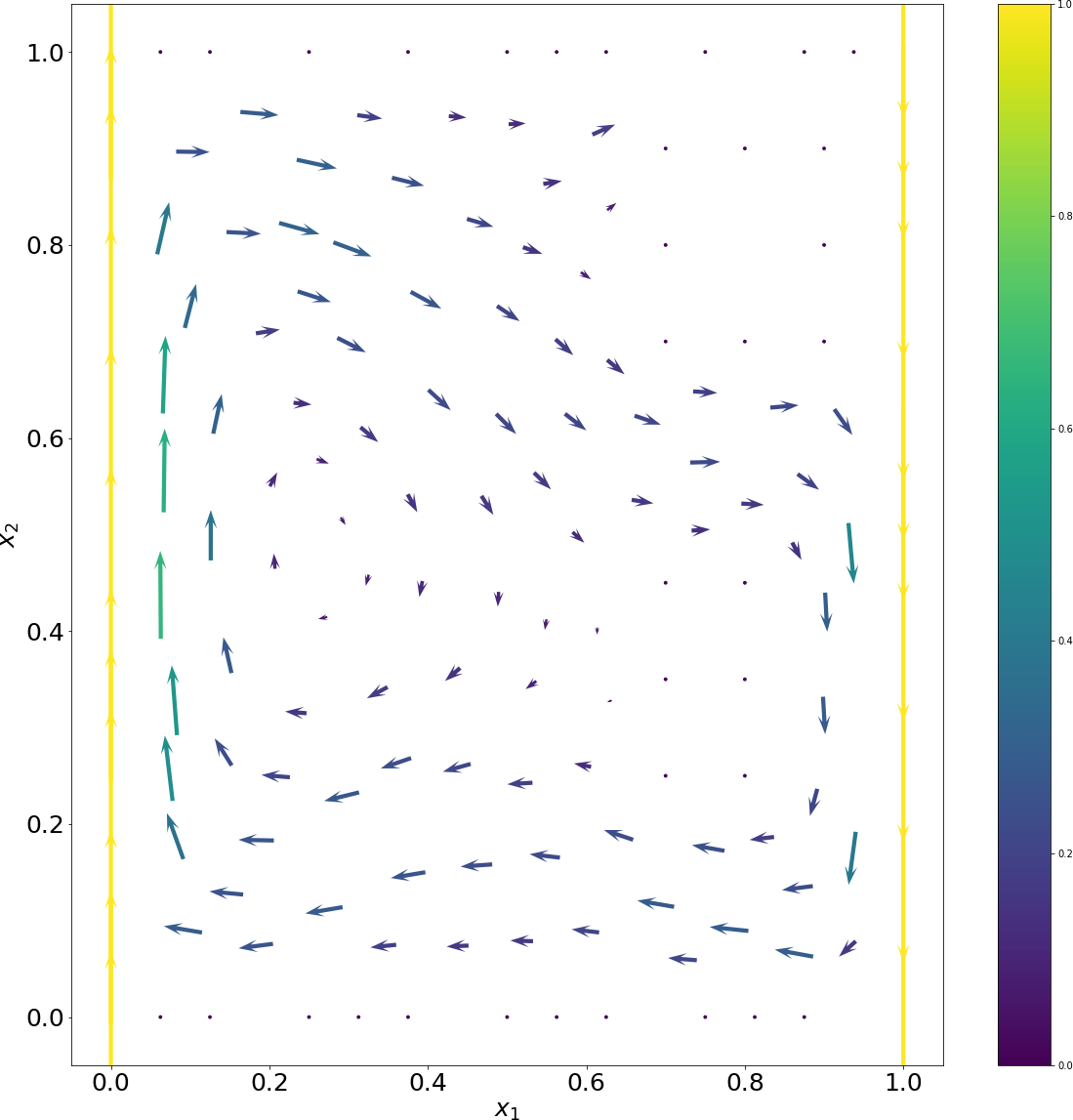}
        \caption{Vector field $v$ in \eqref{eq_PDE_true} and \eqref{eq_vector_field}. 
        \\
        Colour bar shows value of $\Norm{v(x_1,x_2)}_{2}\in [0,1]$.}
        \label{fig_vector_field}
    \end{subfigure}
    \hfill
    \begin{subfigure}[b]{0.475\textwidth}
        \includegraphics[width=\textwidth]{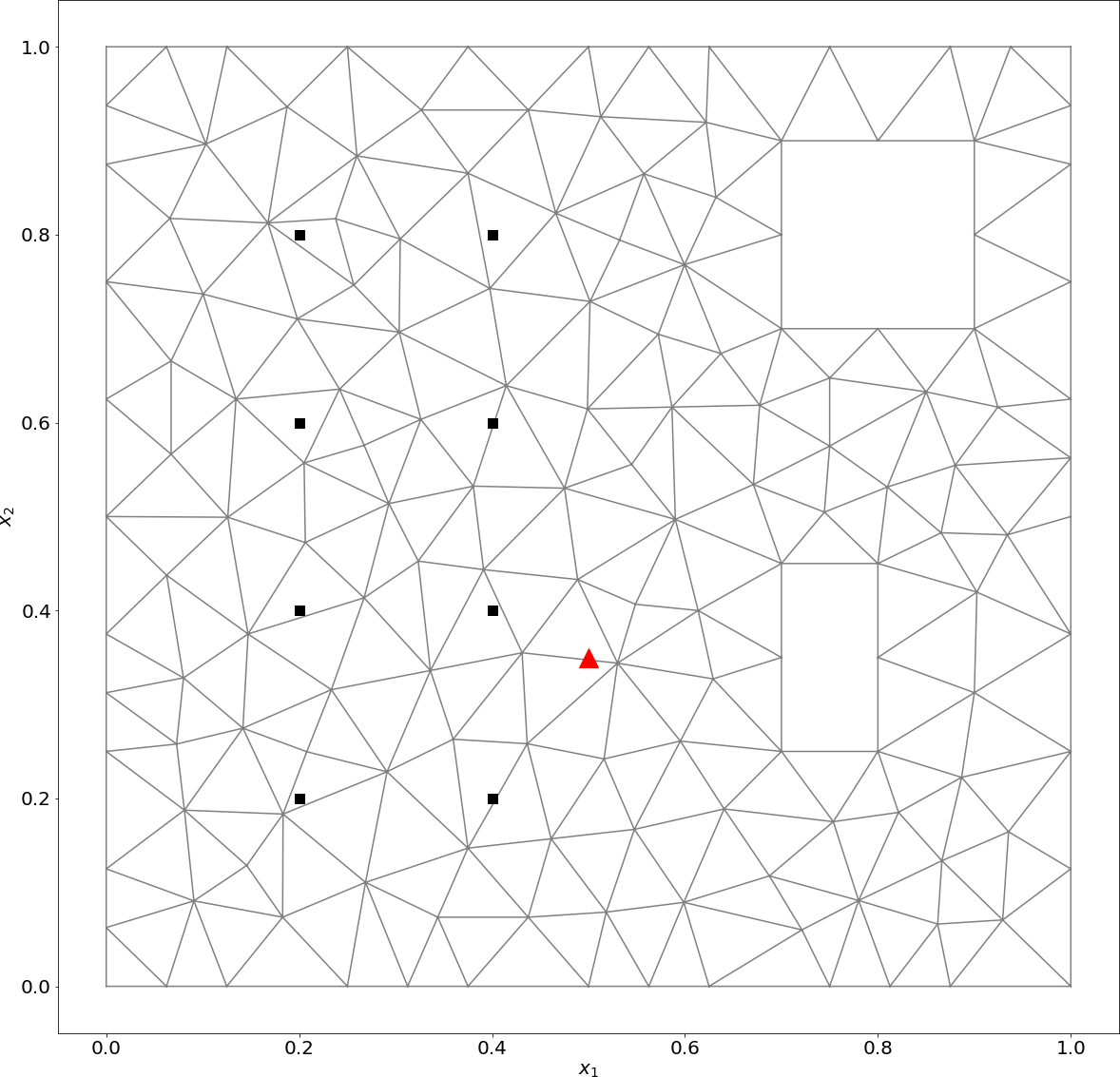}
        \caption{Location $x^{(0)}$ of source $s$ in \eqref{eq_PDE_true} and \eqref{eq_source_term_spatial_component} (\textcolor{red}{$\blacktriangle$}), sensor locations ($\blacksquare$), and obstacles.}
        \label{fig_sensors_source_obstacles}
    \end{subfigure}
    \caption{Plots of vector field, sensors, source term, and obstacles.}
    \label{fig_vector_field_sensors_source_obstacles}
\end{figure}

Next, we define 
\begin{subequations}
\begin{align}
  s(x)\coloneqq& \frac{1}{2\pi L}\exp\left(-\frac{1}{2L}\Abs{x-x^{(0)}}^2\right),
  \label{eq_source_term_spatial_component}
  \\
 \param^{\best}(t)\coloneqq &
 \begin{cases}
 80 & t\in [0,0.2]
 \\
 80+20\exp\left(1-\left(1-16(t-0.45)^2\right)^{-1}\right) & t\in (0.2,0.45]
 \\
 50+50\exp\left(1-\left(1-16(t-0.45)^2\right)^{-1}\right) & t\in (0.45,0.7]
 \\
 50 & t\in (0.7,1.0],
 \end{cases}
\label{eq_source_term_temporal_component}
\end{align}  
 \end{subequations}
and set $L=0.05$ and $x^{(0)}=(0.5,0.35)$.
We will use these functions as reference values for the right-hand side of \eqref{eq_PDE_true}.
We indicate the location parameter $x^{(0)}$ of the spatial function $s$ by a red triangle in \Cref{fig_sensors_source_obstacles} and plot a discrete approximation $\widetilde{\param}^{\best}$ of $\param^{\best}$ in \Cref{fig_BestPosteriorMeans_HighSNR}.

To complete the specification of the PDE-IBVP \eqref{eq_PDE_IBVP_true}, it remains to state the reference initial condition $b^{\best}$ in \eqref{eq_IC_true}.
To this end, we used a sample from the Gaussian distribution with constant mean $m_b\equiv 50$ on $\mcal{D}$ and covariance operator $(-\epsilon\Delta +\alpha I)^{-2}$ for $\epsilon=4.5\times 10^{-3}$ and $\alpha=2.2\times 10^{-1}$ equipped with Robin boundary conditions.
This Gaussian distribution and the justification for its use are given in \cite[Section 5.2]{Alexanderian2021}.
In \Cref{fig_evolutions_for_two_random_initial_conditions}, we show snapshots of the solution $w$ of \eqref{eq_PDE_true}--\eqref{eq_IC_true} at times $t\in\{0,0.2,0.3,0.7,1.0\}$, for two samples from the Gaussian distribution of $b^{\best}$.
For the reference initial condition $b^{\best}$, we used the initial condition shown in the bottom left subfigure of \Cref{fig_evolutions_for_two_random_initial_conditions}.
\begin{figure}
    \centering
    \includegraphics[width=\textwidth]{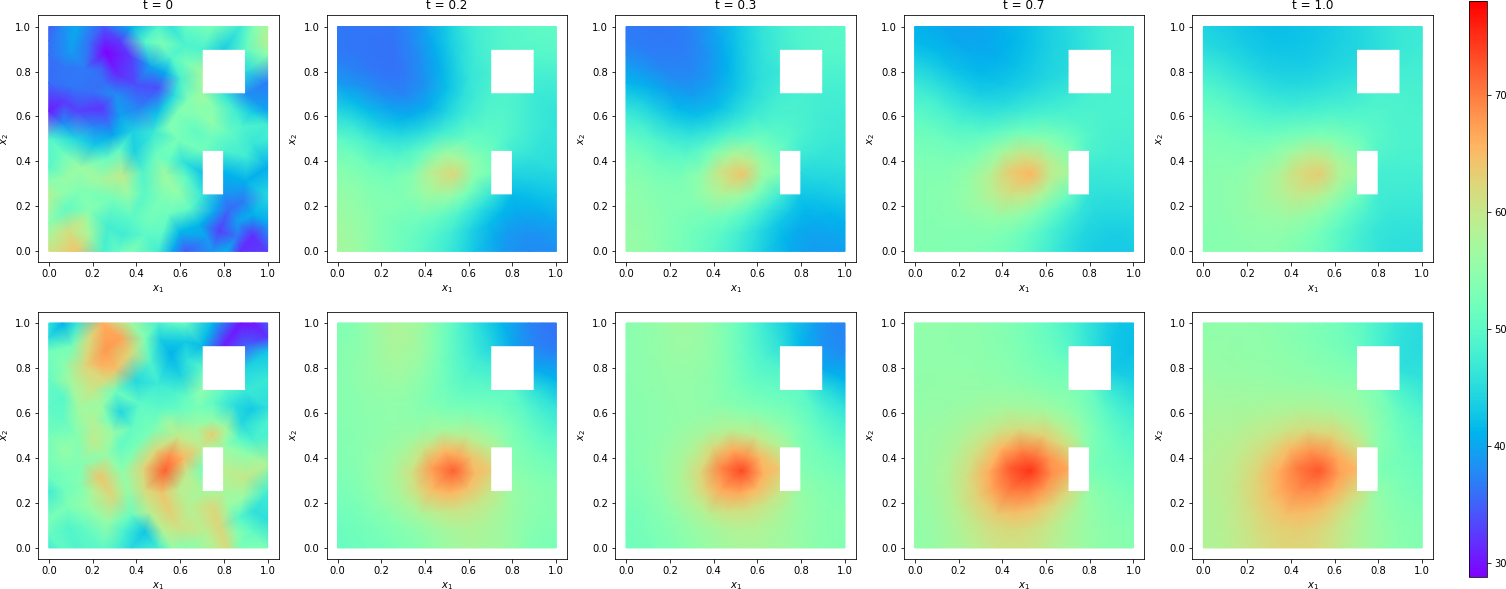}
        \caption{Snapshots of the solution $w$ of \eqref{eq_PDE_true}--\eqref{eq_IC_true} at $t\in\{0,0.2,0.3,0.7,1.0\}$, for two random initial conditions $b^\best$.
        The reference initial condition $b^{\best}$ that was used for generating observations is shown in the bottom left subfigure.}
        \label{fig_evolutions_for_two_random_initial_conditions}
\end{figure}

\subsubsection{Setup of Bayesian inverse problem}
\label{ssec_setup_Bayesian_inverse_problem}

Consider two different observation operators
\begin{subequations}
\label{eq_observation_operators_for_numerical_results}
 \begin{align}
  \label{eq_basic_observation_operator}
  \obsoperator_{B}:\statespace &\to  \bb{R}^{n},\quad \state \mapsto (\state (x^{(i)},t^{(i)}))_{i=1}^{n}
  \\
  \label{eq_differential_observation_operator}
 \obsoperator_{D}:\statespace & \to\bb{R}^{n},\quad \state \mapsto ((\tderiv \state (x^{(i)},t^{(i)})+\mcal{L}\state (x^{(i)},t^{(i)}))_{i=1}^{n},
 \end{align}
\end{subequations}
for $\mcal{L}$ given in \Cref{ssec_advection_diffusion_IBVP_example}.
We refer to $\obsoperator_B$ and $\obsoperator_D$ as the `basic observation operator' and `PDE observation operator' respectively.
To ensure that the numerical results do not depend on the choice of spatiotemporal locations at which observations are collected, we define $\obsoperator_B$ and $\obsoperator_D$ using the same spatiotemporal locations $(x^{(i)},t_i)_{i=1}^{n}$ belonging to the set
 \begin{equation}
 \label{eq_spatiotemporal_locations_of_observations}
  \left\{(x_1,x_2)=(0.2m,0.2n)\in \mcal{D}\ :\ m\in[2],\ n\in [4]\right\}\times \left\{t=0.1+ 0.01\ell\in\mcal{T}\ :\ \ell\in[80]_0\right\}.
 \end{equation}
Thus, for both $\obsoperator_B$ and $\obsoperator_D$, observations are vectors with $n=8\times 81=648$ entries.
In \Cref{fig_sensors_source_obstacles}, we indicate the spatial locations $(x^{(i)})_{i}$ by black squares.
We emphasise that, unlike `classical' experimental design, our goal is not to describe how to choose the spatiotemporal locations of the measurements in order to maximise the information gain for each successive measurement.
Instead, our goal is to show that by choosing a different \emph{type} of measurement, one can mitigate the effect of model error on Bayesian inference.
In this context, the key difference between $\obsoperator_B$ and $\obsoperator_D$ is that $\obsoperator_D\modelerror^{\best}$ vanishes on $\mcal{D}\times\mcal{T}$ because of \eqref{eq_PDE_delta}, but $\obsoperator_B \modelerror^{\best}$ does not.

For the prior $\mu_\param$ on the unknown $\param^{\best}$, we use the same prior as in \cite[Section 5.2]{Alexanderian2021}, namely a Gaussian probability measure on  $L^2(\mcal{T})$ with constant mean $m_\param\equiv 65$ and covariance operator $\Sigma_\param :L^2(\mcal{T})\to L^2(\mcal{T})$ generated by a Mat\'{e}rn covariance kernel
\begin{equation*}
\Sigma_\param z(s)\coloneqq \int_{\mcal{T}} c(s,t')z(t')\rd t',\quad
 c(s,t)\coloneqq  \sigma^2 \left(1+\frac{\sqrt{3}\Abs{s-t}}{\ell}\right)\exp\left(-\frac{\sqrt{3}\Abs{s-t}}{\ell}\right),\quad s,t\in\mcal{T}
\end{equation*}
with $\sigma=80$ and length scale $\ell=0.17$.

Fix $\obsoperator\in\{\obsoperator_B,\obsoperator_D\}$ and a realisation $\data$ of the corresponding random variable $\dataRV$ in \eqref{eq_true_data_for_numerical_simulations}.
In \Cref{ssec_advection_diffusion_IBVP_example}, we observed that $\param'\mapsto \model(\param')$ is a linear mapping.
This implies that the approximate can posterior $\mu^{\data,\appr}_{\param}$ is Gaussian.
Given that we assume $\noise\sim\normaldist{0}{\Sigma_\noise}$ and $\param\sim\mu_\param=\normaldist{m_\param}{\Sigma_\param}$ to be statistically independent, it follows that the covariance of $\obsoperator\model\param+\noise$ and $\param$ is $\obsoperator\model\Sigma_\param$, and the covariance matrix of $\obsoperator\model\param+\noise$ is $(\obsoperator\model) \Sigma_\param (\obsoperator\model)^\ast+\Sigma_\noise$.
We then compute the mean and covariance of $\mu^{\data,\appr}_{\param}$ using the formula for conditioning of Gaussians, see e.g. \cite[Theorem 6.20]{Stuart2010}:
\begin{subequations}
 \label{eq_approximate_posterior_linear_Gaussian_case}
 \begin{align}
 m^{\data,\appr}_{\param}=&m_\param+(\obsoperator \model \Sigma_\param)^\ast\left((\obsoperator \model ) \Sigma_\param (\obsoperator \model)^\ast+\Sigma_\noise \right)^{-1}(\data-\obsoperator \model m_\param)
 \label{eq_approximate_posterior_mean_linear_Gaussian_case}
 \\
 \Sigma^{\data,\appr}_{\param}=&\Sigma_\param- (\obsoperator \model \Sigma_\param)^\ast \left((\obsoperator \model ) \Sigma_\param (\obsoperator \model )^\ast+\Sigma_\noise \right)^{-1}(\obsoperator \model \Sigma_\param).
 \label{eq_approximate_posterior_covariance_linear_Gaussian_case}
\end{align}
\end{subequations}
In \Cref{ssec_advection_diffusion_IBVP_example}, we observed that $\modelerror^{\best}$ is constant with respect to $\param'$, and that $\param'\mapsto \model^\best\param'=\model\param'+\modelerror^\best$ is affine.
Hence, the covariance of $\obsoperator\model^\best\param+\noise$ and $\param\sim \mu_\param$ is $\obsoperator\model\Sigma_\param$, the covariance matrix of $\obsoperator\model^\best\param+\noise$ is $(\obsoperator\model) \Sigma_\param (\obsoperator\model)^\ast+\Sigma_\noise$, and $\mu^{\data,\best}_{\param}$ is Gaussian, with mean and covariance 
\begin{subequations}
 \label{eq_best_posterior_linear_Gaussian_case}
 \begin{align}
 m^{\data,\best}_{\param}=&m_\param+(\obsoperator \model \Sigma_\param)^\ast\left((\obsoperator \model ) \Sigma_\param (\obsoperator \model)^\ast+\Sigma_\noise \right)^{-1}(\data-\obsoperator \model m_\param-\obsoperator \modelerror^{\best})
 \label{eq_best_posterior_mean_linear_Gaussian_case}
 \\
 \Sigma^{\data,\best}_{\param}=&\Sigma_\param- (\obsoperator \model \Sigma_\param)^\ast \left((\obsoperator \model ) \Sigma_\param (\obsoperator \model )^\ast+\Sigma_\noise \right)^{-1}(\obsoperator \model \Sigma_\param).
 \label{eq_best_posterior_covariance_linear_Gaussian_case}
\end{align}
\end{subequations}
By comparing \eqref{eq_approximate_posterior_linear_Gaussian_case} and \eqref{eq_best_posterior_linear_Gaussian_case}, we conclude that the posteriors $\mu^{\data,\appr}_{\param}$ and $\mu^{\data,\best}_{\param}$ differ only in their means, namely by the shift $(\obsoperator \model \Sigma_\param)^\ast\left((\obsoperator \model )^\ast \Sigma_\param (\obsoperator \model)+\Sigma_\noise \right)^{-1} \obsoperator\modelerror^{\best}$.
Thus, if one chooses $\obsoperator$ so that $\obsoperator\modelerror^{\best}=0$, then $\mu^{\data,\best}_{\param}$ and $\mu^{\data,\appr}_{\param}$ agree.
This result is consistent with \Cref{prop_KL_divergence_best_posterior_approx_posterior}.
It has the interpretation that if one uses the observation operator $\obsoperator_D$ to map states to data, then using the approximate model $\model$ to map parameters to states yields exactly the same inference as when one uses the best model $\model^\best$.

\subsection{Implementation}
\label{ssec_implementation}

To solve the IBVPs \eqref{eq_PDE_IBVP_true}, \eqref{eq_PDE_IBVP_zero_IC}, and \eqref{eq_PDE_IBVP_modelerror}, we implemented the space-time finite element method on an unstructured triangular mesh for the spatial domain $\mcal{D}$ and a uniform mesh for the temporal domain $\mcal{T}$, using the hIPPYlib \cite{VillaPetraGhattas21} and FEniCS/DOLFIN \cite{Logg2012} software.
We used Lagrange nodal basis functions to discretise functions. 
That is, given a finite element mesh of $\mcal{D}$ with nodes $(x_i)_{i\in[N_D]}$ and given a finite element mesh of $\mcal{T}$ with nodes $(t_j)_{j\in[N_T]}$, we have a spatial basis $(\varphi_i)_{i\in[N_D]}\subset L^2(\mcal{D})$ of functions and a temporal basis $(\psi_j)_{j\in[N_T]}\subset L^2(\mcal{T})$ such that $\varphi_i(x_k)=\delta_{ik}$ for $i,k\in [N_D]$ and $\psi_j(t_\ell)=\delta_{j\ell}$ for $j,\ell\in [N_T]$.
This yields a space-time basis $(d_k)_{k\in[N_DN_T]}$, where each $d_k=\varphi_{i(k)}\psi_{j(k)}$ belongs to $L^2(\mcal{D}\times\mcal{T})$.
We used piecewise linear basis functions for all functions except for the advection vector field $v$, where we used piecewise quadratic functions.

To compute the numerical approximation $\mathbf{w}\in\bb{R}^{N_DN_T}$ of the solution $w$ of the PDE-IBVP \eqref{eq_PDE_IBVP_true} for a given $\param'$ on the right-hand side of \eqref{eq_PDE_true} and a given $b^{\best}$ in \eqref{eq_IC_true}, we used the weak formulation of \eqref{eq_PDE_IBVP_true} to form a matrix-vector equation $\mathbf{A}\mathbf{\state}=\mathbf{f}$, where
\begin{equation}
 \label{eq_matrix_vector_for_weak_formulation}
 (\mathbf{A})_{k_1,k_2}\coloneqq \Ang{(\tderiv+\mcal{L})d_{k_1},d_{k_2}}_{L^2(\mcal{D}\times\mcal{T})},\quad (\mathbf{f})_{k_2}\coloneqq \Ang{s\param',d_{k_2}}_{L^2(\mcal{D}\times\mcal{T})},\quad k_1,k_2\in[N_DN_T].
\end{equation}
We used quadrature to approximate the $L^2(\mcal{D}\times\mcal{T})$ inner product.
For the initial condition $b^{\best}$, we discretised the covariance operator $(-\epsilon\Delta+\alpha I)^{-2}$ described in \Cref{ssec_setup_PDE_IBVP} and sampled from the resulting finite-dimensional Gaussian distribution on $\bb{R}^{N_D}$. 

For the spatial discretisation, 262 elements were used, yielding $N_D=159$ spatial degrees of freedom and a maximal spatial element diameter of $\Delta x\approx 0.13$.
For the temporal discretisation, 100 elements of uniform diameter $\Delta t=0.01$ were used, yielding $N_T=101$ temporal degrees of freedom.
Since $v$ in \eqref{eq_vector_field} satisfies $\Norm{v}_{L^\infty(\mcal{D})}=1$, the local P\'{e}clet number is $\Norm{v}_{L^\infty(\mcal{D})}\Delta x/(2\kappa)=1.3$; see e.g. \cite[Section 13.2, Equation (13.22)]{Quarteroni2017}.

The discretisations of $m^{\data,\appr}_{\param}$, $m^{\data,\best}_{\param}$, and $\Sigma^{\data,\appr}_{\param}=\Sigma^{\data,\best}_{\param}$ from \eqref{eq_approximate_posterior_linear_Gaussian_case} and \eqref{eq_best_posterior_linear_Gaussian_case} are
\begin{subequations}
\label{eq_discrete_representation_posteriors}
 \begin{align}
\mathbf{m}^{\data,\appr}_{\param} =& \mathbf{m}_{\param} + (\mathbf{OM\Sigma}_{\param})^\ast((\mathbf{OM})\mathbf{\Sigma}_{\param}(\mathbf{OM})^\ast + \Sigma_{\noise})^{-1}(\data-\mathbf{OMm}_{\param})    
\label{eq_discrete_representation_approximate_posterior_mean}
\\
\mathbf{m}^{\data,\best}_{\param} =& \mathbf{m}_{\param} + (\mathbf{OM\Sigma}_{\param})^\ast((\mathbf{OM})\mathbf{\Sigma}_{\param}(\mathbf{OM})^\ast + \Sigma_{\noise})^{-1}(\data-\mathbf{OMm}_{\param}-\mathbf{O}\widetilde{\modelerror}^{\best})    
  \label{eq_discrete_representation_best_posterior_mean}
\\
    \mathbf{\Sigma}^{\data,\appr}_{\param}=\mathbf{\Sigma}^{\data,\best}_{\param} =& \mathbf{\Sigma}_{\param} - (\mathbf{OM\Sigma}_{\param})^\ast((\mathbf{OM})\mathbf{\Sigma}_{\param}(\mathbf{OM})^\ast + \Sigma_{\noise})^{-1}\mathbf{OM\Sigma}_{\param}
    \nonumber
    \\
    =&\mathbf{\Sigma}_{\param}^{1/2}\left(I - (\mathbf{OM}\mathbf{\Sigma}_{\param}^{1/2})^\ast((\mathbf{OM})\mathbf{\Sigma}_{\param}(\mathbf{OM})^\ast + \Sigma_{\noise})^{-1}\mathbf{OM}\mathbf{\Sigma}_{\param}^{1/2}\right)\mathbf{\Sigma}_{\param}^{1/2},
    \label{eq_discrete_representation_approximate_posterior_covariance}
 \end{align}
\end{subequations}
where $\mathbf{\Sigma}_{\param}^{1/2}$ in \eqref{eq_discrete_representation_approximate_posterior_covariance} denotes the symmetric square root of $\mathbf{\Sigma}_{\param}$.
Since $m_\param\equiv 65$, we set $\mathbf{m}_{\param}=(65,\ldots,65)\in\bb{R}^{N_T}$.
For the prior covariance matrix $\mathbf{\Sigma}_\param \in \bb{R}^{N_T\times N_T}$, we set $(\mathbf{\Sigma}_\param)_{i,j}\coloneqq \int_{\mcal{T}} c(t_i,t')\psi_j(t')\rd t'$ and approximated the integral by quadrature.
We discretised the approximate model $\model$ as a matrix $\mathbf{M}\in \bb{R}^{N_DN_T\times N_T}$, where we computed the $j$-th column of $\mathbf{M}$ for $j\in[N_T]$ by replacing $\param'$ in the linear form associated to \eqref{eq_PDE_true} with the temporal basis function $\psi_j$ in the definition of the right-hand side $\mathbf{f}$; see \eqref{eq_matrix_vector_for_weak_formulation}.
We discretised the basic observation operator $\obsoperator_{B}$ in \eqref{eq_basic_observation_operator} and the PDE observation operator $\obsoperator_{D}$ in \eqref{eq_differential_observation_operator} as $n\times N_DN_T$ matrices $\mathbf{O}_{B}$ and $\mathbf{O}_D$ respectively, where $\mathbf{O}_{B}\mathbf{u}$ and $\mathbf{O}_D \mathbf{u}$ return the values of the piecewise linear spatiotemporal functions at the locations given in \eqref{eq_spatiotemporal_locations_of_observations} represented by $\mathbf{u}$ and $\mathbf{A\state}$ respectively.
To evaluate functions off the finite element nodes, we used linear interpolation.

It remains to state how we generated the data $y$ used in \eqref{eq_discrete_representation_approximate_posterior_mean} and \eqref{eq_discrete_representation_best_posterior_mean}.
Let $\widetilde{\param}^\best\in\bb{R}^{N_T}$ denote the discretisation of the truth $\param^{\best}$ using the temporal basis $(\psi_j)_j$.
We solved the matrix-vector equation $\mathbf{A}\widetilde{\modelerror}^{\best}=0$ that corresponds to \eqref{eq_PDE_IBVP_modelerror} for $\widetilde{\modelerror}^{\best}\in\bb{R}^{N_DN_T}$ and defined the discretisation of the true state to be $\mathbf{u}^{\best}\coloneqq \mathbf{M} \widetilde{\param}^{\best}+\widetilde{\modelerror}^{\best}$.
For each of the choices $\mathbf{O}=\mathbf{O}_B$ and $\mathbf{O}=\mathbf{O}_D$, we generated data by sampling the random variable 
\begin{equation}
 \label{eq_true_data_for_numerical_simulations}
 \dataRV=\mathbf{O} \mathbf{\state}^{\best}+\noise,\quad \noise \sim \normaldist{0}{\Sigma_\noise}
\end{equation}
for $\Sigma_\noise\coloneqq \sigma_\noise^2 I$, where we set the standard deviation $\sigma_\noise$ to be the product of a prescribed signal-to-noise ratio (SNR) with the median of the underlying signal $\mathbf{O}\mathbf{\state}^{\best}$.

For our experiments, we considered a lower SNR and a higher SNR, where we set $\sigma_\noise$ to be the median of the underlying signal, scaled by $\mfrak{s}=0.1$ and $\mfrak{s}=0.02$ respectively.
Our SNR regimes are reasonable, given that in \cite[Section 5.3]{Alexanderian2021} the value $\sigma_\noise=0.25$ is used for signals whose entries take values in the interval $[51,54]$, for a corresponding $\mfrak{s}$ that is at most $0.004$.

The minimum, median, and maximum of the signal $\mathbf{O}_B\mathbf{\state}^{\best}$ was 52.18, 56.63, and 68.53 respectively, which is consistent with the figures in the bottom row of \Cref{fig_evolutions_for_two_random_initial_conditions} and yields $\sigma_\noise\approx 5.66$ and $\sigma_\noise\approx 1.13$ for the lower SNR and higher SNR cases respectively.
In contrast, the minimum, median, and maximum of the signal $\mathbf{O}_D\mathbf{\state}^{\best}$ was 0, $2\times 10^{-6}$, and $3.86\times 10^{-3 }$ respectively.
This is consistent with the fact that $\mathbf{O}_D\mathbf{\state}^{\best}$ returns the function represented by the vector $\mathbf{f}$ in \eqref{eq_matrix_vector_for_weak_formulation}: the entries of $\mathbf{f}$ are given by integrals of the function $s(x)$ that decreases exponentially as $\Abs{x-x^{(0)}}^{2}$ increases, and the integration regions have space-time `volume' proportional to $(\Delta x)^2 \Delta t\approx 10^{-3}$. 
Thus, for $\mathbf{O}_D$, $\sigma_\noise \approx 2\times 10^{-7}$ and $\sigma_\noise \approx 4\times 10^{-8}$ in the lower SNR and higher SNR cases respectively.

\subsection{Behaviour of discretised best posterior mean in small noise limit}
\label{ssec_behaviour_of_best_posterior_mean_in_small_noise_limit}

In this section, we analyse the behaviour in the small noise limit of the discretised best posterior mean defined in \eqref{eq_discrete_representation_best_posterior_mean}.
By the preceding section, for $\mathbf{O}\in\{\mathbf{O}_B,\mathbf{O}_D\}$, $\mathbf{OM}$ is a $n\times N_T$ matrix with $n=648$ and $N_T=101$.
Hence, $\mathbf{OM}$ is not invertible, but it admits a unique Moore--Penrose pseudoinverse $(\mathbf{OM})^+$.
Multiplying both sides of \eqref{eq_discrete_representation_best_posterior_mean} by $(\mathbf{OM})^{+}(\mathbf{OM})$, using the symmetry of $\mathbf{\Sigma}_{\param}$, and adding zero, we obtain
\begin{align*}
&(\mathbf{OM})^{+}(\mathbf{OM}) \mathbf{m}^{\data,\best}_{\param}
\\
=&(\mathbf{OM})^{+}(\mathbf{OM})\left( \mathbf{m}_{\param} + (\mathbf{OM}\mathbf{\Sigma}_{\param})^\ast((\mathbf{OM})\mathbf{\Sigma}_{\param}(\mathbf{OM})^\ast + \Sigma_{\noise})^{-1}(\data-\mathbf{OMm}_{\param}-\mathbf{O}\widetilde{\modelerror}^{\best})\right)
\\
=&(\mathbf{OM})^{+}(\mathbf{OM})\mathbf{m}_{\param}
\\
&+(\mathbf{OM})^{+}(\mathbf{OM}) \mathbf{\Sigma}_{\param}(\mathbf{OM})^\ast((\mathbf{OM})\mathbf{\Sigma}_{\param}(\mathbf{OM})^\ast + \Sigma_{\noise})^{-1}(\data-\mathbf{OMm}_{\param}-\mathbf{O}\widetilde{\modelerror}^{\best})
\\
=&(\mathbf{OM})^{+}(\mathbf{OM})\mathbf{m}_{\param}
\\
&+(\mathbf{OM})^{+}\left(\Sigma_\noise+(\mathbf{OM}) \mathbf{\Sigma}_{\param}(\mathbf{OM})^\ast\right)\left((\mathbf{OM})\mathbf{\Sigma}_{\param}(\mathbf{OM})^\ast + \Sigma_{\noise}\right)^{-1}(\data-\mathbf{OMm}_{\param}-\mathbf{O}\widetilde{\modelerror}^{\best})
\\
&-(\mathbf{OM})^{+}\Sigma_\noise \left((\mathbf{OM})\mathbf{\Sigma}_{\param}(\mathbf{OM})^\ast + \Sigma_{\noise}\right)^{-1}(\data-\mathbf{OMm}_{\param}-\mathbf{O}\widetilde{\modelerror}^{\best})
\\
=& (\mathbf{OM})^{+}(\mathbf{OM})\mathbf{m}_{\param}+(\mathbf{OM})^{+}(\data-\mathbf{OMm}_{\param}-\mathbf{O}\widetilde{\modelerror}^{\best})
\\
&-(\mathbf{OM})^{+}\Sigma_\noise \left((\mathbf{OM})\mathbf{\Sigma}_{\param}(\mathbf{OM})^\ast + \Sigma_{\noise}\right)^{-1}(\data-\mathbf{OMm}_{\param}-\mathbf{O}\widetilde{\modelerror}^{\best}).
\end{align*}
By the definition \eqref{eq_true_data_for_numerical_simulations} of $\dataRV$, the definition $\mathbf{\state}^{\best}\coloneqq \mathbf{M}\widetilde{\param}^{\best}+\widetilde{\modelerror}^{\best}$, and the linearity of $\mathbf{O}$,
\begin{equation*}
 \dataRV-\mathbf{O}\mathbf{M}\mathbf{m}_\param-\mathbf{O}\widetilde{\modelerror}^{\best}=(\mathbf{O}\mathbf{M}\widetilde{\param}^{\best}+\mathbf{O}\widetilde{\modelerror}^{\best}+\noise)-\mathbf{O}\mathbf{M}\mathbf{m}_\param-\mathbf{O}\widetilde{\modelerror}^{\best}=\mathbf{O}\mathbf{M}(\widetilde{\param}^{\best}-\mathbf{m}_\param)+\noise,
\end{equation*}
so $(\data-\mathbf{OMm}_{\param}-\mathbf{O}\widetilde{\modelerror}^{\best})=\mathbf{OM}(\widetilde{\param}^{\best}-\mathbf{m}_\param)+\widetilde{\noise}$, where $\widetilde{\noise}$ denotes the specific realisation of $\noise$ that yields $\data$.
If $R\coloneqq (\mathbf{OM})^{+}\Sigma_\noise \left((\mathbf{OM})\mathbf{\Sigma}_{\param}(\mathbf{OM})^\ast + \Sigma_{\noise}\right)^{-1}(\data-\mathbf{OMm}_{\param}-\mathbf{O}\widetilde{\modelerror}^{\best})$ is negligible, then the preceding observations imply that
\begin{align*}
 (\mathbf{OM})^{+}(\mathbf{OM}) \mathbf{m}^{\data,\best}_{\param}\approx & (\mathbf{OM})^{+}(\mathbf{OM})\mathbf{m}_{\param}+(\mathbf{OM})^{+}(\mathbf{OM}(\widetilde{\param}^{\best}-\mathbf{m}_\param)+\widetilde{\noise})
 \\
 =&(\mathbf{OM})^{+}(\mathbf{OM})\widetilde{\param}^{\best}+(\mathbf{OM})^{+}\widetilde{\noise},
\end{align*}
where $\widetilde{\noise}$ will be negligible with increasing probability as $\Norm{\Sigma_\noise}_{\op}\to 0$.

By the properties of the pseudoinverse, $(\mathbf{OM})^+(\mathbf{OM})$ is the projection to the support of $\mathbf{OM}$, i.e. to the orthogonal complement in $\bb{R}^{N_T}$ of $\kernel(\mathbf{OM})$.
We discuss an important property of this projection for the case $\mathbf{O}\leftarrow\mathbf{O}_D$, since we will use this property in \Cref{ssec_results} below.
Let $\mathbf{v}\in\bb{R}^{N_T}$ be the vector of coefficients of an arbitrary $v$ in the span of the temporal basis functions $(\psi_j)_{j\in[N_T]}$; see \Cref{ssec_implementation}.
Recall that $\mathbf{Mv}$ is an approximation of the weak solution of \eqref{eq_PDE_IBVP_zero_IC} with $\param'\leftarrow v$.
As per the discussion at the end of \Cref{sec_example}, it follows from the definition of $\obsoperator_D$ in \eqref{eq_observation_operators_for_numerical_results} that $\obsoperator_D \model v=(s(x^{(i)})v(t_i))_{i\in[n]}$, where $(x^{(i)},t_i)_{i\in[n]}$ are given in \eqref{eq_spatiotemporal_locations_of_observations}.
By the interpolation property of the temporal basis $(\psi_j)_{j\in[N_T]}$, it follows that $(\mathbf{O}_D \mathbf{Mv})_i= s(x^{(i)})v(t_i)=s(x^{(i)})\mathbf{v}_i$ for every $i\in[n]$.

We now use the preceding argument to identify the orthogonal complement of $\kernel(\mathbf{O}_D\mathbf{M})$.
Denote the support of a vector $\mathbf{v}\in\bb{R}^{N_T}$ by $\textup{supp}(\mathbf{v})\coloneqq\{j\in[N_T]\ :\ \mathbf{v}_j\neq 0\}$. 
Recall from \Cref{ssec_implementation} that $\Delta t=0.01$ and denote the index set of observation times given in \eqref{eq_spatiotemporal_locations_of_observations} by $I\coloneqq \{i\in[N_T]\ :\ (i-1) \Delta t\in\{ 0.1+0.01\ell\ :\ \ell\in[80]_0\}\}$.
Since $(x_1,x_2)=(0.4,0.4)$ is one of the spatial locations $x^{(i)}$ in \eqref{eq_spatiotemporal_locations_of_observations} at which measurements are taken, and since $(0.4,0.4)$ lies in the support of the function $s(\cdot)$ defined in \eqref{eq_source_term_spatial_component}, the preceding paragraph implies that $s(0.4,0.4)$ is strictly positive, and that $\mathbf{O}_D \mathbf{Mv}=0$ if and only if $\mathbf{v}_i=0$ for every $i\in I\subset [N_T]$.
Thus,
\begin{equation*}
 \kernel(\mathbf{O}_D\mathbf{M})^\perp=\spn\left\{\mathbf{v}\in\bb{R}^{N_T}\ :\ \textup{supp}(\mathbf{v})\subseteq I \right\}.
\end{equation*}
Since $(\mathbf{O}_D\mathbf{M})^+(\mathbf{O}_D\mathbf{M})$ is the projection to the orthogonal complement in $\bb{R}^{N_T}$ of $\kernel(\mathbf{O}_D\mathbf{M})$, it follows that for every $\mathbf{v}\in\bb{R}^{N_T}$, $\left((\mathbf{O}_D\mathbf{M})^+(\mathbf{O}_D\mathbf{M})\mathbf{v}\right)_j= \mathbf{v}_j$ if $j\in I$ and is zero otherwise.
In particular, we expect that as $\Norm{\Sigma_\noise}_{\op}$ decreases to zero, the best posterior mean $\mathbf{m}^{\data,\best}_{\param}$ for $\mathbf{O}_D$ will more closely match the discretised true parameter $\widetilde{\param}^{\best}$ on the index set $I$ of observation times.
An important corollary is that if $\mathbf{O}=\mathbf{O}_D$ and in particular if $\mathbf{O}_D\widetilde{\modelerror}^{\best}$ is close to $0$, then by comparing \eqref{eq_discrete_representation_approximate_posterior_mean} and \eqref{eq_discrete_representation_best_posterior_mean}, the approximate posterior mean $\mathbf{m}^{\data,\appr}_{\param}$ will be close to the discretised true parameter $\widetilde{\param}^{\best}$.

We close this section by observing that unlike $\mathbf{O}_D$, $(\mathbf{O}_B\mathbf{Mv})_i$ may be nonzero even if $\mathbf{v}_i=0$.
This is because $(\mathbf{O}_B \mathbf{Mv})_i$ approximates the value of the solution of \eqref{eq_PDE_IBVP_zero_IC} at some spatiotemporal point $(x^{(i)},t_i)$.
By the effects of diffusion and advection, $(\mathbf{O}_B \mathbf{Mv})_i$ may be nonzero even if $s(x^{(i)})v(t_i)=0$, e.g. if at spatiotemporal points near $(x^{(i)},t_i)$ the source term of the PDE is sufficiently large.
Hence, we do not expect the best posterior mean for $\mathbf{O}_B$ to agree with $\widetilde{\param}^{\best}$ to the same extent that the best posterior mean for $\mathbf{O}_D$ agrees with $\widetilde{\param}^{\best}$ on the index set $I$ of observation times.

\subsection{Results}
\label{ssec_results}

\begin{figure}[!ht]
\centering
    \begin{subfigure}[b]{0.475\textwidth}
        \includegraphics[width=\textwidth]{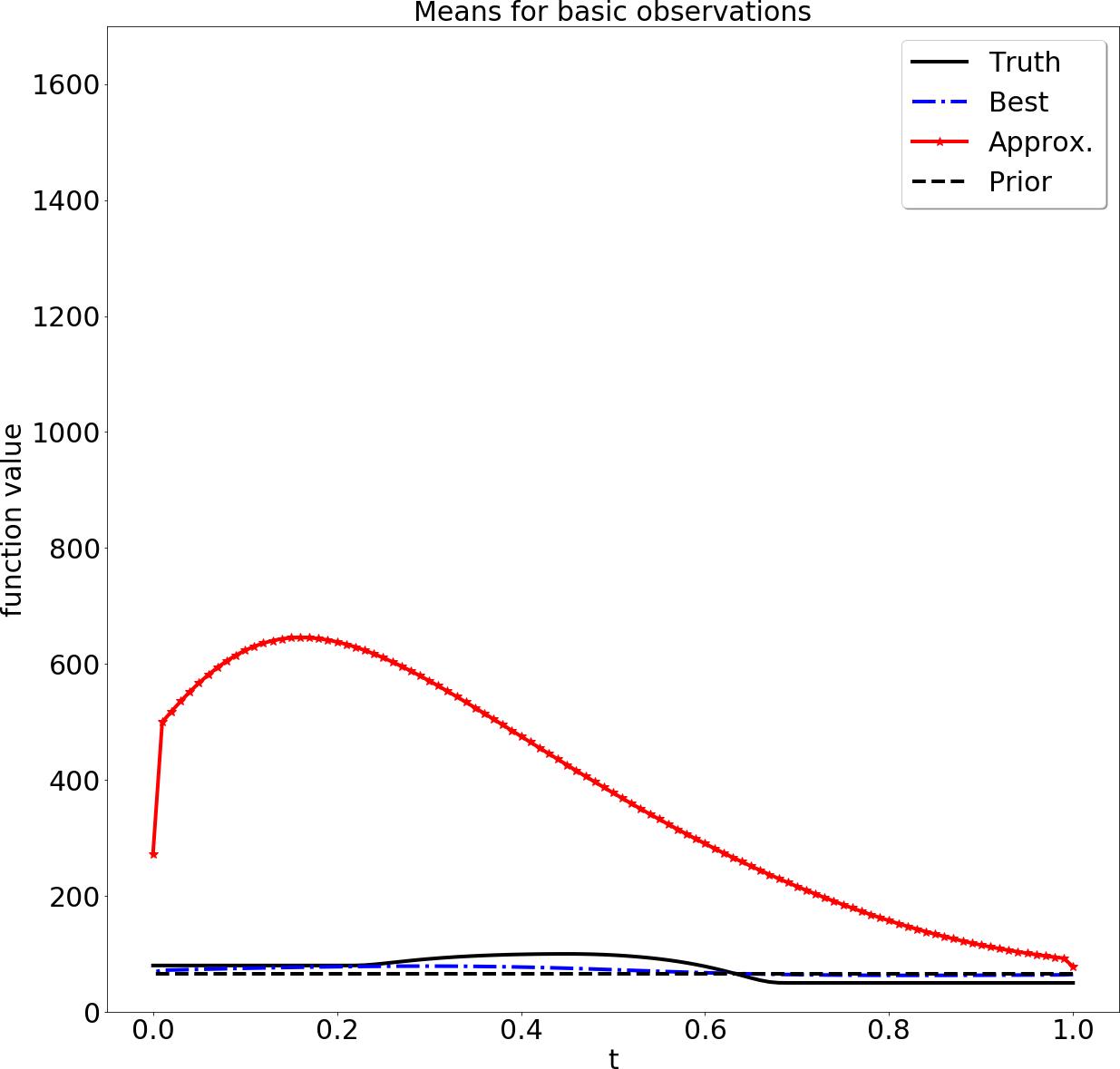}
        \caption{$\mathbf{m}^{\data,\appr}_\param$ and $\mathbf{m}^{\data,\best}_\param$ for $\mathbf{O}_B$, lower SNR}
        \label{fig_PosteriorMeans_BasicObservation_LowSNR}
    \end{subfigure}
      \hfill    
    \begin{subfigure}[b]{0.475\textwidth}
        \includegraphics[width=\textwidth]{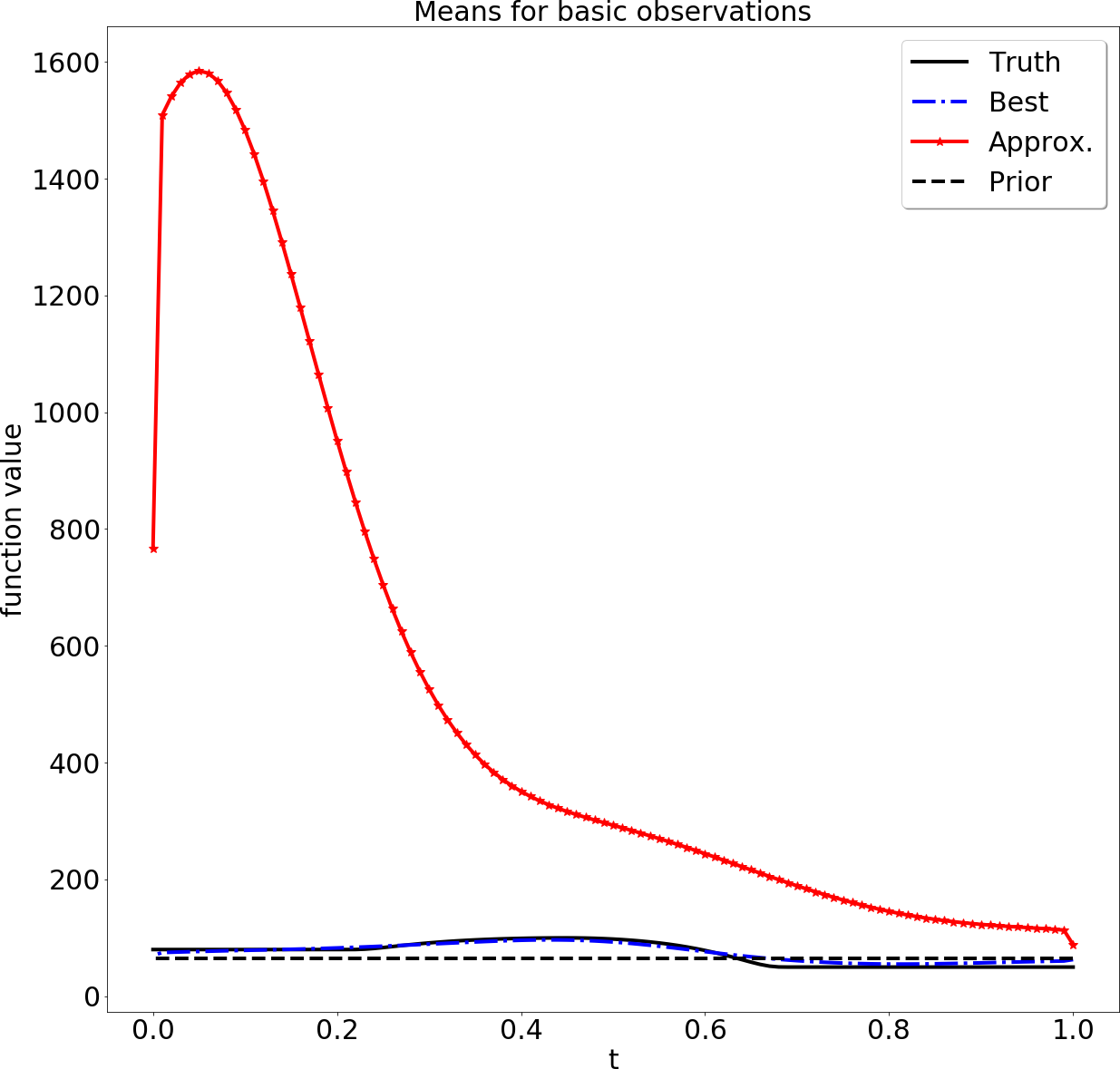}
        \caption{$\mathbf{m}^{\data,\appr}_\param$ and $\mathbf{m}^{\data,\best}_\param$ for $\mathbf{O}_B$, higher SNR}
        \label{fig_PosteriorMeans_BasicObservation_HighSNR}
    \end{subfigure}
    \caption{Comparison of basic posterior mean $\mathbf{m}^{\data,\appr}_{\param}$ (red solid line with red {$\blacktriangle$} markers) and $\mathbf{m}^{\data,\best}_{\param}$ (blue dash-dotted line) from \eqref{eq_discrete_representation_posteriors} for basic observation operator $\mathbf{O}_B$, against prior mean $\mathbf{m}_{\param}$ (black dashed line) and truth $\widetilde{\param}^{\best}$ (black solid line).}
    \label{fig_posterior_means_basic_observation_operator}
\end{figure}
\Cref{fig_PosteriorMeans_BasicObservation_LowSNR} and \Cref{fig_PosteriorMeans_BasicObservation_HighSNR} show the approximate posterior mean $\mathbf{m}^{\data,\appr}_{\param}$ and the best posterior mean $\mathbf{m}^{\data,\best}_{\param}$ for the basic observation operator $\mathbf{O}_B$, for the lower SNR and higher SNR respectively.
By \eqref{eq_discrete_representation_approximate_posterior_mean} and \eqref{eq_discrete_representation_best_posterior_mean}, 
\begin{equation*}
 \mathbf{m}^{\data,\appr}_{\param}-\mathbf{m}^{\data,\best}_{\param}=  (\mathbf{O}_B\mathbf{M\Sigma}_{\param})^\ast((\mathbf{O}_B\mathbf{M})\mathbf{\Sigma}_{\param}(\mathbf{O}_B\mathbf{M})^\ast + \Sigma_{\noise})^{-1}(\mathbf{O}_B\widetilde{\modelerror}^{\best})
\end{equation*}
is captured by the difference between the red and blue curves in \Cref{fig_PosteriorMeans_BasicObservation_LowSNR} and \Cref{fig_PosteriorMeans_BasicObservation_HighSNR}.
In both figures, the difference eventually decreases with time.
This is because $\widetilde{\modelerror}^{\best}$ is the numerical solution of the PDE-IBVP \eqref{eq_PDE_IBVP_modelerror}, which has zero forcing and allows for flux across the boundary, and thus the values of $\widetilde{\modelerror}^{\best}$ decay with time.
Nevertheless, the difference is large relative to the truth $\widetilde{\param}^{\best}$ over the time interval $\mcal{T}$.
These observations show the importance of the model error for this specific inverse problem.

By comparing \Cref{fig_PosteriorMeans_BasicObservation_LowSNR} and \Cref{fig_PosteriorMeans_BasicObservation_HighSNR}, we observe that the difference $\mathbf{m}^{\data,\appr}_{\param}-\mathbf{m}^{\data,\best}_{\param}$ increases considerably after increasing the SNR.
This occurs partly because the lower SNR corresponds to a larger $\sigma_\noise$ in \eqref{eq_true_data_for_numerical_simulations}, which in turn leads to larger values of $\noise$ and thus to entries of $\mathbf{O}\widetilde{\modelerror}^{\best}+\noise$ that are more likely to be closer to zero.
This behaviour is consistent with what we expect, namely that ignoring model error can lead to larger errors in the posterior mean when the data are informative and the model error is large.

\begin{figure}[!ht]
\centering
     \begin{subfigure}[b]{0.475\textwidth}
        \includegraphics[width=\textwidth]{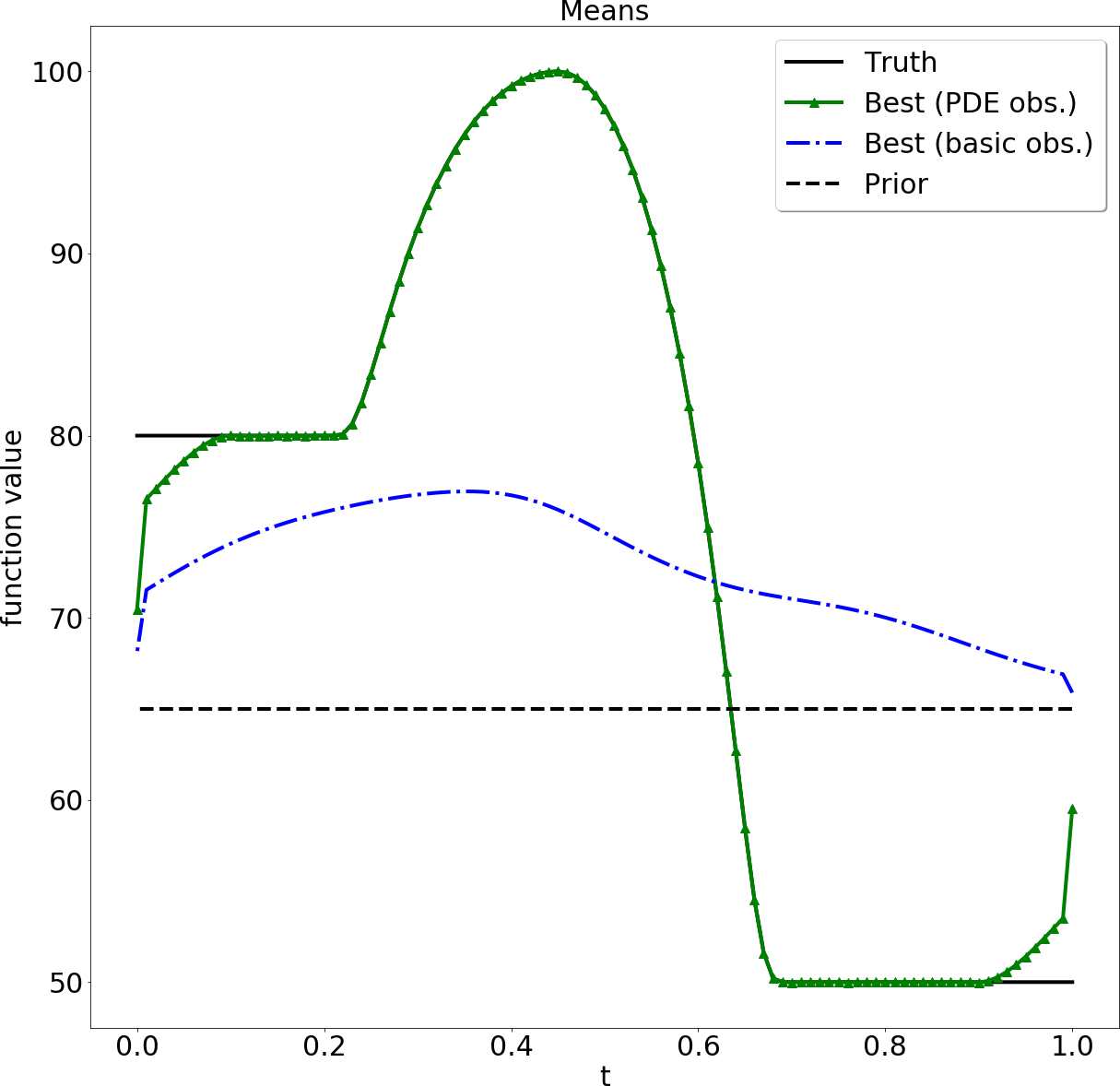}
        \caption{$\mathbf{m}^{\data,\best}_\param$ for $\mathbf{O}_B$ and $\mathbf{m}^{\data,\best}_{\param}$ for $\mathbf{O}_D$, lower SNR}
        \label{fig_BestPosteriorMeans_LowSNR}
    \end{subfigure}
      \hfill    
      \begin{subfigure}[b]{0.475\textwidth}
        \includegraphics[width=\textwidth]{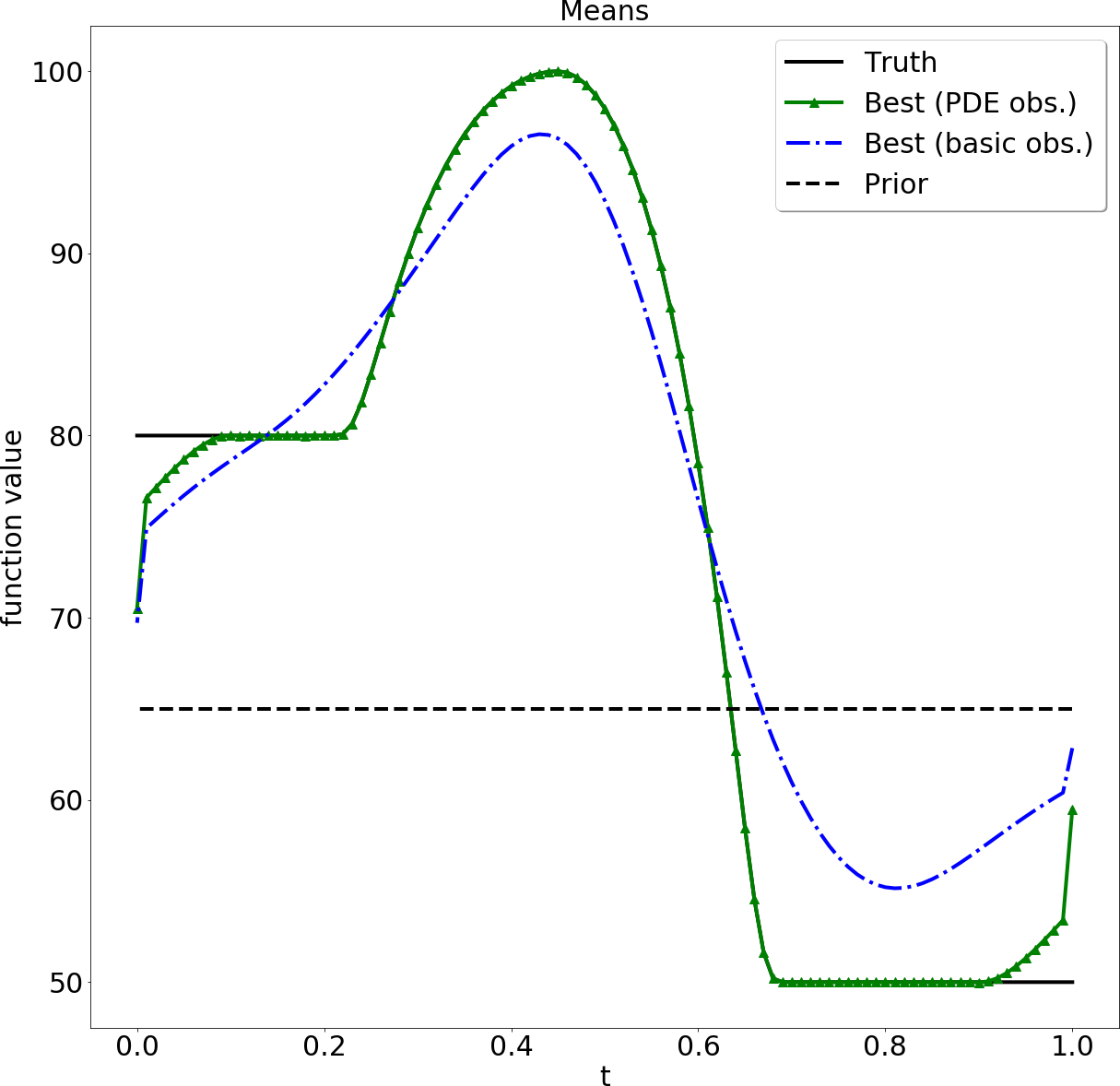}
        \caption{$\mathbf{m}^{\data,\best}_\param$ for $\mathbf{O}_B$ and $\mathbf{m}^{\data,\best}_{\param}$ for $\mathbf{O}_D$, higher SNR}
        \label{fig_BestPosteriorMeans_HighSNR}
    \end{subfigure}
\caption{Comparison of best posterior means $\mathbf{m}^{\data,\best}_{\param}$ from \eqref{eq_discrete_representation_best_posterior_mean} for basic observation operator $\mathbf{O}_{B}$ (blue dash-dotted line) and PDE observation operator $\mathbf{O}_{D}$ (green solid line with green $\blacktriangle$ markers), against prior mean $\mathbf{m}_{\param}$ (black dashed line) and truth $\widetilde{\param}^{\best}$ (black solid line).}
    \label{fig_best_posterior_means}
\end{figure}
In \Cref{fig_BestPosteriorMeans_LowSNR} (respectively, \Cref{fig_BestPosteriorMeans_HighSNR}), we show for the lower (resp. higher) SNR setting the best posterior mean $\mathbf{m}^{\data,\best}_{\param}$ for $\mathbf{O}_B$, the best posterior mean $\mathbf{m}^{\data,\best}_{\param}$ for $\mathbf{O}_D$, and the truth $\widetilde{\param}^{\best}$.
For $\mathbf{O}_B$, $\mathbf{m}^{\data,\best}_{\param}$ fails to capture important features of $\widetilde{\param}^{\best}$ for the lower SNR; it does better for the higher SNR, but under- and over-predictions remain visible.
In contrast, $\mathbf{m}^{\data,\best}_\param$ for $\mathbf{O}_D$ matches the truth $\widetilde{\param}^{\best}$ very closely over the observation time window $0.1\leq t \leq 0.9$ for both the lower and higher SNR.
The close agreement between the best posterior mean $\mathbf{m}^{\data,\best}_{\param}$ and $\widetilde{\param}^{\best}$ --- and the different performance of the best posterior mean depending on whether $\mathbf{O}_D$ or $\mathbf{O}_B$ is used --- are consistent with the discussion in \Cref{ssec_behaviour_of_best_posterior_mean_in_small_noise_limit}.

\begin{figure}[ht]
\centering
    \begin{subfigure}[b]{0.475\textwidth}
        \includegraphics[width=\textwidth]{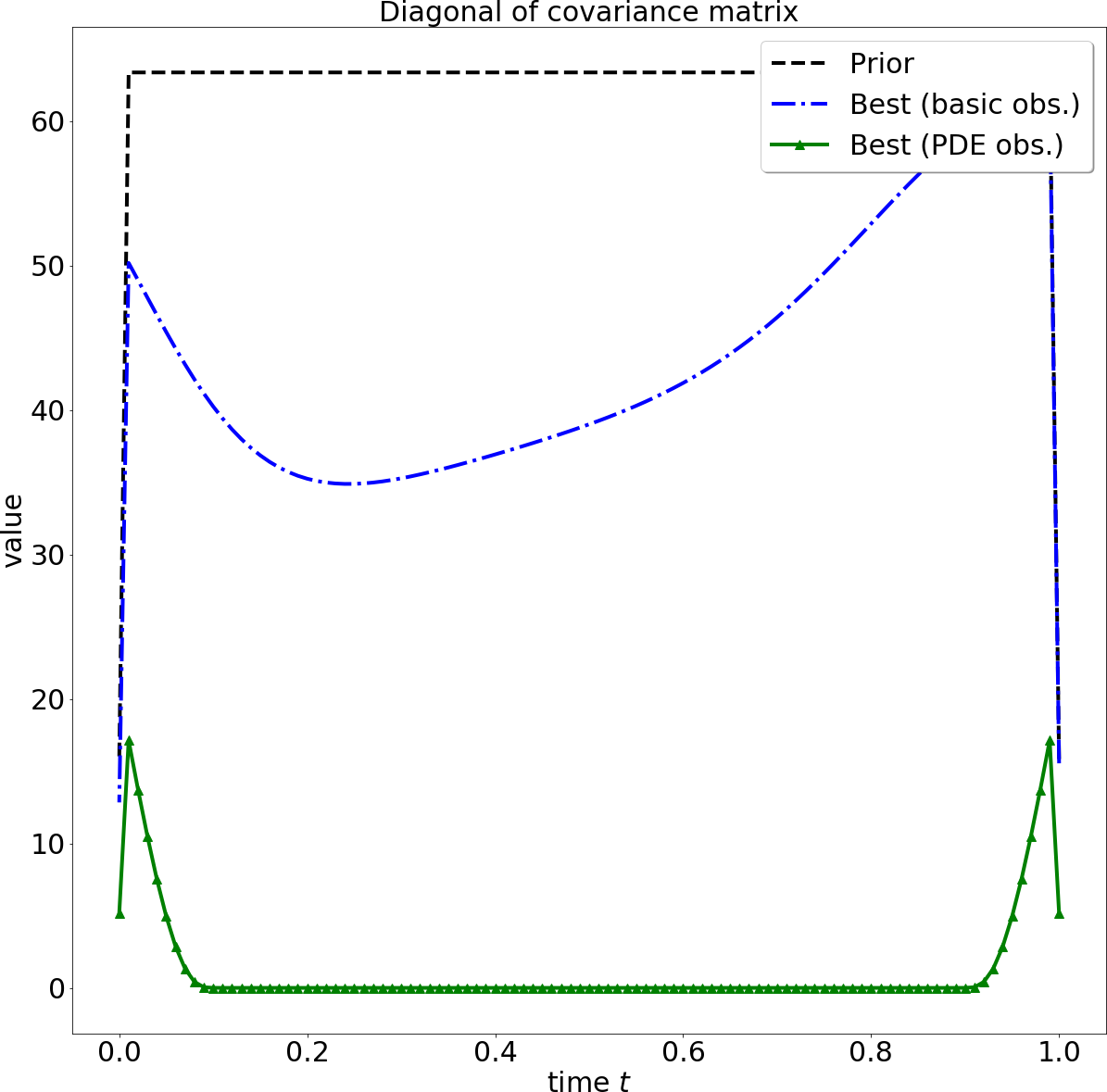}
        \caption{Diagonal entries of $\mathbf{\Sigma}^{\data,\best}_{\param}$ for $\mathbf{O}_B$ and $\mathbf{O}_D$, lower SNR}
        \label{fig_BestPosteriorCovariance_LowSNR}
    \end{subfigure}
      \hfill    
      \begin{subfigure}[b]{0.475\textwidth}
        \includegraphics[width=\textwidth]{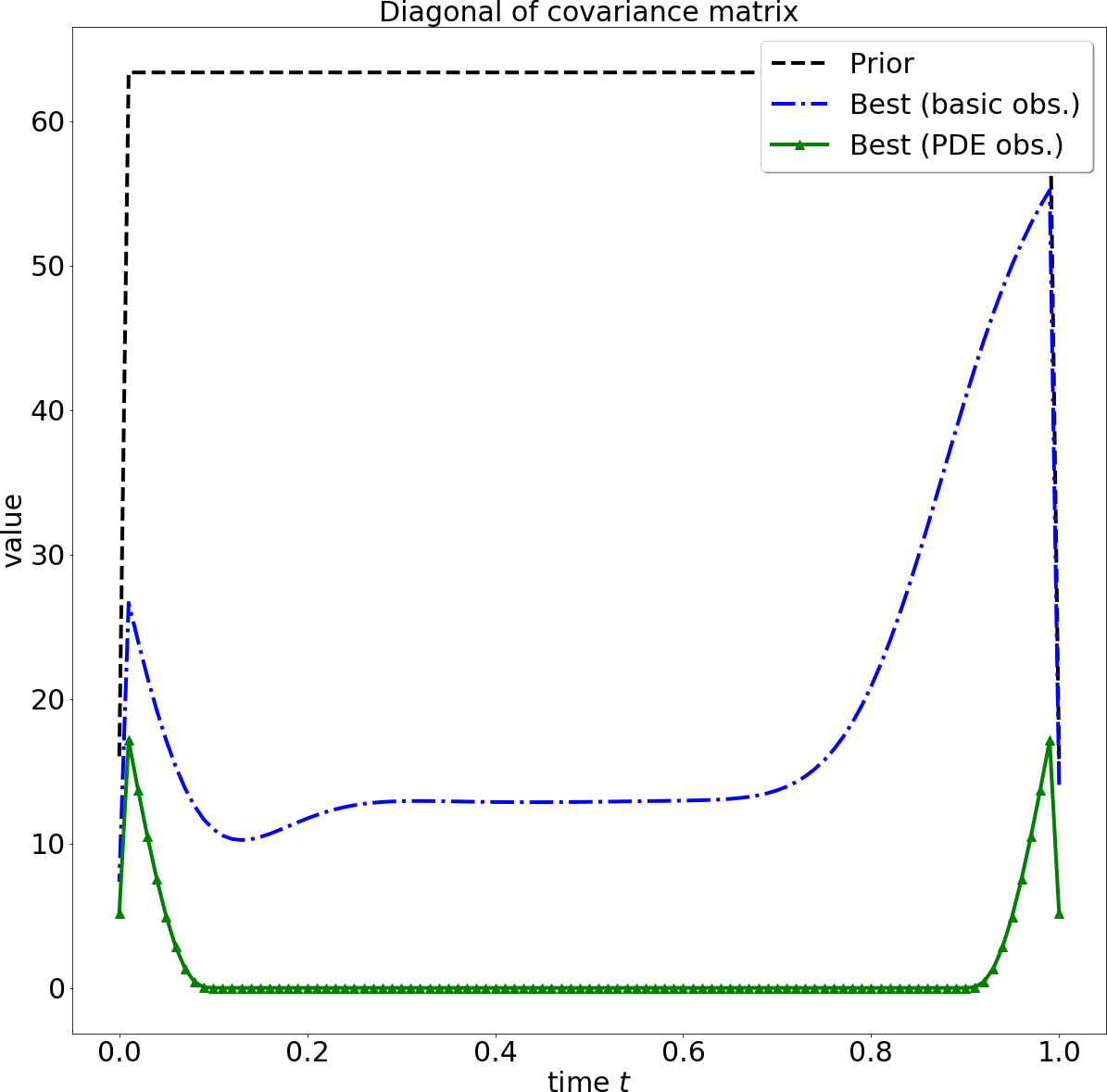}
        \caption{Diagonal entries of $\mathbf{\Sigma}^{\data,\best}_{\param}$ for $\mathbf{O}_B$ and $\mathbf{O}_D$, higher SNR}
        \label{fig_BestPosteriorCovariance_HighSNR}
    \end{subfigure}
    \caption{Diagonal entries of prior covariance $\mathbf{\Sigma}_{\param}$ (black dashed line) and posterior covariance $\mathbf{\Sigma}^{\data,\best}_{\param}$ from \eqref{eq_discrete_representation_approximate_posterior_covariance} for basic observation operator $\mathbf{O}_B$ (blue dash-dotted line) and PDE observation operator $\mathbf{O}_D$ (green solid line with green $\blacktriangle$ markers) as a function of time, for lower SNR (left) and higher SNR (right).}
    \label{fig_posterior_covariance_comparisons}
\end{figure}

In \Cref{fig_BestPosteriorCovariance_LowSNR} and \Cref{fig_BestPosteriorCovariance_HighSNR}, we compare the diagonal entries of the best posterior covariance matrix $\mathbf{\Sigma}^{\data,\best}_{\param}$ for $\mathbf{O}_B$ and for $\mathbf{O}_D$, in the lower SNR and higher SNR cases respectively.
We observe that the diagonal of $\mathbf{\Sigma}^{\data,\best}_{\param}$ for $\mathbf{O}_D$ is close to zero over the time observation window, whereas the diagonal of $\mathbf{\Sigma}^{\data,\best}_{\param}$ for $\mathbf{O}_B$ is not.
By \eqref{eq_discrete_representation_approximate_posterior_covariance}, the difference in behaviour is due to the size of the noise covariance $\Sigma_\noise$.
For $\mathbf{O}_D$, the noise standard deviation satisfies $\sigma_\noise=O(10^{-7})$ and $\sigma_\noise=O(10^{-8})$ in the lower and higher SNR cases respectively.
Thus $\Sigma_\noise$ is almost zero for $\mathbf{O}_D$, and $\mathbf{\Sigma}^{\data,\best}_{\param}$ is close to zero as well.
In contrast, for $\mathbf{O}_B$, $\sigma_\noise \approx 5.66$ and $\sigma_\noise\approx 1.13$ in the lower and higher SNR cases respectively, so that $\Sigma_\noise$ is not negligible.
The differences in $\mathbf{\Sigma}^{\data,\best}_{\param}$ for $\mathbf{O}_B$ in the lower and higher SNR cases are also due to the size of $\Sigma_\noise$.

By comparing the curve for the approximate posterior mean $\mathbf{m}^{\data,\appr}_{\param}$ for $\mathbf{O}_B$ from \Cref{fig_PosteriorMeans_BasicObservation_LowSNR} with the curve for $\mathbf{\Sigma}^{\data,\best}_{\param}$ for $\mathbf{O}_B$ in \Cref{fig_BestPosteriorCovariance_LowSNR} and recalling that $\mathbf{\Sigma}^{\data,\best}_{\param}=\mathbf{\Sigma}^{\data,\appr}_{\param}$, we note that the error of $\mathbf{m}^{\data,\appr}_{\param}$ with respect to $\widetilde{\param}^{\best}$ is the largest over the subinterval $0.1\leq t\leq 0.2$, but the marginal posterior variance is smallest over the same interval.
In other words, the approximate posterior $\mu^{\data,\appr}_{\param}$ is most confident over the region where the bias in the approximate posterior mean is the largest. 
The approximate posterior is said to be `overconfident'; see e.g. \cite[Section 8]{AbdulleGaregnani2020} and the references therein for other examples of overconfidence in the presence of model error.
By comparing \Cref{fig_BestPosteriorMeans_HighSNR} and \Cref{fig_BestPosteriorCovariance_HighSNR}, we see that the phenomenon of overconfidence becomes worse for higher SNRs, i.e. for more informative data.

We make some important remarks about the results in 
\Cref{fig_best_posterior_means} and 
\Cref{fig_posterior_covariance_comparisons}. These results show that if one uses an observation operator that satisfies $\obsoperator \modelerror^\best=0$, then the best posterior $\mu^{\data,\best}_{\param}$ and the approximate posterior $\mu^{\data,\appr}_{\param}$ are concentrated around the true parameter.
We emphasise that these results are for the specific inverse problem presented in \Cref{ssec_setup_Bayesian_inverse_problem}. 
Our explanation in \Cref{ssec_behaviour_of_best_posterior_mean_in_small_noise_limit} for why the best posterior mean $\mathbf{m}^{\data,\best}_{\param}$ is close to $\param^{\best}$, and our explanation in the preceding paragraphs for why the diagonal of $\mathbf{\Sigma}^{\data,\best}_{\param}$ is close to zero over the time observation window, both make use of the size of the noise covariance $\Sigma_\noise$.
This suggests that whether one can obtain similar results for other inverse problems can depend on other factors, and not only on the condition $\obsoperator \modelerror^\best=0$.
Indeed, in the last paragraph of \Cref{sec_example}, we showed for a modification of our inverse problem that the condition $\obsoperator \modelerror^\best=0$ can even lead to uninformative observations.

\section{Conclusion}
\label{sec_conclusion}

We considered Bayesian inverse problems in the presence of model error in the finite-dimensional data setting with additive, Gaussian noise, and the parameter-to-observable map is the composition of a linear observation operator with a possibly nonlinear model.
Under the assumption that there exists a unique best model and best parameter, the `model error' is then the difference between the `approximate' model that one uses and the best model.
The approximate model and best model generate the `approximate posterior' and `best posterior' respectively.

We described some natural approaches for accounting for model error and the associated posteriors.
We used the local Lipschitz stability property of posteriors with respect to perturbations in the likelihood to bound Kullback--Leibler divergences of the posterior associated to each approach, with respect to either the best posterior or the approximate posterior.
These bounds have two important properties: first, they control the Kullback--Leibler divergence in terms of quantities that depend on the observation operator and the objects used to account for model error; and second, the terms in the bounds are finite under mild hypotheses, e.g. $L^1$-integrability of the misfits with respect to the prior.

Our bounds yield sufficient (respectively, necessary) conditions on the observation operator and the objects used by each approach to yield a posterior that agrees with the best posterior (resp. that differs from the approximate posterior).
These conditions provide guidelines for how to choose observation operators to mitigate the effect of model error on Bayesian inference, given a chosen approach.
A unifying theme of these conditions is the importance of the kernel of the observation operator.

Using a linear Gaussian inverse problem involving an advection-diffusion-reaction PDE, we illustrated the sufficient condition for the approximate posterior to agree with the best posterior.
We then reported numerical results from some simulations of this inverse problem under two different signal-to-noise settings, and used these results to discuss the importance of the sufficient condition itself, of observation noise, and of the information conveyed by the observation operator. 

It is natural to ask whether it would be possible to combine classical optimal experimental design with the approach presented above. For example, in an ideal setting where one has access to multiple observation operators $(\obsoperator_i)_{i\in I}$ such that $\modelerror^\best \in \kernel(\obsoperator_i)$ for every $i\in I$, one could aim to optimise the information gain over the set $(\obsoperator_i)_{i\in I}$.
Another potential avenue for future research would be to investigate the design of observation operators that satisfy the sufficient conditions; we expect this to be challenging and highly problem-specific.
It would be of interest to study the setting of nonlinear observation operators and other approaches for accounting for model error.

\section*{Acknowledgements}

The research of NC and HB has received funding from the ERC under the European Union’s Horizon 2020 Research and Innovation Programme --- Grant Agreement ID \href{https://cordis.europa.eu/project/id/818473}{818473}. The research of HCL has been partially funded by the Deutsche Forschungsgemeinschaft (DFG) --- Project-ID \href{https://gepris.dfg.de/gepris/projekt/318763901}{318763901} --- SFB1294.
The authors would like to thank Nicole Aretz for discussions on the numerics of PDEs and Giuseppe Carere for helpful comments.
In addition, the authors would like to thank the three reviewers and review editor for their feedback.

\appendix

\section{Technical lemmas}
\label{sec_technical_lemmas}

Let $d\in\bb{N}$ and $M_i\in\bb{R}^{d\times d}$, $i=1,2$ be symmetric and positive definite.

Recall the notation \eqref{eq_matrix_weighted_inner_product_norm} for a matrix-weighted norm.
\begin{lemma}
 \label{lem_linear_algebra}
Suppose that $M_1,M_2\in\bb{R}^{d\times d}$ are symmetric, that $M_1$ is positive definite, and that $M_2$ is nonnegative definite.  
Then $M_1+M_2$ is symmetric positive definite and $M_1^{-1}-(M_1+M_2)^{-1}$ is symmetric nonnegative definite.
In addition, 
\begin{equation}
\label{eq_equivalence_of_matrix_induced_norms}
\frac{ \Abs{z}_{(M_1+M_2)^{-1}}}{\Norm{ (M_1+M_2)^{-1/2}M_1^{1/2}}}_{\op}\leq \Abs{z}_{M_1^{-1}}\leq \Norm{ M_1^{-1/2}(M_1+M_2)^{1/2}}_{\op} \Abs{z}_{(M_1+M_2)^{-1}},\quad z\in\bb{R}^{d}.
\end{equation}
\end{lemma}
\begin{proof}
By the assumptions on $M_1$ and $M_2$, $M_1+M_2$ is positive definite.
As the difference of two symmetric matrices, $M_1^{-1}-(M_1+M_2)^{-1}$ is symmetric. 
To show that it is nonnegative, we use the following rearrangement of the Woodbury formula, which we take from \cite[Equation (3)]{Chang2006}:
\begin{align*}
  (M_1+M_2)^{-1}=&M_1^{-1}-M_1^{-1}M_2(I+M_1^{-1}M_2)^{-1}M_1^{-1}
 \nonumber
 \\
 \Longleftrightarrow M_1^{-1}-(M_1+M_2)^{-1}=&M_1^{-1}M_2(I+M_1^{-1}M_2)^{-1}M_1^{-1}.
\end{align*}
Invertibility of $I+M_1^{-1}M_2=M_1^{-1}(M_1+M_2)$ follows from the invertibility of $M_1^{-1}$ and $M_1+M_2$.
From the second equation above, $M_1^{-1}-(M_1+M_2)^{-1}$ inherits the nonnegative definiteness of $M_2$.

The first inequality of \eqref{eq_equivalence_of_matrix_induced_norms} follows by
\begin{align*}
 \Abs{z}_{(M_1+M_2)^{-1}} & =\Abs{ (M_1+M_2)^{-1/2}z}=\Abs{ (M_1+M_2)^{-1/2}M_1^{1/2}M_1^{-1/2}z}
 \\
 &\leq \Norm{ (M_1+M_2)^{-1/2}M_1^{1/2}}_{\op}\Abs{ M_1^{-1/2}z} 
 \\
 &= \Norm{ (M_1+M_2)^{-1/2}M_1^{1/2}}_{\op} \Abs{ z}_{M_1^{-1}}.
\end{align*}
The second inequality of \eqref{eq_equivalence_of_matrix_induced_norms} follows by
\begin{align*}
 \Abs{z}_{M_1^{-1}} & =\Abs{ M_1^{-1/2}z}=\Abs{ M_1^{-1/2}(M_1+M_2)^{1/2}(M_1+M_2)^{-1/2}z}
 \\
 &\leq \Norm{ M_1^{-1/2}(M_1+M_2)^{1/2}}_{\op}\Abs{ (M_1+M_2)^{-1/2}z}
 \\
 &=\Norm{ M_1^{-1/2}(M_1+M_2)^{1/2}}_{\op}\Abs{ z}_{(M_1+M_2)^{-1}}.
\end{align*}
This completes the proof of \Cref{lem_linear_algebra}.
\end{proof}

We will use the following bound in the proofs below.
\begin{lemma}
 \label{lem_L1norm_diff_quadratic_forms}
 Let $(E,d_E)$ be a metric space, $\mu\in\mcal{P}(E)$, and $f$ and $g$ be $\bb{R}^{d}$-valued measurable functions on $E$.
 Let $M_1\in\bb{R}^{d\times d}$ be symmetric and positive definite and $M_2\in\bb{R}^{d\times d}$ be symmetric nonnegative definite.
 Then
 \begin{align}
 \label{eq_lem_L1norm_diff_quadratic_forms_same_weighting_matrix}
  &\Norm{ \Abs{f}^{2}_{M_1^{-1}}-\Abs{f+g}^{2}_{M_1^{-1}}}_{L^1_\mu}\leq \Norm{ \Abs{g}^{2}_{M_1^{-1}}}_{L^1_\mu}^{1/2}\left( \Norm{\Abs{f+g}^{2}_{M_1^{-1}}}_{L^1_\mu}^{1/2}+\Norm{ \Abs{f}^{2}_{M_1^{-1}}}_{L^1_\mu}^{1/2}\right).
  \end{align}
  More generally,
\begin{align}
 & \Norm{ \Abs{f}^{2}_{M_1^{-1}}-\Abs{f+g}^{2}_{(M_1+M_2)^{-1}}}_{L^1_\mu}
  \nonumber
  \\
  \leq & \Norm{ \Abs{g}^{2}_{M_1^{-1}}}_{L^1_\mu}^{1/2}
  \left(\Norm{M_1^{-1/2}(M_1+M_2)^{1/2}}_{\op}\cdot\Norm{\Abs{f+g}^{2}_{(M_1+M_2)^{-1}}}_{L^1_\mu}^{1/2}+\Norm{ \Abs{f}^{2}_{M_1^{-1}}}_{L^1_\mu}^{1/2}\right)
  \label{eq_lem_L1norm_diff_quadratic_forms_different_weighting_matrix}
  \\
  &+\Norm{ \Abs{f+g}_{M_1^{-1}-(M_1+M_2)^{-1}}^{2} }_{L^1_\mu},
 \nonumber
 \end{align}
 where $\Abs{f+g}_{M_1^{-1}-(M_1+M_2)^{-1}}^{2}\coloneqq \Abs{f+g}_{M_1^{-1}}^{2}-\Abs{f+g}_{(M_1+M_2)^{-1}}^{2}$.
 In addition, 
 \begin{align}
     &\Norm{ \Abs{g}^{2}_{M_1^{-1}}}_{L^1_\mu}^{1/2}\leq \Norm{ \Abs{f}_{M_1^{-1}}^{2}}_{L^1_\mu}^{1/2}+ \Norm{M_1^{-1/2}(M_1+M_2)^{1/2}}_{\op}\Norm{ \Abs{f+g}_{(M_1+M_2)^{-1}}^{2}}_{L^1_\mu}^{1/2}
     \label{eq_bound_L1norm_g}
     \\
     &\Norm{ \Abs{f+g}_{M_1^{-1}-(M_1+M_2)^{-1}}^{2} }_{L^1_\mu}\leq \left(1+\Norm{M_1^{-1/2}(M_1+M_2)^{1/2}}_{\op}\right)\Norm{ \Abs{f+g}_{(M_1+M_2)^{-1}}^{2} }_{L^1_\mu}.
     \label{eq_bound_L1norm_fplusg_difference_of_weighting_matrices}
 \end{align}
  \end{lemma}
\begin{proof}
We claim that for arbitrary $a,b\in\bb{R}^{d}$, symmetric positive definite $M_1\in\bb{R}^{d\times d}$, and symmetric nonnegative definite $M_2\in\bb{R}^{d\times d}$,
 \begin{equation}
     \label{eq_difference_of_quadratic_forms}
     \Abs{a}^{2}_{M_1^{-1}}-\Abs{a+b}^{2}_{(M_1+M_2)^{-1}}=-\ang{b,2a+b}_{M_1^{-1}}+\Abs{a+b}^{2}_{M_1^{-1}-(M_1+M_2)^{-1}}.
 \end{equation}
 Recall that $(M_1+M_2)^{-1}$ exists and is positive definite, by \Cref{lem_linear_algebra}.

Using the matrix-weighted inner product and norm notation from \eqref{eq_matrix_weighted_inner_product_norm},
\begin{align*}
 &\Abs{a}^{2}_{M_1^{-1}}-\Abs{a+b}^{2}_{M_1^{-1}}
 \\
 =&a^\top M_1^{-1} a-(a+b)^\top M_1^{-1}(a+b) 
 \\
 = &a^\top M_1^{-1}a -(a^\top M_1^{-1} a+ 2a^\top M_1^{-1} b+b^\top M_1^{-1}b)
 \\
 =&-b^\top M_1^{-1}(2a+b)
 \\
 =&-\ang{b,2a+b}_{M_1^{-1}}.
 \end{align*}
 This implies \eqref{eq_difference_of_quadratic_forms}, since
 \begin{align*}
 &\Abs{a}_{M_1^{-1}}^{2}-\Abs{a+b}^{2}_{M_1^{-1}}+ \Abs{a+b}^{2}_{M_1^{-1}}-\Abs{a+b}^{2}_{(M_1+M_2)^{-1}}
 \\
 =&-\ang{b,2a+b}_{M_1^{-1}}+ \Abs{a+b}^{2}_{M_1^{-1}}-\Abs{a+b}^{2}_{(M_1+M_2)^{-1}}.
 \end{align*}
 Now let $f$, $g$, and $\mu$ be as in the statement of the lemma. 
 Then by \eqref{eq_difference_of_quadratic_forms} and the triangle inequality,
 \begin{align}
 \Norm{ \Abs{f}_{M_1^{-1}}^{2}-\Abs{f+g}_{(M_1+M_2)^{-1}}^{2} }_{L^1_\mu}=& \Norm{ -\ang{g,2f+g}_{M_1^{-1}}+\Abs{f+g}_{M_1^{-1}-(M_1+M_2)^{-1}}^{2} }_{L^1_\mu}
  \nonumber
  \\
  \leq & \Norm{\ang{g,2f+g}_{M_1^{-1}}}_{L^1_\mu}+\Norm{ \Abs{f+g}_{M_1^{-1}-(M_1+M_2)^{-1}}^{2} }_{L^1_\mu}.
  \label{lem_L1norm_diff_quadratic_forms_proof_eq_01}
 \end{align}
Next,
\begin{align}
 \Norm{\ang{g,2f+g}_{M_1^{-1}}}_{L^1_\mu} \leq& \Norm{ \Abs{g}_{M_1^{-1}}\Abs{2f+g}_{M_1^{-1}}}_{L^1_\mu} 
 \nonumber
 \\
 \leq & \Norm{ \Abs{g}_{M_1^{-1}}}_{L^2_\mu}\Norm{\Abs{2f+g}_{M_1^{-1}}}_{L^2_\mu}
 \nonumber
 \\
 \leq & \Norm{ \Abs{g}_{M_1^{-1}}}_{L^2_\mu}\left( \Norm{\Abs{f+g}_{M_1^{-1}}}_{L^2_\mu}+\Norm{\Abs{f}_{M_1^{-1}}}_{L^2_\mu}\right)
 \nonumber
 \\
 = & \Norm{ \Abs{g}_{M_1^{-1}}^{2}}_{L^1_\mu}^{1/2}\left( \Norm{\Abs{f+g}_{M_1^{-1}}^{2}}_{L^1_\mu}^{1/2}+\Norm{\Abs{f}_{M_1^{-1}}^{2}}_{L^1_\mu}^{1/2}\right).
\label{lem_L1norm_diff_quadratic_forms_proof_eq_02}
 \end{align}
The first and second inequalities follow by applying the Cauchy--Schwarz inequality with respect to $\ang{\cdot,\cdot}_{M_1^{-1}}$ and $\Norm{\cdot}_{L^1_{\mu}}$ respectively.
The third inequality and the equation follow from the $\Norm{\cdot}_{L^2_\mu}$-triangle inequality and the definition of the $L^p_\mu$ norm for $p=1,2$.
By \eqref{lem_L1norm_diff_quadratic_forms_proof_eq_02}, we bound the first term on the right-hand side of \eqref{lem_L1norm_diff_quadratic_forms_proof_eq_01}. 
By using $M_2\leftarrow 0$, the second term on the right-hand side of \eqref{lem_L1norm_diff_quadratic_forms_proof_eq_01} vanishes.
Thus \eqref{eq_lem_L1norm_diff_quadratic_forms_same_weighting_matrix} follows from \eqref{lem_L1norm_diff_quadratic_forms_proof_eq_01}.

Next, we bound the first term inside the parentheses on the right-hand side of \eqref{lem_L1norm_diff_quadratic_forms_proof_eq_02}.
By \eqref{eq_equivalence_of_matrix_induced_norms},
\begin{equation*}
 \Norm{\Abs{f+g}_{M_1^{-1}}^{2}}_{L^1_\mu}^{1/2} \leq  \Norm{M_1^{-1/2}(M_1+M_2)^{1/2}}_{\op} \Norm{\Abs{f+g}_{(M_1+M_2)^{-1}}^{2}}_{L^1_\mu}^{1/2}.
\end{equation*}
Using the above bound yields \eqref{eq_lem_L1norm_diff_quadratic_forms_different_weighting_matrix}.

To prove \eqref{eq_bound_L1norm_g},
\begin{align*}
 \Norm{ \Abs{g}^{2}_{M_1^{-1}}}_{L^1_\mu}^{1/2} =& \Norm{ \Abs{ -f+(f+g)}_{M_1^{-1}}}_{L^2_\mu}
 \\
 \leq& \Norm{\Abs{f}_{M_1^{-1}}}_{L^2_{\mu}}+\Norm{\Abs{f+g}_{M_1^{-1}}}_{L^2_\mu}
 \\
 \leq& \Norm{\Abs{f}_{M_1^{-1}}^{2}}_{L^1_{\mu}}^{1/2}+\Norm{M_1^{-1/2}(M_1+M_2)^{1/2}}_{\op}\Norm{\Abs{f+g}_{(M_1+M_2)^{-1}}^{2}}_{L^1_\mu}^{1/2}
\end{align*}
where the last inequality uses \eqref{eq_equivalence_of_matrix_induced_norms}.
To prove \eqref{eq_bound_L1norm_fplusg_difference_of_weighting_matrices}, observe that 
\begin{align*}
 \Norm{\Abs{f+g}_{M_1^{-1}-(M_1+M_2)^{-1}}^{2}}_{L^1_\mu}=&\Norm{\Abs{f+g}_{M_1^{-1}}^{2}-\Abs{f+g}_{(M_1+M_2)^{-1}}^{2}}_{L^1_\mu}
 \\
 \leq & \Norm{\Abs{f+g}_{M_1^{-1}}^{2}}_{L^1_\mu}+\Norm{\Abs{f+g}_{(M_1+M_2)^{-1}}^{2}}_{L^1_\mu}
 \\
 \leq & \left(\Norm{M_1^{-1/2}(M_1+M_2)^{1/2}}_{\op}+1\right)\Norm{\Abs{f+g}_{(M_1+M_2)^{-1}}^{2}}_{L^1_\mu}
 \end{align*}
 where the first and second inequality follow from the triangle inequality and \eqref{eq_equivalence_of_matrix_induced_norms}.
This completes the proof of \Cref{lem_L1norm_diff_quadratic_forms}.
\end{proof}
 
 \section{Deferred proofs}
 \label{sec_deferred_proofs}
 
 \subsection{Proof for approximate posterior}
 \label{ssec_proofs_KL_bound_approximate_posterior}
 
 \Cref{lem_L1error_best_misfit_approx_misfit} bounds $\Norm{ \Phi^{\data,\best} - \Phi^{\data,\appr} }_{L^{q}_{\mu_{\param}}}$ in terms of the observed model error $\obsoperator \modelerror^{\best}$, under the hypothesis that $\Phi^{\data,\best}\in L^1_{\mu_\param}$ and $\Phi^{\data,\appr}\in L^1_{\mu_\param}$.
The bound is given in \eqref{eq_L1error_best_misfit_approx_misfit}:
\begin{equation*}
\Norm{ \Phi^{\data,\best} - \Phi^{\data,\appr} }_{L^{1}_{\mu_{\param}}}\leq 2^{-1/2} \Norm{ \Abs{\obsoperator \modelerror^{\best}}_{\Sigma_{\noise}^{-1}}}_{L^{2}_{\mu_\param}} \left( \Norm{ \Phi^{\data,\best}}^{1/2}_{L^{1}_{\mu_\param}}+\Norm{\Phi^{\data,\appr}}^{1/2}_{L^{1}_{\mu_\param}}\right).
\end{equation*}
\begin{proof}[Proof of \Cref{lem_L1error_best_misfit_approx_misfit}]
Recall from \eqref{eq_best_misfit_and_posterior} and \eqref{eq_approx_misfit_and_posterior} that $\Phi^{\data,\best}(\param')=\tfrac{1}{2}\Abs{\data-\obsoperator \model^{\best}(\param')}_{\Sigma_\noise^{-1}}^{2}$ and $\Phi^{\data,\appr}(\param')=\tfrac{1}{2}\Abs{\data-\obsoperator \model(\param')}_{\Sigma_\noise^{-1}}^{2}$ respectively.
By these definitions,
\begin{equation}
\label{eq_square_root_L1_norm_best_misfit}
\Norm{ 2\cdot\tfrac{1}{2} \Abs{y-\obsoperator  \model^{\best} }_{\Sigma_{\noise}^{-1}}^{2}}_{L^1_{\mu_\param}}^{1/2}= \Norm{ 2\Phi^{\data,\best}}_{L^1_{\mu_\param}}^{1/2}=\sqrt{2} \Norm{\Phi^{\data,\best}}_{L^1_{\mu_\param}}^{1/2}
\end{equation}
and similarly 
\begin{equation}
\label{eq_square_root_L1_norm_approx_misfit}
 \Norm{\Abs{\data-\obsoperator  \model}^{2}_{\Sigma_{\noise}^{-1}}}^{1/2}_{L^{1}_{\mu_\param}}=\sqrt{2} \Norm{\Phi^{\data,\appr}}^{1/2}_{L^{1}_{\mu_\param}}.
\end{equation}
Now set $f\leftarrow \data-\obsoperator \model^{\best}$, $g\leftarrow \obsoperator  \modelerror^{\best}$, $\mu\leftarrow \mu_\param$, $M_1\leftarrow \Sigma_\noise$, and $M_2\leftarrow 0$.
By \eqref{eq_model_error_definition}, we have $f+g=\data-\obsoperator  (\model^{\best}-\modelerror^{\best})=\data-\obsoperator   \model$. 
Hence $\Abs{f}_{M_1^{-1}}^{2}=2\Phi^{\data,\best}$ and $\Abs{f+g}_{M_1^{-1}}^{2}=2\Phi^{\data,\appr}$.
Applying \eqref{eq_lem_L1norm_diff_quadratic_forms_same_weighting_matrix} from \Cref{lem_L1norm_diff_quadratic_forms} with these choices, and then \eqref{eq_square_root_L1_norm_best_misfit} and \eqref{eq_square_root_L1_norm_approx_misfit}, yields 
\begin{align*}
 2\Norm{\Phi^{\data,\best}-\Phi^{\data,\appr}}_{L^1_{\mu_\param}}\leq &  \Norm{ \Abs{\obsoperator \modelerror^{\best}}_{\Sigma_{\noise}^{-1}}^{2}}_{L^{1}_{\mu_\param}}^{1/2} \left( \Norm{ \Abs{y-\obsoperator  \model^{\best} }_{\Sigma_{\noise}^{-1}}^{2}}_{L^1_{\mu_\param}}^{1/2}+\Norm{\Abs{y-\obsoperator  \model }_{\Sigma_{\noise}^{-1}}^{2}}_{L^1_{\mu_\param}}^{1/2}\right)
 \\
 =& \Norm{ \Abs{\obsoperator \modelerror^{\best}}_{\Sigma_{\noise}^{-1}}^{2}}_{L^{1}_{\mu_\param}}^{1/2} \sqrt{2}\left( \Norm{ \Phi^{\data,\best}}^{1/2}_{L^{1}_{\mu_\param}}+\Norm{\Phi^{\data,\appr}}^{1/2}_{L^{1}_{\mu_\param}}\right).
\end{align*}
This proves \eqref{eq_L1error_best_misfit_approx_misfit}.
The bound on $\Norm{ \Abs{\obsoperator \modelerror^{\best}}_{\Sigma_{\noise}^{-1}}^{2}}_{L^{1}_{\mu_\param}}^{1/2} $ in the statement of \Cref{lem_L1error_best_misfit_approx_misfit} follows from \eqref{eq_bound_L1norm_g} in \Cref{lem_L1norm_diff_quadratic_forms} with the choices above:
\begin{equation*}
 \Norm{ \Abs{\obsoperator \modelerror^{\best}}_{\Sigma_{\noise}^{-1}}^{2}}_{L^{1}_{\mu_\param}}^{1/2} \leq \sqrt{2}\left( \Norm{ \Phi^{\data,\best}}^{1/2}_{L^{1}_{\mu_\param}}+\Norm{\Phi^{\data,\appr}}^{1/2}_{L^{1}_{\mu_\param}}\right).
\end{equation*}
This completes the proof of \Cref{lem_L1error_best_misfit_approx_misfit}.
\end{proof}

\subsection{Proofs for enhanced noise posterior}
\label{ssec_proofs_KL_bound_enhanced_noise_posterior}

\Cref{lem_L1error_best_misfit_enhanced_noise_misfit} bounds $\Norm{ \Phi^{\data,\best} - \Phi^{\data,\enh} }_{L^1_{\mu_{\param}}}$ in terms of the observed model error $\obsoperator \modelerror^{\best}$ and the Gaussian model $\normaldist{m_{\enhancednoisemodel}}{\Sigma_{\enhancednoisemodel}}$ of $\modelerror^{\best}(\param^{\best})$.
In particular, if $\Phi^{\data,\best}\in L^1_{\mu_\param}$ and $\Phi^{\data,\enh}\in L^1_{\mu_\param}$, then for $C_{\enh}\coloneqq \Norm{\Sigma_\noise^{-1/2}( \enhancednoisecovariance)^{1/2}}$ as in \eqref{eq_equivalence_of_norms_enhanced_noise}, the bound \eqref{eq_L1error_best_misfit_enhanced_noise_misfit} 
\begin{align*}
\Norm{ \Phi^{\data,\best} - \Phi^{\data,\enh} }_{L^{1}_{\mu_{\param}}}\leq &  2^{-1/2} \Norm{ \Abs{ \obsoperator (\modelerror^{\best}-m_{\enhancednoisemodel})}_{\Sigma_{\noise}^{-1}}^{2}}_{L^1_{\mu_\param}}^{1/2} \left( \Norm{\Phi^{\data,\best}}^{1/2}_{L^1_{\mu_\param}}+ C_{\enh}\Norm{\Phi^{\data,\enh}}_{L^1_{\mu_\param}}^{1/2}\right)
\\
&+2^{-1}\Norm{ \Abs{\data   -\obsoperator \model-\obsoperator m_{\enhancednoisemodel}}^{2}_{ \Sigma_\noise^{-1}-(\enhancednoisecovariance)^{-1}}}_{L^1_{\mu_\param}},
\end{align*}
holds, and all terms on the right-hand side are finite.

\begin{proof}[Proof of \Cref{lem_L1error_best_misfit_enhanced_noise_misfit}]
 In the same way that we proved \eqref{eq_square_root_L1_norm_best_misfit}, we can use the definition \eqref {eq_enhanced_noise_misfit_and_posterior} of $\Phi^{\data,\enh}$ to prove
\begin{equation}
\Norm{\Abs{\data-\obsoperator \model -\obsoperator m_{\enhancednoisemodel}}_{(\enhancednoisecovariance)^{-1}}^{2}}_{L^1_{\mu_\param}}^{1/2}
 =\sqrt{2}\Norm{\Phi^{\data,\enh}}_{L^1_{\mu_\param}}^{1/2}.
 \label{eq_square_root_L1_norm_enhanced_noise_misfit}
 \end{equation}
Let $f\leftarrow \data-\obsoperator  \model^{\best}$, $g\leftarrow \obsoperator   (\modelerror^{\dagger}-m_{\enhancednoisemodel})$, $\mu\leftarrow \mu_\param$, $M_1\leftarrow \Sigma_\noise$, and $M_2\leftarrow \obsoperator \Sigma_{\enhancednoisemodel}\obsoperator^{\ast}$.
Then 
\begin{equation*}
 f+g= \data -\obsoperator   (\model^{\best}-\modelerror^{\dagger})-\obsoperator m_{\enhancednoisemodel}=\data-\obsoperator   \model-\obsoperator m_{\enhancednoisemodel},
\end{equation*}
 $\Abs{f}_{M_1^{-1}}^{2}=2\Phi^{\data,\dagger}$, and $\Abs{f+g}_{(M_1+M_2)^{-1}}^{2}=2\Phi^{\data,\enh}$. 
Applying \eqref{eq_lem_L1norm_diff_quadratic_forms_different_weighting_matrix}  from \Cref{lem_L1norm_diff_quadratic_forms} yields 
\begin{align}
 2 \Norm{ \Phi^{\data,\best} - \Phi^{\data,\enh}}_{L^1_{\mu_\param}}\leq & \Norm{ \Abs{ \obsoperator (\modelerror^{\best} -m_{\enhancednoisemodel})}_{\Sigma_{\noise}^{-1}}^{2} }_{L^1_{\mu_\param}}^{1/2}\left( C_{\enh} \Norm{2\Phi^{\data,\enh}}_{L^1_{\mu_\param}}^{1/2}+\Norm{2\Phi^{\data,\best}}_{L^1_{\mu_\param}}^{1/2} \right)
\label{eq_intermediate02}
 \\
 & +\Norm{ \Abs{\data   -\obsoperator \model -\obsoperator m_{\enhancednoisemodel}}^{2}_{ \Sigma_\noise^{-1}-(\enhancednoisecovariance)^{-1}}}_{L^1_{\mu_\param}}.
\nonumber
 \end{align}
 This proves \eqref{eq_L1error_best_misfit_enhanced_noise_misfit}.
 By \eqref{eq_bound_L1norm_g} and $C_{\enh}$ as in \eqref{eq_equivalence_of_norms_enhanced_noise},
 \begin{equation*}
  \Norm{ \Abs{ \obsoperator (\modelerror^{\best} -m_{\enhancednoisemodel})}_{\Sigma_{\noise}^{-1}}^{2} }_{L^1_{\mu_\param}}^{1/2} \leq \Norm{2\Phi^{\data,\best}}_{L^1_{\mu_\param}}^{1/2}+C_{\enh}\Norm{2\Phi^{\data,\enh}}_{L^1_{\mu_\param}}^{1/2}.
 \end{equation*}
 By \eqref{eq_bound_L1norm_fplusg_difference_of_weighting_matrices} from \Cref{lem_L1norm_diff_quadratic_forms},
 \begin{equation*}
 \Norm{ \Abs{\data   -\obsoperator \model -\obsoperator m_{\enhancednoisemodel}}^{2}_{ \Sigma_\noise^{-1}-(\enhancednoisecovariance)^{-1}}}_{L^1_{\mu_\param}}\leq (C_{\enh}+1)\Norm{2\Phi^{\data,\enh}}_{L^1_{\mu_\param}}.
 \end{equation*}
 This completes the proof of \Cref{lem_L1error_best_misfit_enhanced_noise_misfit}.
\end{proof}

\Cref{lem_L1norm_diff_Phi_approx_minus_Phi_enhanced} asserts that, under the hypotheses that $\Phi^{\data,\appr}\in L^1_{\mu_\param}$ and $\Phi^{\data,\enh}\in L^1_{\mu_\param}$, then for $C_{\enh}\coloneqq \Norm{\Sigma_\noise^{-1/2}(\enhancednoisecovariance)^{1/2}}_{\op}$ as in \eqref{eq_equivalence_of_norms_enhanced_noise}, the bound \eqref{eq_L1norm_diff_Phi_approx_minus_Phi_enhanced} 
\begin{align*}
\Norm{ \Phi^{\data,\appr} - \Phi^{\data,\enh} }_{L^{1}_{\mu_{\param}}}\leq &  2^{-1/2} \Abs{ \obsoperator m_{\enhancednoisemodel}}_{\Sigma_{\noise}^{-1}} \left( \Norm{\Phi^{\data,\appr}}^{1/2}_{L^1_{\mu_\param}}+ C_{\enh}\Norm{\Phi^{\data,\enh}}_{L^1_{\mu_\param}}^{1/2}\right)
\\
&+2^{-1}\Norm{ \Abs{\data   -\obsoperator \model-\obsoperator m_{\enhancednoisemodel}}^{2}_{ \Sigma_\noise^{-1}-(\enhancednoisecovariance)^{-1}}}_{L^1_{\mu_\param}},
\end{align*}
holds, and all terms on the right-hand side are finite.
\begin{proof}[Proof of \Cref{lem_L1norm_diff_Phi_approx_minus_Phi_enhanced}]
Let $f\leftarrow  \data - \obsoperator  \model$, $g\leftarrow -\obsoperator m_{\enhancednoisemodel}$, $M_1\leftarrow \Sigma_\noise^{-1}$, $M_2\leftarrow \obsoperator \Sigma_{\enhancednoisemodel} \obsoperator^{\ast}$, and $\mu\leftarrow \mu_\param$.
Then $\Abs{f}_{M_1^{-1}}^{2}=2\Phi^{\data,\appr}$ and $\Abs{f+g}_{(M_1+M_2)^{-1}}^{2}=2\Phi^{\data,\enh}$.
Applying \eqref{eq_lem_L1norm_diff_quadratic_forms_different_weighting_matrix} from \Cref{lem_L1norm_diff_quadratic_forms} yields 
\begin{align*}
 &2 \Norm{ \Phi^{\data,\appr} - \Phi^{\data,\enh}}_{L^1_{\mu_\param}}
\\
\leq & \Abs{ \obsoperator m_{\enhancednoisemodel}}_{\Sigma_{\noise}^{-1}}\sqrt{2}\left( \Norm{\Phi^{\data,\appr}}^{1/2}_{L^1_{\mu_\param}}+C_{\enh} \Norm{ \Phi^{\data,\enh}}^{1/2}_{L^1_{\mu_\param}}\right) +\Norm{ \Abs{\data   -\obsoperator \model -\obsoperator m_{\enhancednoisemodel}}^{2}_{ \Sigma_\noise^{-1}-(\enhancednoisecovariance)^{-1}}}_{L^1_{\mu_\param}},
 \end{align*}
  for $C_{\enh}\coloneqq \Norm{\Sigma_\noise^{-1/2}(\enhancednoisecovariance)^{1/2}}_{\op}$ as in \eqref{eq_equivalence_of_norms_enhanced_noise}.
 This proves \eqref{eq_L1norm_diff_Phi_approx_minus_Phi_enhanced}.
 Next, \eqref{eq_bound_L1norm_fplusg_difference_of_weighting_matrices} yields
 \begin{equation*}
  \Norm{ \Abs{\data   -\obsoperator \model -\obsoperator m_{\enhancednoisemodel}}^{2}_{ \Sigma_\noise^{-1}-(\enhancednoisecovariance)^{-1}}}_{L^1_{\mu_\param}}\leq 2 \left(C_{\enh}+1\right)\Norm{\Phi^{\data,\enh}}_{L^1_{\mu_\param}}.
 \end{equation*}
This completes the proof of \Cref{lem_L1norm_diff_Phi_approx_minus_Phi_enhanced}.
\end{proof}

\subsection{Proofs for joint posterior}
\label{ssec_proofs_KL_bound_joint_posterior}

In \Cref{lem_L1norm_diff_Phi_best_minus_Phi_joint}, one assumes that $\Phi^{\data,\best}$ as defined in \eqref{eq_best_misfit_and_posterior} belongs to $ L^1_{\mu_\param }$, and also that $\Phi^{\data,\joint}\in L^1_{\mu_\param\otimes\mu_\modelerror}$.
The resulting bound \eqref{eq_L1norm_diff_Phi_best_minus_Phi_joint} is
 \begin{equation*}
  \Norm{\Phi^{\data,\best}-\Phi^{\data,\joint}}_{L^1_{\mu_\param\otimes\mu_\modelerror}}\leq 2^{-1/2} \Norm{ \Abs{\obsoperator (\modelerror^{\best}-\modelerror)(\param)}_{\Sigma_\noise^{-1}}^{2}}_{L^1_{\bb{P}}}^{1/2} \left( \Norm{\Phi^{\data,\joint}}_{L^1_{\mu_\param\otimes\mu_\modelerror}}^{1/2}+ \Norm{\Phi^{\data,\best}}_{L^1_{\mu_\param}}^{1/2}\right).
 \end{equation*}
\begin{proof}[Proof of \Cref{lem_L1norm_diff_Phi_best_minus_Phi_joint}]
 Let $f(\param)\leftarrow \data-\obsoperator \model^{\best}(\param)$, $g(\param,\modelerror)\leftarrow \obsoperator  (\modelerror^{\best}-\modelerror)(\param)$, $M_1\leftarrow \Sigma_\noise$, $M_2\leftarrow 0$, and $\mu\leftarrow\bb{P}$.
 Then $\Abs{f(\param)}_{M_1^{-1}}^{2}=2\Phi^{\data,\best}(\param,\modelerror)$ and $\Abs{f(\param)+g(\param,\modelerror)}_{M_1^{-1}}^{2}=2\Phi^{\data,\joint}(\param,\modelerror)$.
Using the same argument that we used to prove \eqref{eq_square_root_L1_norm_best_misfit}, it follows from \eqref {eq_joint_misfit_and_posterior} that
\begin{equation}
 \label{eq_square_root_L1_norm_joint_inference_misfit}
 \Norm{ \Abs{ \data-\obsoperator  \model(\param)-\obsoperator   \modelerror(\param)}_{\Sigma_\noise^{-1}}^{2}}_{L^1_{\bb{P}}}^{1/2}=\sqrt{2}\Norm{\Phi^{\data,\joint}}_{L^1_{\mu_\param\otimes\mu_\modelerror}}^{1/2}.
\end{equation}
By \eqref{eq_square_root_L1_norm_best_misfit} and \eqref{eq_lifted_misfits_same_norms} from \Cref{lem_properties_of_lifted_placeholder_objects},
\begin{equation}
\label{eq_square_root_L1_norm_best_misfit_lifted}
 \Norm{\Abs{\data-\obsoperator \model^{\best}(\param)}_{\Sigma_\noise^{-1}}^{2}}_{L^1_{\bb{P}}}^{1/2}=\sqrt{2}\Norm{\Phi^{\data,\best}}_{L^1_{\mu_\param}}^{1/2}=\sqrt{2}\Norm{\Phi^{\data,\best}}_{L^1_{\mu_\param\otimes\mu_\modelerror}}^{1/2}.
\end{equation}
Applying \eqref{eq_lem_L1norm_diff_quadratic_forms_same_weighting_matrix} \Cref{lem_L1norm_diff_quadratic_forms} with these choices yields \eqref{eq_L1norm_diff_Phi_best_minus_Phi_joint}:
\begin{align*}
 2\Norm{\Phi^{\data,\appr}-\Phi^{\data,\joint}}_{L^1_{\mu_\param\otimes\mu_\modelerror}} \leq & \Norm{ \Abs{\obsoperator (\modelerror^{\best}-\modelerror)(\param)}_{\Sigma_\noise^{-1}}^{2}}_{L^1_{\bb{P}}}^{1/2} \sqrt{2}\left( \Norm{\Phi^{\data,\joint}}_{L^1_{\mu_\param\otimes\mu_\modelerror}}^{1/2}+ \Norm{\Phi^{\data,\best}}_{L^1_{\mu_\param}}^{1/2}\right).
\end{align*}
Next,
\begin{equation*}
  \Norm{ \Abs{\obsoperator (\modelerror^{\best}-\modelerror)(\param)}_{\Sigma_\noise^{-1}}^{2}}_{L^1_{\bb{P}}}^{1/2} \leq\sqrt{2}\left( \Norm{\Phi^{\data,\joint}}_{L^1_{\mu_\param\otimes\mu_\modelerror}}^{1/2}+ \Norm{\Phi^{\data,\best}}_{L^1_{\mu_\param}}^{1/2}\right),
 \end{equation*}
 follows from \eqref{eq_square_root_L1_norm_joint_inference_misfit}, \eqref{eq_square_root_L1_norm_best_misfit_lifted}, and \eqref{eq_bound_L1norm_g} of \Cref{lem_L1norm_diff_quadratic_forms} with the choices stated above.
\end{proof}

In \Cref{lem_L1norm_diff_Phi_approx_minus_Phi_joint}, one assumes that $\Phi^{\data,\appr}$ as defined in \eqref{eq_approx_misfit_and_posterior} belongs to $ L^1_{\mu_\param }$, and also that $\Phi^{\data,\joint}\in L^1_{\mu_\param\otimes\mu_\modelerror}$.
The resulting bound \eqref{eq_L1norm_diff_Phi_approx_minus_Phi_joint} is
 \begin{equation*}
  \Norm{\Phi^{\data,\appr}-\Phi^{\data,\joint}}_{L^1_{\mu_\param\otimes\mu_\modelerror}}\leq 2^{-1/2}\Norm{ \Abs{\obsoperator \modelerror(\param)}_{\Sigma_\noise^{-1}}^{2}}_{L^1_{\bb{P}}}^{1/2} \left( \Norm{\Phi^{\data,\joint}}_{L^1_{\mu_\param\otimes\mu_\modelerror}}^{1/2}+ \Norm{\Phi^{\data,\appr}}_{L^1_{\mu_\param}}^{1/2}\right).
 \end{equation*}
 The proof of \Cref{lem_L1norm_diff_Phi_approx_minus_Phi_joint} is very similar to the proof of \Cref{lem_L1norm_diff_Phi_best_minus_Phi_joint} above. 
\begin{proof}[Proof of \Cref{lem_L1norm_diff_Phi_approx_minus_Phi_joint}]
 Let $f(\param)\leftarrow \data-\obsoperator \model(\param)$, $g(\param,\modelerror)\leftarrow \obsoperator  (-\modelerror)(\param)$, $M_1\leftarrow \Sigma_\noise$, $M_2\leftarrow 0$, and $\mu \leftarrow \bb{P}$.
 Then $\Abs{f(\param)+g(\param,\modelerror)}_{M_1^{-1}}^{2}=2\Phi^{\data,\joint}(\param,\modelerror)$ and $\Abs{f(\param)}_{M_1^{-1}}^{2}=2\Phi^{\data,\appr}(\param,\modelerror)$.
 Analogously to \eqref{eq_square_root_L1_norm_best_misfit_lifted}, we have by \eqref{eq_approx_misfit_and_posterior} and \eqref{eq_lifted_misfits_same_norms} of \Cref{lem_properties_of_lifted_placeholder_objects} that
\begin{equation}
\label{eq_square_root_L1_norm_approx_misfit_lifted}
  \Norm{\Abs{\data-\obsoperator \model(\param)}_{\Sigma_\noise^{-1}}^{2}}_{L^1_{\bb{P}}}^{1/2}=\sqrt{2}\Norm{\Phi^{\data,\appr}}_{L^1_{\mu_\param}}^{1/2}=\sqrt{2}\Norm{\Phi^{\data,\appr}}_{L^1_{\mu_\param\otimes\mu_\modelerror}}^{1/2}.
\end{equation}
Applying \eqref{eq_lem_L1norm_diff_quadratic_forms_same_weighting_matrix} of \Cref{lem_L1norm_diff_quadratic_forms} with the choices above yields \eqref{eq_L1norm_diff_Phi_approx_minus_Phi_joint}:
\begin{align*}
 2\Norm{\Phi^{\data,\appr}-\Phi^{\data,\joint}}_{L^1_{\mu_\param\otimes\mu_\modelerror}} \leq & \Norm{ \Abs{\obsoperator (-\modelerror)(\param)}_{\Sigma_\noise^{-1}}^{2}}_{L^1_{\bb{P}}}^{1/2} \sqrt{2}\left( \Norm{\Phi^{\data,\joint}}_{L^1_{\mu_\param\otimes\mu_\modelerror}}^{1/2}+ \Norm{\Phi^{\data,\appr}}_{L^1_{\mu_\param}}^{1/2}\right).
\end{align*}
Next,
 \begin{equation*}
  \Norm{ \Abs{\obsoperator (-\modelerror)(\param)}_{\Sigma_\noise^{-1}}^{2}}_{L^1_{\bb{P}}}^{1/2}
  \leq \sqrt{2}\left(\Norm{\Phi^{\data,\appr}}_{L^1_{\mu_\param}}^{1/2} +\Norm{\Phi^{\data,\joint}}_{L^1_{\mu_\param\otimes\mu_\modelerror}}^{1/2} \right)
 \end{equation*}
follows from \eqref{eq_square_root_L1_norm_joint_inference_misfit}, \eqref{eq_square_root_L1_norm_approx_misfit_lifted}, and \eqref{eq_bound_L1norm_g} of \Cref{lem_L1norm_diff_quadratic_forms}.
\end{proof}

\begin{lemma}
 \label{lem_L1norm_diff_Phi_enhanced_minus_Phi_joint}
 Let $\Phi^{\data,\enh}$ be defined as in \eqref{eq_lifted_misfit_and_posterior} with $\bullet=\enh$.
 If $\Phi^{\data,\enh}\in L^1_{\mu_\param}$ and $\Phi^{\data,\joint}\in L^1_{\mu_\param\otimes\mu_\modelerror}$, then
 \begin{align}
 \label{eq_L1norm_diff_Phi_enhanced_minus_Phi_joint}
  \Norm{\Phi^{\data,\enh}-\Phi^{\data,\joint}}_{L^1_{\mu_\param\otimes\mu_\modelerror}}\leq& 2^{-1/2}\Norm{ \Abs{\obsoperator (\modelerror(\param)-m_{\enhancednoisemodel})}_{\Sigma_\noise^{-1}}^{2}}_{L^1_{\bb{P}}}^{1/2}\left( C_{\enh} \Norm{\Phi^{\data,\enh}}_{L^1_{\mu_\param}}^{1/2}+\Norm{\Phi^{\data,\joint}}_{L^1_{\mu_\param\otimes\mu_\modelerror}}^{1/2}\right)
 \\
 &+ 2^{-1}\Norm{ \Abs{ \data-\obsoperator  \model(\param)-\obsoperator m_{\enhancednoisemodel}}_{\Sigma_\noise^{-1}-(\enhancednoisecovariance)^{-1}}^{2}}_{L^1_{\bb{P}}}.
 \nonumber
 \end{align}
 Furthermore,
 \begin{align*}
 &\Norm{ \Abs{\obsoperator (\modelerror(\param)-m_{\enhancednoisemodel})}_{\Sigma_\noise^{-1}}^{2}}_{L^1_{\bb{P}}}^{1/2}\leq \sqrt{2}\left(C_{\enh}\Norm{\Phi^{\data,\enh}}_{L^1_{\mu_\param}}^{1/2}+ \Norm{\Phi^{\data,\joint}}_{L^1_{\mu_\param\otimes \mu_\modelerror}}^{1/2}\right) 
 \\
 &\Norm{ \Abs{ \data-\obsoperator  \model(\param)-\obsoperator m_{\enhancednoisemodel}}_{\Sigma_\noise^{-1}-(\enhancednoisecovariance)^{-1}}^{2}}_{L^1_{\bb{P}}}\leq 2(C_{\enh}+1)\Norm{\Phi^{\data,\enh}}_{L^1_{\mu_\param}}.
 \end{align*}
\end{lemma}
\begin{proof}[Proof of \Cref{lem_L1norm_diff_Phi_enhanced_minus_Phi_joint}]
 Let $f(\param,\modelerror)\leftarrow \data-\obsoperator (\model(\param)+\modelerror(\param))$, $g(\param,\modelerror)\leftarrow \obsoperator  (\modelerror(\param)-m_{\enhancednoisemodel})$, $M_1\leftarrow \Sigma_\noise$, $M_2\leftarrow \obsoperator \Sigma_{\enhancednoisemodel}\obsoperator^{\ast}$, and $\mu\leftarrow \bb{P}$.
Then $\Abs{f(\param,\modelerror)}_{M_1^{-1}}^{2}=2\Phi^{\data,\joint}(\param,\modelerror)$ and $\Abs{(f+g)(\param,\modelerror)}_{(M_1+M_2)^{-1}}^{2}=2\Phi^{\data,\enh}(\param,\modelerror)$.
Analogously to \eqref{eq_square_root_L1_norm_best_misfit_lifted}, we have by \eqref {eq_enhanced_noise_misfit_and_posterior} and \eqref{eq_lifted_misfits_same_norms} of \Cref{lem_properties_of_lifted_placeholder_objects} that
\begin{equation}
\label{eq_square_root_L1_norm_enhanced_noise_misfit_lifted}
  \Norm{\Abs{\data-\obsoperator \model(\param)-\obsoperator  m_{\enhancednoisemodel}}_{(\enhancednoisecovariance)^{-1}}^{2}}_{L^1_{\bb{P}}}^{1/2}=\sqrt{2}\Norm{\Phi^{\data,\enh}}_{L^1_{\mu_\param}}^{1/2}=\sqrt{2}\Norm{\Phi^{\data,\enh}}_{L^1_{\mu_\param\otimes\mu_\modelerror}}^{1/2}.
\end{equation}
Applying \eqref{eq_lem_L1norm_diff_quadratic_forms_different_weighting_matrix} of \Cref{lem_L1norm_diff_quadratic_forms} with the choices above yields \eqref{eq_L1norm_diff_Phi_enhanced_minus_Phi_joint}:
\begin{align*}
 2\Norm{\Phi^{\data,\enh}-\Phi^{\data,\joint}}_{L^1_{\mu_\param\otimes\mu_\modelerror}} \leq & \Norm{ \Abs{\obsoperator (\modelerror(\param)-m_{\enhancednoisemodel})}_{\Sigma_\noise^{-1}}^{2}}_{L^1_{\bb{P}}}^{1/2} \left( C_{\enh} \Norm{2\Phi^{\data,\enh}}_{L^1_{\mu_\param}}^{1/2}+\Norm{2\Phi^{\data,\joint}}_{L^1_{\mu_\param\otimes\mu_\modelerror}}^{1/2}\right)
 \\
 &+ \Norm{ \Abs{ \data-\obsoperator  \model(\param)-\obsoperator m_{\enhancednoisemodel}}_{\Sigma_\noise^{-1}-(\enhancednoisecovariance)^{-1}}^{2}}_{L^1_{\bb{P}}}.
\end{align*}
Next, we apply \eqref{eq_bound_L1norm_g} and \eqref{eq_bound_L1norm_fplusg_difference_of_weighting_matrices} from \Cref{lem_L1norm_diff_quadratic_forms}, and use \eqref{eq_square_root_L1_norm_enhanced_noise_misfit_lifted}:
\begin{align*}
 &\Norm{ \Abs{\obsoperator (\modelerror(\param)-m_{\enhancednoisemodel})}_{\Sigma_\noise^{-1}}^{2}}_{L^1_{\bb{P}}}^{1/2} \leq \sqrt{2}\left(C_{\enh}\Norm{\Phi^{\data,\enh}}_{L^1_{\mu_\param}}^{1/2}+ \Norm{\Phi^{\data,\joint}}_{L^1_{\mu_\param\otimes \mu_\modelerror}}^{1/2}\right) 
 \\
 &\Norm{ \Abs{ \data-\obsoperator  \model(\param)-\obsoperator m_{\enhancednoisemodel}}_{\Sigma_\noise^{-1}-(\enhancednoisecovariance)^{-1}}^{2}}_{L^1_{\bb{P}}}\leq 2(C_{\enh}+1)\Norm{\Phi^{\data,\enh}}_{L^1_{\mu_\param}}.
\end{align*}
This completes the proof of \Cref{lem_L1norm_diff_Phi_enhanced_minus_Phi_joint}.
\end{proof}

\begin{proposition}
 \label{prop_KL_divergence_enhanced_noise_posterior_joint_posterior}
 Let $\Phi^{\data,\enh}$ and $\Phi^{\data,\joint}$ be as in \Cref{lem_L1norm_diff_Phi_enhanced_minus_Phi_joint}, and let $\mu^{\data,\enh}_{\param,\modelerror}$ be as in \eqref{eq_lifted_misfit_and_posterior} with $\bullet=\enh$.
 Then 
 \begin{equation*}
  \max\{ d_{\KL}(\mu^{\data,\enh}_{\param,\modelerror}\Vert\mu^{\data,\joint}_{\param,\modelerror}), d_{\KL}(\mu^{\data,\joint}_{\param,\modelerror}\Vert \mu^{\data,\enh}_{\param,\modelerror})\}
  \leq C  \Norm{ \Abs{\obsoperator (m_{\enhancednoisemodel}-\modelerror)(\param)}_{\Sigma_\noise^{-1}}^{2}}_{L^1_{\bb{P}}}^{1/2},
 \end{equation*}
 where
 \begin{equation*}
  C= \exp\left(2 \Norm{\Phi^{\data,\joint}}_{L^1_{\mu_\param\otimes\mu_\modelerror}}+2\Norm{\Phi^{\data,\enh}}_{L^1_{\mu_\param}}\right) \max\left\{ 1, 2^{1/2} \left( C_{\enh} \Norm{\Phi^{\data,\enh}}_{L^1_{\mu_\param\otimes\mu_\modelerror}}^{1/2}+\Norm{\Phi^{\data,\joint}}_{L^1_{\mu_\param\otimes\mu_\modelerror}}^{1/2}\right)\right\}.
 \end{equation*}
\end{proposition}
\begin{proof}[Proof of \Cref{prop_KL_divergence_enhanced_noise_posterior_joint_posterior}]
 Applying \Cref{thm_theorem11_and_proposition6_Sprungk2020} and \Cref{lem_L1norm_diff_Phi_enhanced_minus_Phi_joint} yields
 \begin{align*}
  &\max\{ d_{\KL}(\mu^{\data,\enh}_{\param,\modelerror}\Vert\mu^{\data,\joint}_{\param,\modelerror}), d_{\KL}(\mu^{\data,\joint}_{\param,\modelerror}\Vert \mu^{\data,\enh}_{\param,\modelerror})\}
  \\
  \leq & 2\exp\left(2 \Norm{\Phi^{\data,\joint}}_{L^1_{\mu_\param\otimes\mu_\modelerror}}+2\Norm{\Phi^{\data,\enh}}_{L^1_{\mu_\param\otimes\mu_\modelerror}}\right) \Norm{\Phi^{\data,\enh}-\Phi^{\data,\joint}}_{L^1_{\mu_\param\otimes\mu_\modelerror}}
  \\
  \leq & \exp\left(2 \Norm{\Phi^{\data,\joint}}_{L^1_{\mu_\param\otimes\mu_\modelerror}}+2\Norm{\Phi^{\data,\enh}}_{L^1_{\mu_\param\otimes\mu_\modelerror}}\right) \max\{ 2^{1/2} \left( C_{\enh} \Norm{\Phi^{\data,\enh}}_{L^1_{\mu_\param\otimes\mu_\modelerror}}^{1/2}+\Norm{\Phi^{\data,\joint}}_{L^1_{\mu_\param\otimes\mu_\modelerror}}^{1/2}\right),1\}
  \\
  &\times\left( \Norm{ \Abs{\obsoperator (m_{\enhancednoisemodel}-\modelerror)(\param)}_{\Sigma_\noise^{-1}}^{2}}_{L^1_{\bb{P}}}+\Norm{ \Abs{ \data-\obsoperator  \model(\param)-\obsoperator m_{\enhancednoisemodel}}_{\Sigma_\noise^{-1}-(\enhancednoisecovariance)^{-1}}^{2}}_{L^1_{\bb{P}}}\right).
  \end{align*}
  Using the definition of $C$ given in the statement of \Cref{prop_KL_divergence_enhanced_noise_posterior_joint_posterior} and \eqref{eq_lifted_misfits_same_norms} from \Cref{lem_properties_of_lifted_placeholder_objects} completes the proof.
\end{proof}

\bibliographystyle{amsplain}
\bibliography{references}

\end{document}